\documentclass[twoside,11pt,reqno]{amsart}
\usepackage{amsmath,amssymb,amscd,mathrsfs,stmaryrd,wasysym,epic}
\usepackage{latexsym,amsthm,bbm}
\makeatletter
\@addtoreset{equation}{section}
\@addtoreset{figure}{section}

\newtheorem{Proposition}{Proposition}[section]
\newtheorem{Lemma}[Proposition]{Lemma}
\newtheorem{Theorem}[Proposition]{Theorem}

\newtheorem{Corollary}[Proposition]{Corollary}

\newtheorem{Remark}[Proposition]{Remark}

\newtheorem{Example}[Proposition]{Example}

\newbox\squ  % box character for ends of proofs
\setbox\squ=\hbox{\vrule width.6pt
   \vbox{\hrule height.6pt width.4em\kern1ex
         \hrule height.6pt}%
   \vrule width.6pt}

\def\b{\mathbbm{b}}
\def\b{\mathtt{b}}

\def\f{{\mathbf{f}}}

\def\hgt{{\operatorname{ht}}}

\def\x{x}

\def\KP{\operatorname{KP}}
\def\MP{\operatorname{MP}}
\def\mp{\operatorname{mp}}

\def\dtau{\tau}
\def\dx{x}

\def\Rep#1{\text{{$\operatorname{Rep}$}(}#1{)}}
\def\Proj#1{\text{{$\operatorname{Proj}$}(}#1{)}}

\def\k{\Bbbk}
\def\k{\mathbb{K}}
\def\NH{N\!\!\:H}
\def\W{\langle I \rangle}

\def\X{M}

\def\hgt{\operatorname{ht}}
\def\deg{\operatorname{deg}}

\def\op{\operatorname{op}}

\def\height{\operatorname{ht}}

\def\id{\operatorname{id}}

\def\Q{{\mathbb Q}}
\def\Z{{\mathbb Z}}
\def\A{{\mathscr A}}
\def\AA{{\mathbb A}}
\def\N{{\mathbb N}}

\def\rad{{\operatorname{rad}\:}}
\def\H{{\mathrm H}}

\def\Ind{{\operatorname{ind}}}
\def\res{{\operatorname{res}}}
\def\Res{{\operatorname{res}}}
\def\im{{\operatorname{im}\:}}
\def\soc{{\operatorname{soc}\:}}

\def\eps{{\varepsilon}}

\def\phi{{\varphi}}
\def\phi{{\varphi}}

\def\al{{\alpha}}

\def\@underbar#1{\settowidth{\@tempdimb}{#1}\@tempdimb=0.8\@tempdimb
                   \ooalign{#1\crcr
                         \hfil\rule[-.5mm]{\@tempdimb}{.4pt}\hfil}}

\def\bi{\text{\boldmath$i$}}

\def\bed{\text{\boldmath$i$}}

\def\bj{\text{\boldmath$j$}}
\def\bk{\text{\boldmath$k$}}

\def\END{{\operatorname{End}}}
\def\hom{{\operatorname{hom}}}
\def\HOM{{\operatorname{Hom}}}
\def\EXT{{\operatorname{Ext}}}
\def\ext{{\operatorname{ext}}}

\def\DIM{\operatorname{Dim}}
\def\CH{\operatorname{Ch}}

\begin{document}

\title[Khovanov-Lauda-Rouquier algebras]{Homological properties of
finite type Khovanov-Lauda-Rouquier algebras}

\author{Jonathan Brundan, Alexander Kleshchev and Peter J. McNamara}

\address{Department of Mathematics, University of Oregon, Eugene, OR 97403, USA}
\email{brundan@uoregon.edu}

\address{Department of Mathematics, University of Oregon, Eugene, OR 97403, USA}
\email{klesh@uoregon.edu}

\address{Department of Mathematics, Stanford University, Stanford, CA
  94305, USA }
\email{petermc@math.stanford.edu}

\thanks{Research of the first two authors
supported in part by NSF grant no. DMS-1161094.
The second author also acknowledges support from the Humboldt Foundation.}

\begin{abstract}
We give an algebraic construction of
{\em standard modules}---infinite dimensional modules categorifying the PBW basis of
the underlying quantized enveloping algebra---for
Khovanov-Lauda-Rouquier algebras in all
finite types. 
This allows us to prove in an elementary way
that these algebras satisfy the homological properties of an ``affine
quasi-hereditary algebra.'' 
In simply-laced types these properties were established originally by
Kato via a geometric approach. We also construct some Koszul-like projective resolutions of
standard modules corresponding to multiplicity-free positive roots.
\end{abstract}

\maketitle

\section{Introduction}

Working over $\Q(q)$ for an indeterminate $q$,
let $\f$ be the quantized enveloping algebra associated to
a maximal nilpotent subalgebra of a
finite dimensional complex semisimple 
Lie algebra $\mathfrak{g}$.
It is naturally $Q^+$-graded
\begin{equation*}
\f = \bigoplus_{\alpha
  \in Q^+} \f_\alpha,
\end{equation*}
where $Q^+$ 
denotes $\N$-linear combinations of 
the simple roots $\{\alpha_i\:|\:i \in I\}$.
Moreover $\f$ is equipped with several distinguished bases, including 
Lusztig's {\em canonical basis} (Kashiwara's lower global crystal base)
and various {\em PBW bases}, one for each choice $\prec$ of convex
ordering of the set $R^+$ of positive roots.
Passing to dual bases
with respect to Lusztig's
form $(\cdot,\cdot)$ on $\f$, 
we obtain the {\em dual canonical basis} (Kashiwara's upper global crystal base)
and some {\em dual PBW bases}.
See \cite{Lu}, \cite{Lubook} and \cite{Ka}.
Lusztig's approach gives a {categorification} of $\f$ in terms of certain categories of sheaves on a quiver variety. Multiplication on $\f$ comes from
Lusztig's induction functor, and in simply-laced types 
the canonical basis arises from the irreducible perverse sheaves in
these categories.

In 2008
Khovanov and Lauda \cite{KL,KL2} and Rouquier \cite{R}
introduced for any field $\k$ a (locally unital) graded $\k$-algebra 
\begin{equation*}
H = \bigoplus_{\alpha \in Q^+} H_\alpha,
\end{equation*}
known as
the 
{\em Khovanov-Lauda-Rouquier
  algebra} (KLR for short).
Let $\Proj{H}$ be the additive category of finitely
generated graded 
projective left $H$-modules.
We make the split Grothendieck group
$[\Proj{H}]$ of this category
into a $\Z[q,q^{-1}]$-algebra,
with
multiplication arising from the
induction product $\circ$ on modules over the KLR algebra
and
action of $q$
induced by 
upwards degree shift.
Khovanov and Lauda showed that $\Proj{H}$ also provides a
categorification of $\f$: there is a
unique algebra isomorphism
\begin{equation*}
\gamma: \f \stackrel{\sim}{\rightarrow}
\Q(q)
\otimes_{\Z[q,q^{-1}]} [\Proj{H}],
\qquad
\theta_i \mapsto [H_{\alpha_i}],
\end{equation*}
where $\theta_i$ is the generator of $\f$ corresponding to simple
root $\alpha_i$.
In simply-laced types for $\k$ of characteristic zero,
Rouquier \cite{R2}  and Varagnolo and Vasserot \cite{VV} have
shown further that this algebraic categorification of $\f$ 
is equivalent to Lusztig's geometric one;
in particular $\gamma$ maps 
the canonical basis of $\f$ to
the basis for $[\Proj{H}]$ arising from the 
isomorphism classes of graded self-dual 
indecomposable projective modules.
See also the work of Maksimau \cite{Mak} where the geometric
realization of KLR algebras
has been extended to fields positive characteristic.

The setup can also be dualized. 
Let $\Rep{H}$ be the abelian category of all finite
dimensional graded left $H$-modules.
Its Grothendieck group $[\Rep{H}]$ is again a $\Z[q,q^{-1}]$-algebra.
Taking a dual map to $\gamma$ yields another algebra isomorphism
\begin{equation*}
\gamma^*:\Q(q) \otimes_{\Z[q,q^{-1}]} [\Rep{H}] 
\stackrel{\sim}{\rightarrow} \f.
\end{equation*}
In simply-laced types with $\operatorname{char} \k = 0$,
this sends the basis for $[\Rep{H}]$ arising from isomorphism
classes of graded self-dual irreducible $H$-modules
to the dual canonical basis for $\f$.
In general, the graded self-dual irreducible
$H$-modules still yield a basis for $\f$, but this basis
can be different from the dual
canonical basis;  for an example in type G$_2$ in
characteristic zero see \cite{T}; for an example in type A
in positive characteristic see \cite{W}
and also Example~\ref{willcex} below.
However
several people (e.g. \cite{KOH}) have observed that it
is always a {\em perfect basis} in the sense of 
Berenstein and Kazhdan
\cite{BK}.
This implies for any ground field that the irreducible $H$-modules are
parametrized in a canonical way by Kashiwara's crystal $B(\infty)$ associated
to $\f$, a result established originally by Lauda and Vazirani
\cite{LV} without using the theory of perfect bases.

Using the geometric approach of Varagnolo and Vasserot, hence
for simply-laced types over fields of characteristic zero only, 
Kato \cite{Kato} has explained further how to lift the PBW and dual PBW bases of
$\f$ to certain graded modules $\{\widetilde{E}_b\:|\: b \in
B(\infty)\}$ and $\{E_b\:|\:b \in B(\infty)\}$
over KLR algebras.
We refer to these
modules as {\em standard} and {\em proper standard} modules, respectively,
motivated by the similarity to the theory of properly stratified
algebras \cite{Dlab}.
Kato establishes that each proper standard module $E_b$ has irreducible
head $L_b$, and the modules $\{L_b\:|\:b \in
B(\infty)\}$ give a complete set of graded self-dual irreducible
$H$-modules. The
standard module $\widetilde{E}_b$ is infinite dimensional, and should be
viewed informally
as a ``maximal self-extension'' of 
the finite dimensional proper standard module $E_b$.
Kato's work shows in particular that each of the algebras 
$H_\alpha$ 
has finite global dimension.

More recently in \cite{Mac}, the third author
has found
a purely algebraic way
to 
introduce proper standard modules,
similar in spirit to the approach via
Lyndon words developed in \cite{KR2}
but more general as it makes sense for an arbitrary choice of
the convex ordering $\prec$.
It produces in the end the same collection of
proper standard modules as above but indexed instead by the
set $\KP$ of {\em Kostant partitions}, i.e. non-increasing sequences
$\lambda = (\lambda_1 \succeq \cdots \succeq \lambda_l)$ of positive
roots. 
Switching to this notation, we henceforth denote the {proper standard
  module} 
corresponding to $\lambda$ by 
$\bar{\Delta}(\lambda)$. This module has irreducible head $L(\lambda)$, and
the modules $\{L(\lambda)\:|\:\lambda \in \KP\}$ give a complete
set of graded self-dual irreducible $H$-modules.
Moreover there is a partial order $\preceq$ on $\KP$
with respect to which
the {\em decomposition matrix}
$([\bar\Delta(\lambda):L(\mu)])_{\lambda, \mu \in \KP}$ is
unitriangular, i.e.
$$
[\bar\Delta(\lambda):L(\lambda)] = 1,
\qquad
[\bar\Delta(\lambda): L(\mu)] = 0
\text{ for $\mu \not\preceq\lambda$}.
$$
All of this theory works also for non-simply-laced types and 
ground fields $\k$ of positive
characteristic. 
Finally \cite{Mac} gives a purely algebraic way to 
compute the global dimension
of $H_\alpha$: in all cases it is equal to 
the height of $\alpha \in Q^+$.

Letting $P(\lambda)$ denote the projective cover of $L(\lambda)$,
the standard module $\Delta(\lambda)$ corresponding to $\lambda$ may
be defined as
$$
\Delta(\lambda) := P(\lambda) \bigg/ 
\sum_{\phantom{q^n}\mu\not\preceq\lambda\phantom{q^n}}
\:\:\sum_{f:P(\mu) \rightarrow P(\lambda)} \im f.
$$
Taking graded duals, we also have the {\em costandard module} $\nabla(\lambda) := \Delta(\lambda)^\circledast$
and the {\em proper costandard module} $\bar\nabla(\lambda) := \bar\Delta(\lambda)^\circledast$.
For simply-laced types in characteristic zero, Kato showed that these modules
satisfy various homological properties familiar from the
theory of quasi-hereditary algebras. Perhaps the most important of
these is the following:
$$
\EXT^d_H(\Delta(\lambda), \bar\nabla(\mu)) \cong
\left\{\begin{array}{ll}
\k&\text{if $d = 0$ and $\lambda = \mu$,}\\
0&\text{otherwise.}
\end{array}
\right.
$$
There are many pleasant consequences. For example, one can deduce
that the projective module $P(\lambda)$ has a finite filtration with
sections of the form $\Delta(\mu)$ and multiplicities satisfying
{\em BGG reciprocity}:
$$
[P(\lambda):\Delta(\mu)] = [\bar\Delta(\mu):L(\lambda)].
$$
The main purpose of this article is to explain an elementary approach to
the proof of these homological properties starting from the results of
\cite{Mac}. Our results apply to all finite types and fields
$\k$ of
arbitrary characteristic.
In fact we will see that the formal characters of the standard and 
proper standard modules do not depend on the characteristic of the
ground field (unlike for either $P(\lambda)$ or $L(\lambda)$).

The basic idea is to exploit a different definition of
the standard module $\Delta(\lambda)$.
To start with, we construct {\em root
modules} $\Delta(\alpha)$ for each $\alpha \in R^+$ by taking an
inverse limit of some iterated self-extensions
of the irreducible module $L(\alpha)$; these modules categorify
Lusztig's root vectors $r_\alpha \in \f$.
A key new observation is that the endomorphism algebra
of a product $\Delta(\alpha)^{\circ m}$ of $m$ copies of
$\Delta(\alpha)$ is
isomorphic to the {\em nil Hecke algebra} $\NH_m$.
Hence we can define the {\em divided power module}
$\Delta(\alpha^m)$ by using a primitive idempotent in the nil Hecke algebra to project to 
an indecomposable direct summand of $\Delta(\alpha)^{\circ m}$.
For $\lambda =
(\gamma_1^{m_1},\dots,\gamma_s^{m_s})$ with $\gamma_1 \succ \cdots
\succ \gamma_s$ we then show that
$$
\Delta(\lambda) \cong \Delta(\gamma_1^{m_1}) \circ \cdots \circ
\Delta(\gamma_s^{m_s}),
$$
and proceed  to derive the homological properties by applying generalized
Frobenius reciprocity;
see Theorem~\ref{shp} for the final result.

We also explain an alternative way to construct the
root module $\Delta(\alpha)$ by induction on height. 
For simple $\alpha$ the
root module $\Delta(\alpha)$ is just the regular module $H_\alpha$.
Then for a non-simple positive root $\alpha$, we pick a {\em minimal
  pair} $(\beta,\gamma)$ for $\alpha$ in the sense 
of \cite{Mac} (in particular 
$\beta$ and $\gamma$ are positive roots summing to $\alpha$)
and show in Theorem~\ref{inj3} that there exists a short exact sequence
$$
0 \longrightarrow q^{-\beta\cdot\gamma} \Delta(\beta) \circ \Delta(\gamma)
\stackrel{\phi}{\longrightarrow} \Delta(\gamma) \circ \Delta(\beta) \longrightarrow
[p_{\beta,\gamma}+1]
\Delta(\alpha) \longrightarrow 0,
$$
where $p_{\beta,\gamma}$ is the largest integer $p$
such that $\beta - p \gamma$ is a root, and $[n]$ denotes 
the quantum integer.
The map $\phi$ in this short exact sequence is defined explicitly,
so $\Delta(\alpha)$ could instead
be {defined} recursively in terms of its cokernel.
%Note also on passing to the Grothendieck group that this short exact sequence
%implies
%$$
%r_\alpha = (r_\gamma r_\beta - q^{-\beta\cdot\gamma} r_\beta r_\gamma)
%/ [p_{\beta,\gamma}+1],
%$$
%giving a useful recursive formula for the root element
%$r_\alpha$.

For multiplicity-free positive roots ($=$ all positive roots in type
A) the above short exact sequences can be assembled into some explicit
projective resolutions
of the root modules, which can be viewed as a variation
on the classical Koszul resolution from commutative algebra; see Theorem~\ref{expres}.
The first non-trivial example comes from the highest root $\alpha$
in type A$_3$. Adopting the notation of Example~\ref{eg1} below, our resolution of $\Delta(\alpha)$ in
this special case produces an exact sequence
$$
0 \longrightarrow q^2 H_\alpha 1_{321}
\stackrel{(-\tau_{1}\tau_{2}\:\:\:\tau_{2})}{\longrightarrow}
q H_\alpha 1_{213}
\oplus
q H_\alpha 1_{312}
\stackrel{\binom{\tau_1}{\tau_1\tau_2}}{\longrightarrow}
H_\alpha 1_{123}
\longrightarrow \Delta(\alpha)
\longrightarrow 0,
$$
where we view elements of the direct sum as row vectors and the maps
are defined by right multiplication by the given matrices.

This article supersedes the preprint \cite{BrK} which considered
simply-laced types only. We also point out that Theorem~\ref{as} below
proves the length two conjecture formulated in \cite[Conjecture 2.16]{BrK}.

\vspace{2mm}
\noindent
{\em Conventions.}
By a module 
over a $\Z$-graded algebra $H$, we {\em always} mean a graded left
$H$-module;
likewise all submodules, quotient modules, and so on are graded.
We write $\rad V$ (resp.\ $\soc V$) 
for the intersection of all maximal submodules 
(resp. the sum of all irreducible submodules) of $V$. 
We write $q$ for the upwards degree shift functor:
if $V = \bigoplus_{n \in \Z} V_n$ 
then 
$qV$ 
has $(qV)_n := V_{n-1}$.
More generally, given a formal Laurent series $f(q) = \sum_{n \in \Z} f_n
q^n$ with coefficients $f_n \in \N$,
$f(q) V$ denotes $\bigoplus_{n \in \Z} q^n V^{\oplus f_n}$.
For modules $U$ and $V$,
we write
$\hom_H(U, V)$
for homogeneous $H$-module homomorphisms, reserving 
$\HOM_H(U, V)$ for the graded vector space
$\bigoplus_{n \in \Z} \HOM_H(U, V)_n$
where
$$
\HOM_H(U, V)_n := \hom_H(q^n U, V) = \hom_H(U, q^{-n}V).
$$
We define $\ext^d_H(U,V)$ and
$\EXT^d_H(U,V)$ similarly.
If $V$ is a locally finite dimensional graded vector space,
its graded dimension is $$
\DIM V := \sum_{n \in \Z} (\dim V_n) q^n.
$$
For formal Laurent series $f(q) = \sum_{n \in \Z} f_n q^n$ and $g(q) =
\sum_{n \in \Z} g_n q^n$, we write $f(q) \leq g(q)$ if $f_n \leq g_n$
for all $n \in \Z$.
We need the following generality several times.

\begin{Lemma}\label{mittagleffler}
Let $H$ be a $\Z$-graded algebra which is locally finite dimensional
and bounded below.
Fix $d > 0$ and let $U$ and $V$ be finitely generated $H$-modules.
If $\ext^d_H(U, L) = 0$ for all irreducible subquotients $L$ of $V$, then
$\ext^d_{H}(U,V)=0$.
\end{Lemma}

\begin{proof}
If $V$ is finite dimensional this is an easy induction exercise 
using the long exact sequence.
Now assume that $V$ is infinite dimensional. 
The assumptions imply that $V$ has an exhaustive filtration
$V = V_0 \supseteq V_1 \supseteq \cdots$ in which
each $V / V_r$ is finite dimensional; for example one can let
$V_r$ be the submodule generated by all homogeneous vectors of degree $\geq r$ in $V$.
Then we have that $V = \varprojlim(V/V_r)$. 
By \cite[Theorem 3.5.8]{Wei}, there is
a short exact sequence
$$
0 \longrightarrow
{\varprojlim}^1 \ext^{d-1}_{A}(U, V/V_r) \longrightarrow \ext^d_{A}(U,V)
\longrightarrow \varprojlim \ext^{d}_{A}(U, V/V_r) \longrightarrow 0.
$$
The last term is zero, so
we just need to show that $\varprojlim^1 \ext^{d-1}_{A}(U, V/V_r)=0$.
This follows by \cite[Proposition 3.5.7]{Wei} if we can show that the
tower 
$\left(\ext^{d-1}_{H}(U, V/V_r)\right)$ satisfies the Mittag-Leffler condition, i.e. the
natural map $ \ext^{d-1}_{H}(U, V/V_{r+1}) \rightarrow  \ext^{d-1}_{H}(U, V/V_r)$ is
surjective for each $r \geq 0$. 
This follows on applying $\hom_{H}(U, -)$ to
the short exact sequence $0  \rightarrow V_r/V_{r+1}  \rightarrow V/V_{r+1}
\rightarrow V/V_r  \rightarrow 0$.
\end{proof}

\section{KLR algebras}

We begin by collecting some basic facts about the representation theory of
finite type KLR algebras. 
The discussion of the contravariant form on proper standard modules
in $\S$\ref{sscontra} is new.

\subsection{The twisted bialgebra $\f$}
Let $\mathfrak{g}$ be a finite dimensional complex semisimple Lie
algebra.
Let 
$R$ be the root system
of $\mathfrak{g}$ with respect to some Cartan subalgebra, $R^+ \subset
R$ be a set of positive roots, and
$\{\alpha_i\:|\:i \in I\}$ be the corresponding simple roots.
Let $Q := \bigoplus_{i \in I} \Z \alpha_i$ be the {\em root lattice},
$Q^+ := \bigoplus_{i \in I} \N \alpha_i$, and 
define the {\em height} of $\alpha  = \sum_{i \in I} c_i \alpha_i \in
Q^+$
from $\height(\alpha) := \sum_{i \in I} c_i$.
Let
$$
Q \times Q \rightarrow \Z, \qquad
(\alpha,\beta) \mapsto \alpha\cdot\beta
$$
be a positive definite symmetric bilinear form
normalized so that $d_i := \alpha_i \cdot \alpha_i
/ 2$ is a positive integer for each $i \in I$ and $\sum_{i \in I} d_i$ as small as
possible; in particular this means that $d_i = 1$ for all $i$ if
$\mathfrak{g}$ is simply-laced. Then for $\alpha \in R^+$ we let $d_\alpha := \alpha\cdot
\alpha / 2$.
The {\em Cartan matrix} is the matrix $C = (c_{i,j})_{i,j \in I}$
defined from $c_{i,j} := \frac{1}{d_i} \alpha_i \cdot \alpha_j$.
Finally we have the {\em Weyl group} $W$, which is the subgroup of
$GL(Q)$ generated by the simple reflections $\{s_i\:|\:i \in I\}$
defined from $s_i(\beta) := \beta - \frac{1}{d_i}(\alpha_i \cdot \beta) \alpha_i$.

Now let $q$ be an
 indeterminate 
and $\AA := \Q(q)$.
For $n \in \Z$
let $[n]$ be the quantum
integer $(q^n-q^{-n}) / (q-q^{-1})$.
Assuming $n \geq 0$ let
$[n]^! := [n][n-1]\cdots [1]$ be the quantum factorial.
More generally, for $i \in I$ (resp.\ $\alpha \in R^+$)
let $[n]_i$ and $[n]_i^!$ (resp.\ $[n]_\alpha$ and $[n]_\alpha^!$)
denote the quantum integer and quantum factorial with $q$
replaced by $q_i := q^{d_i}$ (resp.\ $q_\alpha := q^{d_\alpha}$).
Let $\f$ be the free associative $\AA$-algebra
on generators $\{\theta_i\:|\:i \in I\}$
subject to the {\em quantum Serre relations}
\begin{equation*}
\sum_{r+s=1 - c_{i,j}}
(-1)^r
\theta_i^{(r)}\theta_j \theta_i^{(s)} = 0
\end{equation*}
for all $i, j \in I$ and $r \geq 1$,
where $\theta_i^{(r)}$ denotes the {\em divided power} $\theta_i^r /
[r]_i^!$.

There is a $Q^+$-grading $\f = \bigoplus_{\alpha \in Q^+} \f_\alpha$
defined so that $\theta_i$ is in degree $\alpha_i$.
Viewing $\f \otimes \f$ as an algebra 
with multiplication $(a \otimes b)(c \otimes d) := q^{-\beta \cdot \gamma} ac
\otimes bd$ for $a \in \f_\alpha, b \in \f_\beta, c \in \f_\gamma$ and
$d \in \f_\delta$,
there is a unique algebra homomorphism
\begin{equation}\label{twistedco}
r:\f \rightarrow \f \otimes \f,
\qquad
\theta_i\mapsto \theta_i \otimes 1 + 1 \otimes \theta_i
\end{equation}
making $\f$ into a twisted bialgebra.
In his book \cite[$\S$1.2.5, $\S$33.1.2]{Lubook}, Lusztig shows
further that $\f$ possesses a unique non-degenerate
symmetric bilinear form $(\cdot,\cdot)$ such that
$$
(1,1) = 1, 
\qquad
(\theta_i, \theta_j) = \frac{\delta_{i,j}}{1-q_i^{2}},
\qquad
(ab, c) = (a \otimes b, r(c))
$$
for all $i,j \in I$ and $a,b,c \in \f$; on the right hand side of the last equation
 $(\cdot,\cdot)$ is the product form on $\f \otimes \f$
defined from $(a \otimes b, c \otimes d) := (a,c) (b,d)$.
Note here that our $q$ is Lusztig's $v^{-1}$.

Let $\A := \Z[q,q^{-1}] \subset \AA$. Lusztig's $\A$-form $\f_\A$ for $\f$ is the $\A$-subalgebra of
$\f$ generated by all $\theta_i^{(r)}$.
Also let $\f_\A^*$ be 
the dual of $\f_\A$ with respect to the form $(\cdot,\cdot)$, i.e.
$\f_\A^* := \left\{y \in \f\:\big|\:(x,y)\in\A\text{ for all }x \in
  \f_\A\right\}$.
It is another $\A$-subalgebra of $\f$.
Moreover, both $\f_\A$ and $\f_\A^*$ are free as $\A$-modules, and we can identify
$$
\f = \AA \otimes_\A \f_\A = \AA \otimes_\A \f_\A^*.
$$
The field $\AA$ possesses a unique automorphism called the
{\em bar involution} such that $\overline{q} = q^{-1}$.
With respect to this involution, 
let $\b:\f \rightarrow \f$
be the anti-linear algebra automorphism
such that $\b(\theta_i) = \theta_i$ for all $i \in I$.
Also let 
$\b^*:\f \rightarrow \f$
be the adjoint anti-linear
map to $\b$ 
with respect to Lusztig's form, so $\b^*$ is
defined from 
$(x, \b^*(y)) = \overline{(\b(x), y)}$
for any $x, y \in \f$.
The maps $\b$ and $\b^*$ preserve $\f_\A$ and $\f_\A^*$, respectively.

Next let $\langle I \rangle$ be the free monoid on $I$, that is, the set
of all {\em words} $\bi = i_1 \cdots i_n$ for $n \geq 0$ and $i_1,\dots,i_n
\in I$ with multiplication given by concatenation of words.
For  a word $\bi = i_1 \cdots i_n$ of length $n$
and a permutation $w \in S_n$,  
we let
\begin{align*}
|\bi| &:= \alpha_{i_1}+\cdots+\alpha_{i_n},
&w(\bi)&:= i_{w^{-1}(1)} \cdots i_{w^{-1}(n)},\\
\theta_\bi &:= \theta_{i_1}\cdots \theta_{i_n},
&
\deg(w;\bi) &:= -\sum_{\substack{1 \leq j < k \leq n\\ w(j) > w(k)}}
\alpha_{i_j}\cdot\alpha_{i_k}.
\end{align*}
Setting
$\langle I \rangle_\alpha := \big\{\bi \in \langle I
\rangle\:\big|\:|\bi|=\alpha\big\},$
the monomials $\big\{\theta_\bi\:\big|\:\bi \in\langle I \rangle_\alpha\big\}$
span $\f_\alpha$.
\iffalse
The form $(\cdot,\cdot)$ is given explicitly on these elements by
the formula
\begin{equation}\label{formy}
(\theta_\bi, \theta_\bj) = \frac{1}{(1-q_{i_1}^{2}) \cdots (1-q_{i_n}^{2})}
\sum_{\substack{w \in S_n \\ w(\bi) = \bj}}
q^{\deg(w;\bi)}
\end{equation}
for all words $\bi, \bj \in \langle I \rangle$ with $\bi$ of length
$n$.
\fi
The {\em quantum shuffle algebra} is
the free $\A$-module $\A \W = \bigoplus_{\alpha \in Q^+} \A \W_\alpha$
on basis $\W$, viewed as an $\A$-algebra via the 
{\em shuffle product} $\circ$
defined on words $\bi$ and $\bj$ of lengths $m$ and $n$, respectively, by
\begin{equation}\label{shuffle}
\bi \circ \bj := \sum_{\substack{w \in S_{m+n} \\ w(1) < \cdots <
    w(m)\\w(m+1) < \cdots < w(m+n)}}
q^{\deg(w;\bi\bj)} w(\bi\bj).
\end{equation}
As observed originally by Green \cite{Green} and Rosso \cite{Rosso0},  
there is an injective $\A$-algebra homomorphism
\begin{equation}\label{charc}
\CH:\f_\A^* \rightarrow 
\A \W, \qquad
x \mapsto \sum_{\bi \in \W} (\theta_\bi, x) \bi.
\end{equation}
This intertwines the anti-linear involution $\b^*$ on $\f_\A^*$
with the bar involution on $\A\W$, which is defined from
$\overline{\sum_{\bi \in \W} a_\bi \bi} := \sum_{\bi \in \W}
\overline{a}_\bi \bi$.
Using (\ref{shuffle}), one checks further that $\overline{\bi \circ \bj}
= q^{|\bi|\cdot|\bj|} \bj \circ \bi$.
Hence
\begin{equation}\label{bstar}
\b^*(xy) = q^{\beta\cdot \gamma} \b^*(y) \b^*(x)
\end{equation}
for $x \in \f_\beta$ and $y \in \f_\gamma$.
Using this and induction on height, it follows
that
\begin{equation}\label{zstar}
\b^*(x) = (-1)^n q^{d_\alpha+d_{i_1}+\cdots+d_{i_n}} \b(\sigma(x))
\end{equation}
for $x \in \f_\alpha$ and $\alpha = \alpha_{i_1}+\cdots+\alpha_{i_n}
\in Q^+$,
where $\sigma:\f \rightarrow \f$ is the algebra anti-automorphism
such that $\sigma(\theta_i) = \theta_i$ for each $i \in I$.

\subsection{The KLR algebra}\label{newSeccat}
Fix now a field $\k$.
Also choose signs $\eps_{i,j}$ for all $i,j \in I$ with $c_{i,j}
< 0$  so that $\eps_{i,j}\eps_{j,i} = -1$.
For $\alpha \in Q^+$ of height $n$, the {\em KLR algebra} $H_\alpha$
is the associative, unital $\k$-algebra
defined by generators 
$$
\{1_\bi\:|\:\bi \in \langle I \rangle_\alpha\}
\cup\{\x_1,\dots,\x_n\}
\cup\{\tau_1,\dots,\tau_{n-1}\}
$$
subject only to the following relations:
\begin{itemize}
\item$\x_k \x_l = \x_l \x_k$;
\item
the elements $\left\{1_\bi\:|\:\bi \in \langle I \rangle_\alpha\right\}$
are mutually orthogonal idempotents whose sum is the identity
$1_\alpha \in H_\alpha$;
\item
$\x_k 1_\bi = 1_\bi\x_k$ and 
$\tau_k 1_\bi = 1_{(k\:k\!+\!1)(\bi)}
\tau_k$;
\item 
$(\tau_k \x_l -\x_{(k\:k\!+\!1)(l)} \tau_k)1_\bi 
=\left\{
\begin{array}{ll}
\phantom{-}1_\bi&\text{if $i_k = i_{k+1}$ and $l=k+1$},\\
-1_\bi&\text{if $i_k = i_{k+1}$ and $l=k$},\\
\phantom{-}0&\text{otherwise};
\end{array}\right.$
\item
$
\tau_k^2 1_\bi = 
\left\{
\begin{array}{ll}
0&\text{if $i_k=i_{k+1}$,}\\
\eps_{i_k,i_{k+1}}\big({x_k}^{-c_{i_k,i_{k+1}}}-{x_{k+1}}^{-c_{i_{k+1},i_k}}\big)1_\bi&\text{if
  $c_{i_k,i_{k+1}}< 0$,}\\
1_\bi&\text{otherwise;}
\end{array}\right.
$
\item $\tau_k \tau_l = \tau_l \tau_k$ if $|k-l|>1$;
\item 
$(\tau_{k+1} \tau_{k} \tau_{k+1} -
  \tau_{k}\tau_{k+1}\tau_{k}) 1_\bi
=$\\
\phantom{\hspace{15mm}}$\left\{
\begin{array}{ll}
\displaystyle\sum_{r+s=-1-c_{i_k,i_{k+1}}}
\!\!\!\!\!\!\!
\eps_{i_k,i_{k+1}} x_k^r x_{k+2}^s 1_\bi&\text{if
  $c_{i_k,i_{k+1}} < 0$ and $i_k =
  i_{k+2}$,}\\
\hspace{9mm}0&\text{otherwise.}
\end{array}\right.
$
\end{itemize}
The algebra $H_\alpha$ is $\Z$-graded with $1_\bi$ in
degree zero, $\x_k 1_\bi$ in degree $2 d_{i_k}$ and 
$\tau_k 1_\bi$ in degree $-\alpha_{i_k} \cdot \alpha_{i_{k+1}}$.
There is also an anti-automorphism $T:H_\alpha \rightarrow H_\alpha$
which fixes all the generators.

Here are
a few other basic facts about the structure of these algebras established in
\cite{KL, KL2} or \cite{R}.
Fix once and for all a reduced expression for each $w \in S_n$
  and let $\tau_w$ be the corresponding product of the
  $\tau$-generators
of $H_\alpha$. Note that $\tau_w 1_\bi$ is of degree $\deg(w;\bi)$.
The monomials
\begin{equation}\label{basisthm}
\{\x_1^{k_1} \cdots \x_n^{k_n} \tau_w  1_\bi\:|\:w \in S_n,
k_1,\dots,k_n \geq 0, \bi \in \W_\alpha\}
\end{equation} 
give a basis for
$H_\alpha$.
In particular, $H_\alpha$ is locally finite dimensional and bounded
below.
There is also an explicit description of the center $Z(H_\alpha)$,
from which it follows that $H_\alpha$ is free of
finite rank as a module over its center; forgetting the grading the
rank is $(n!)^2$.

For $m \geq 1$ and $i \in I$,
the KLR algebra $H_{m\alpha_i}$ is identified with the
{\em nil Hecke algebra} $\NH_m$, that is, the algebra with generators
$x_1,\dots,x_m$ and 
$\tau_1,\dots,\tau_{m-1}$ subject to the following relations:
$x_i x_j = x_j x_i$; $\tau_i x_j = x_j \tau_i$ for $j \neq
i,i+1$; $\tau_i x_{i+1} = x_i \tau_i + 1$; $x_{i+1} \tau_i
= \tau_i x_i + 1$; $\tau_i^2 = 0$;
and the usual type A braid relations amongst
$\tau_1,\dots,\tau_{m-1}$.
It is well known
that the nil Hecke algebra is a matrix algebra over its center;
see e.g. \cite[$\S$2]{R2} for a recent exposition.
Moreover,
writing $w_{[1,m]}$ for the longest element of $S_m$,
the degree zero 
element
\begin{equation}\label{idemp}
e_m := \x_2 \x_3^2 \cdots \x_m^{m-1} \tau_{w_{[1,m]}}
\end{equation} 
is a primitive idempotent, hence
$P(\alpha_i^m) := q_i^{\frac{1}{2}m(m-1)}H_{m\alpha_i} e_m$ is an indecomposable projective $H_{m\alpha_i}$-module.
The degree shift here has been chosen so that 
irreducible head $L(\alpha_i^m)$ of $P(\alpha_i^m)$ has graded dimension $[m]_i^!$.
Thus 
$H_{m\alpha_i}\cong [m]^!_i P(\alpha_i^m)$ as a left module.

For
$\beta,\gamma \in Q^+$, there is an evident
non-unital algebra
embedding $H_\beta \otimes H_\gamma \hookrightarrow H_{\beta+\gamma}$.
We denote the image of the identity
$1_\beta \otimes 1_\gamma \in H_\beta \otimes H_\gamma$
by $1_{\beta,\gamma} \in H_{\beta+\gamma}$. Then for
an $H_{\beta+\gamma}$-module $U$ and
an $H_\beta \otimes H_\gamma$-module $V$,
we set
\begin{equation*}
\Res^{\beta+\gamma}_{\beta,\gamma} U := 1_{\beta,\gamma} U,\qquad
\Ind^{\beta+\gamma}_{\beta,\gamma} V := H_{\beta+\gamma} 1_{\beta,\gamma} \otimes_{H_\beta
  \otimes H_\gamma} V,
\end{equation*}
which are naturally $H_\beta \otimes H_\gamma$- and
$H_{\beta+\gamma}$-modules, respectively.
These definitions extend in an obvious way to situations where there
are more than two tensor factors.
The following Mackey-type theorem is of crucial importance. 

\begin{Theorem}\label{mackey}
Suppose we are given $\beta,\gamma,\beta',\gamma' \in Q^+$
of heights $m,n,m',n'$, respectively, such that
$\beta+\gamma = \beta'+\gamma'$.
Setting $k := \min(m,n,m',n')$, let
$\{1 = w_0 < \cdots < w_k\}$ 
be the set of minimal length $S_{m'} \times S_{n'} \backslash
S_{m+n} / S_m \times S_n$-double coset representatives ordered
via the Bruhat order.
For any $H_\beta \otimes H_\gamma$-module $V$,
there is a filtration
$$
0 = V_{-1}\subseteq V_0 \subseteq V_1 \subseteq\cdots\subseteq V_k = 
\Res^{\beta'+\gamma'}_{\beta',\gamma'} \circ \Ind^{\beta+\gamma}_{\beta,\gamma} (V)
$$
defined by $V_j := \sum_{i=0}^j\sum_{w \in (S_{m'} \times S_{n'}) w_i (S_m \times
  S_n)} 1_{\beta',\gamma'}\tau_w 1_{\beta,\gamma}\otimes V$.
Moreover there is a unique isomorphism of $H_{\beta'} \otimes
H_{\gamma'}$-modules
\begin{align*}
V_j / V_{j-1}
&\stackrel{\sim}{\rightarrow}
\bigoplus_{\beta_1,\beta_2,\gamma_1,\gamma_2}
 q^{-\beta_2 \cdot \gamma_1}\Ind^{\beta',\gamma'}_{\beta_1,
\gamma_1,\beta_2, \gamma_2} \circ I^* \circ
\Res^{\beta,\gamma}_{\beta_1,\beta_2,\gamma_1, \gamma_2} (V),\\
1_{\beta',
\gamma'} \tau_{w_j} 1_{\beta,\gamma}
\otimes 
v + V_{j-1}  &\mapsto \sum_{\beta_1,\beta_2,\gamma_1,\gamma_2}
1_{\beta_1,\gamma_1,\beta_2,\gamma_2} \otimes 1_{\beta_1,\beta_2,\gamma_1,\gamma_2} v,
\end{align*}
where 
$I:
H_{\beta_1} \otimes H_{\gamma_1}\otimes H_{\beta_2}
\otimes H_{\gamma_2}\stackrel{\sim}{\rightarrow}
H_{\beta_1} \otimes H_{\beta_2}\otimes H_{\gamma_1}
\otimes H_{\gamma_2}$ is the obvious isomorphism, and
the sums are taken over all
$\beta_1,\beta_2,\gamma_1,\gamma_2 \in Q^+$
such that 
$\beta_1+\beta_2 = \beta, \gamma_1+\gamma_2 = \gamma,
\beta_1+\gamma_1 = \beta', \beta_2+\gamma_2 = \gamma'$
and $\min(\operatorname{ht}(\beta_2), \operatorname{ht}(\gamma_1)) =
j$:
$$
\begin{picture}(100,98)
\put(-25,45){$w_j=$}
\put(11,15){$_{\beta_1}$}
\put(34,15){$_{\beta_2}$}
\put(61,15){$_{\gamma_1}$}
\put(86,15){$_{\gamma_2}$}
\put(11,75){$_{\beta_1}$}
\put(30,75){$_{\gamma_1}$}
\put(57,75){$_{\beta_2}$}
\put(86,75){$_{\gamma_2}$}
\put(7,10){$\underbrace{\hspace{16mm}}_\beta$}
\put(58 ,10){$\underbrace{\hspace{16mm}}_\gamma$}
\put(7,80){$\overbrace{\hspace{12.5mm}}^{\beta'}$}
\put(48 ,80){$\overbrace{\hspace{19mm}}^{\gamma'}$}
\put(10,20){\line(0,1){50}}
\put(20,20){\line(0,1){50}}
\put(30,20){\line(2,5){20}}
\put(40,20){\line(2,5){20}}
\put(50,20){\line(2,5){20}}
\put(60,20){\line(-3,5){30}}
\put(70,20){\line(-3,5){30}}
\put(80,20){\line(0,1){50}}
\put(90,20){\line(0,1){50}}
\put(100,20){\line(0,1){50}}
\end{picture}
$$
\end{Theorem}

\begin{proof}
This follows as in \cite[Proposition 2.18]{KL}. 
\end{proof}

\subsection{The categorification theorem}
Let $\Rep{H_\alpha}$ denote the abelian category of finite dimensional
$H_\alpha$-modules
and set 
$$
\Rep{H} := \bigoplus_{\alpha \in Q^+} \Rep{H_\alpha}.
$$
This is a graded $\k$-linear monoidal category with respect to the induction
product
$U \circ V := \Ind_{\beta,\gamma}^{\beta+\gamma}
(U \boxtimes V)$ for $U \in \Rep{H_\beta}$ and $V \in
\Rep{H_\gamma}$.
Let $[\Rep{H}] = \bigoplus_{\alpha \in Q^+} [\Rep{H_\alpha}]$
denote its Grothendieck ring, which we make into an $\A$-algebra 
so that $q [V] = [q V]$.
Dually, we have the additive category $\Proj{H_\alpha}$ of finitely
generated projective $H_\alpha$-modules and set
$$
\Proj{H} := \bigoplus_{\alpha \in Q^+} \Proj{H_\alpha}.
$$
Again this is a graded $\k$-linear monoidal category with respect to
the induction product, and again the split Grothendieck group
$[\Proj{H}] = \bigoplus_{\alpha \in Q^+} [\Proj{H_\alpha}]$ is
naturally an $\A$-algebra.
Moreover there is a non-degenerate pairing
\begin{equation*}
(\cdot,\cdot):[\Proj{H}] \times [\Rep{H}] \rightarrow
\A
\end{equation*}
defined on $P \in \Proj{H_\alpha}$ and $V \in \Rep{H_\beta}$
by declaring that 
$$
([P], [V]) := 
\left\{
\begin{array}{ll}
\DIM T^*(P) \otimes_{H_\alpha} V&\text{if $\beta = \alpha$,}\\
0&\text{otherwise,}
\end{array}\right.
$$
where $T^*(P)$ denotes $P$ viewed as a right module via the
anti-automorphism $T$.
Finally there are dualities $\circledast$ on $\Rep{H_\alpha}$ and $\#$
on $\Proj{H_\alpha}$ 
inducing antilinear involutions on the Grothendieck groups.
These are defined from $V^\circledast := \HOM_{\k}(V, \k)$
and $P^\# := \HOM_{H_\alpha}(P, H_\alpha)$, respectively, both viewed as left modules
via $T$;
more generally $V^\circledast$ makes
good sense for any $V$ that is locally finite dimensional.
%Also, by \cite[Lemma 3.2]{KR2}, we have that 
%$([P^\#], [V]) = \overline{([P], [V^\circledast])}$.

\begin{Theorem}[Khovanov-Lauda]\label{ct}
There is a unique adjoint pair of $\A$-algebra isomorphisms
$$
\gamma:\f_\A \stackrel{\sim}{\rightarrow} [\Proj{H}],
\qquad
\gamma^*:[\Rep{H}] \stackrel{\sim}{\rightarrow}
\f_\A^*
$$
such that
$\gamma(\theta_i^{(n)})=[P(\alpha_i^n)]$. Under these isomorphisms,
the antilinear involutions $\b$ and $\b^*$ on $\f_\A$ and $\f_\A^*$
correspond to the dualities $\#$ and $\circledast$, respectively.
\end{Theorem}

\begin{proof}
See \cite[$\S$3]{KL} and \cite[Theorem 8]{KL2} for the statement about $\gamma$. The dual
statement is implicit in \cite{KL}; see also \cite[Theorem 4.4]{KR2}.
\end{proof}

Henceforth we will {\em
identify} $\f_\A$ with $[\Proj{H}]$ and $\f_\A^*$ with $[\Rep{H}]$
according to Theorem~\ref{ct}.
Any $H_\alpha$-module $V$
admits a decomposition 
into {\em word spaces}
$V = \bigoplus_{\bi \in \W_\alpha} 1_\bi V$.
Then the {\em character} of $V \in \Rep{H_\alpha}$ is the formal sum
\begin{equation}
\CH V = \sum_{\bi \in \W_\alpha} (\DIM 1_\bi V) \bi \in \A \W_\alpha.
\end{equation}
As $(\theta_\bi, [V]) = ([H_\alpha 1_\bi], [V]) = \DIM 1_\bi H_\alpha
\otimes_{H_\alpha} V = \DIM 1_\bi V$, we have that $\CH V = \CH [V]$,
where $\CH$ on the right hand side is the injective map from (\ref{charc}).
Also note the following, which is the module-theoretic
analogue of (\ref{bstar}).

\begin{Lemma}\label{lvl}
For $U \in \Rep{H_\beta}$ and $V \in \Rep{H_\gamma}$, there is a
natural isomorphism
$(U \circ V)^\circledast \cong q^{\beta\cdot\gamma} V^\circledast \circ U^\circledast.$
\end{Lemma}

\begin{proof}
This is \cite[Theorem 2.2(2)]{LV}.
\end{proof}

\subsection{PBW and dual PBW bases}\label{pbws}
A {\em convex ordering} on $R^+$ is a total order $\prec$
such that
$$
\beta,\gamma, 
\beta+\gamma \in R^+, \beta \prec \gamma \quad\Rightarrow \quad
\beta \prec \beta + \gamma \prec \gamma.
$$
By \cite{Papi}, there is a bijection between convex orderings of 
$R^+$ and reduced expressions for the longest element $w_0$ of $W$:
given a reduced expression
$w_0 = s_{i_1} \cdots s_{i_N}$ the corresponding 
convex ordering on $R^+$ is given by
$$
\alpha_{i_1} \prec
s_{i_1}(\alpha_{i_2})
\prec
s_{i_1} s_{i_2}(\alpha_{i_3})
\prec\cdots\prec s_{i_1} \cdots s_{i_{N-1}}(\alpha_{i_N}).
$$
We assume henceforth that such a convex ordering/reduced expression
has been specified. The following lemma is very useful.

\begin{Lemma}\label{l1}
Suppose we are given positive roots $\alpha, \beta_1,\dots,\beta_k, \gamma_1,\dots,\gamma_l$
such that $\beta_i \preceq \alpha \preceq \gamma_j$
for all $i$ and $j$.
We have that 
$\beta_1+\cdots+\beta_k = \gamma_1+\cdots + \gamma_l$
if and only if $k=l$
 and
$\beta_1=\cdots=\beta_k=\gamma_1=\cdots=\gamma_l = \alpha$.
\end{Lemma}

\begin{proof}
Suppose that $\beta_1+\cdots+\beta_k = \gamma_1+\cdots+\gamma_l$.
We may assume for suitable $0 \leq k' \leq k$ and $0 \leq l' \leq l$
that 
$\beta_i = \alpha$ for $1 \leq i \leq k'$,
$\beta_i \prec \alpha$ for $k'+1 \leq i \leq k$ and
$\gamma_i = \alpha$ for $1 \leq i \leq l'$,
$\gamma_i \succ \alpha$ for $l'+1 \leq i \leq l$.
Then we need to show that  $k=k'=l' = l$.
Assume the convex ordering corresponds to reduced expression
$w_0 = s_{i_1} \cdots s_{i_N}$ as above. 
Then $\alpha = s_{i_1} \cdots s_{i_{j-1}} (\alpha_{i_j})$
for a unique $1 \leq j \leq N$.
If $k' \geq l'$,
let $w := s_{i_j} \cdots s_{i_1}$.
From $\beta_1+\cdots+\beta_k = \gamma_1+\cdots+\gamma_l$, we deduce
that
$$
(k'-l') w(\alpha) + w(\beta_{k'+1})+\cdots+w(\beta_k) =
w(\gamma_{l'+1}) + \cdots + w(\gamma_l).
$$
By \cite[Ch. VI, $\S$6, Cor. 2]{Bou},
the set of positive roots sent to negative roots by $w$ is the set
$\{\alpha' \in R^+\:|\:\alpha' \preceq \alpha\}$.
Hence the left hand side of the above equation 
is a sum of negative roots and the right hand side
is a sum of positive roots. 
So both sides are zero and we deduce that $k=k'= l' = l$.
For the case $k' \leq l'$, argue in a similar way
with $w:=s_{i_{j-1}} \cdots s_{i_1}$, so that the set of
positive roots sent to negative by $w$ is $\{\alpha' \in
R^+\:|\:\alpha' \prec \alpha\}$.
\end{proof}

Corresponding to the chosen convex ordering/reduced expression, 
Lusztig has introduced {\em root vectors}
$\{r_\alpha\:|\:\alpha \in R^+\}$ in $\f$ via a
certain braid group action. The definition uses the
embedding of $\f$ into the
full quantum group $U_q(\mathfrak{g})$ 
so we only summarize it briefly: 
we take the positive embedding $\f \hookrightarrow U_q(\mathfrak{g}),
x \mapsto x^+$ defined from
$\theta_i^+ := E_i$ and use the braid group generators
$T_i := T_{i,+}''$ from \cite[$\S$37.1.3]{Lubook} (recalling our $q$
is Lusztig's $v^{-1}$);
then for $\alpha \in R^+$ the root element $r_\alpha$ is the unique
element of $\f$ such that
\begin{equation*}
r_\alpha^+ = 
T_{i_1} \cdots T_{i_{j-1}}(E_{i_j})
\end{equation*}
if $\alpha = s_{i_1} \cdots s_{i_{j-1}}(\alpha_{i_j}).$
For example, in type $A_2$ with $I = \{1,2\}$ and fixed reduced expression
$w_0 = s_1 s_2 s_1$, so that
$\alpha_1 \prec \alpha_1+\alpha_2 \prec \alpha_2$,
we have that
$r_{\alpha_1} = \theta_1,
r_{\alpha_1+\alpha_2} = \theta_1
\theta_2 - q \theta_2 \theta_1,
r_{\alpha_2} = \theta_2.$
 Also introduce the {\em dual root vector}
\begin{equation}\label{drv}
r_\alpha^* := (1-q_\alpha^{2}) r_\alpha.
\end{equation}
The normalization here ensures that $r_\alpha^*$ is invariant under $\b^*$, as can be checked directly using
(\ref{zstar}) and the formulae in \cite[$\S$37.2.4]{Lubook}.

A {\em Kostant partition} of $\alpha \in Q^+$
is a sequence
$\lambda = (\lambda_1,\dots,\lambda_l)$ of positive roots
such that $\lambda_1 \succeq\cdots\succeq \lambda_l$ and
$\lambda_1+\cdots+\lambda_l = \alpha$.
Denote the set of all Kostant partitions of $\alpha$ by
$\KP(\alpha)$.
For $\lambda = (\lambda_1,\dots,\lambda_l) \in
\KP(\alpha)$, let
$m_\beta(\lambda)$ denote the multiplicity of $\beta \in R^+$
as a part of $\lambda$.
Also set
$\lambda_k' := \lambda_{l+1-k}$ for $k=1,\dots,l$.
Then define a partial order $\preceq$ on $\KP(\alpha)$
so that $\lambda \prec \mu$ if and only if both of the following hold:
\begin{itemize}
\item 
$\lambda_1 = \mu_1,\dots,\lambda_{k-1} = \mu_{k-1}$ and
$\lambda_k \prec \mu_k$ for some $k$ such that $\lambda_k$ and $\mu_k$
are both defined;
\item
$\lambda_1' = \mu_1',\dots,\lambda_{k-1}' = \mu_{k-1}'$ and
$\lambda_k'\succ \mu_k'$ for some $k$ such that $\lambda_k'$ and $\mu_k'$
are both defined.
\end{itemize}
This ordering was introduced in \cite[$\S$3]{Mac}, and the 
following lemmas were noted already there
(at least implicitly).

\begin{Lemma}\label{l4}
For $\alpha \in R^+$ and $m \geq 1$, the 
Kostant partition $(\alpha^m)$ is the 
unique smallest element of $\KP(m\alpha)$.
\end{Lemma}

\begin{proof}
Suppose that $\lambda = (\lambda_1,\dots,\lambda_l) \in \KP(m\alpha)$
satisfies $\lambda \not\succ (\alpha^m)$. Then we either have that $\lambda_1 \preceq
\alpha$ or that $\lambda_1' \succeq \alpha$.
In the former case, $\lambda_k \preceq\alpha$ for all
$k$, while in the latter $\lambda_k \succeq\alpha$ for all $k$.
Either way,
applying Lemma~\ref{l1} to the equality $\lambda_1+\cdots+\lambda_l =
\alpha + \cdots + \alpha$ ($m$ times), we deduce that $\lambda =
(\alpha^m)$.
\end{proof}

\begin{Lemma}\label{l3}
For $\alpha \in R^+$,
suppose that $\lambda \in \KP(\alpha)$ is minimal such that $\lambda
\succ (\alpha)$.
Then $\lambda$ has two parts, i.e. $\lambda = (\beta,\gamma)$ for
positive roots
$\beta \succ \alpha \succ \gamma$.
\end{Lemma}

\begin{proof}
Suppose for a contradiction 
that $\lambda = (\lambda_1,\dots,\lambda_l)$ with $l \geq 3$.
By \cite[Lemma 2.1]{Mac},
we can partition the set $\{1,\dots,l\}$ as $J \sqcup K$ so that
$\beta := \sum_{j \in J} \lambda_j$ and $\gamma := \sum_{k \in K}
\lambda_k$
are positive roots with $\beta \succ \gamma$.
Each $\lambda_j$ is $\preceq \lambda_1$ hence by applying Lemma~\ref{l1}
to the equality $\sum_{j \in J} \lambda_j = \beta$ (taking $\alpha$
there to be the root $\lambda_1$)
we must have that $\beta \preceq \lambda_1$.
Moreover if it happens that $\beta = \lambda_1$
then $\gamma = \alpha - \beta = \lambda_2+\cdots+\lambda_l$
and we see similarly that $\gamma \preceq \lambda_2$.
As $l \geq 3$, this shows that either $\beta \prec \lambda_1$,
or $\beta = \lambda_1$ and $\gamma \prec \lambda_2$.
A similar argument shows that either $\gamma \succ \lambda_1'$,
or $\gamma = \lambda_1'$ and $\beta \succ \lambda_2'$.
Hence $(\beta,\gamma) \prec \lambda$.
But also we know that $(\beta,\gamma) \succ (\alpha)$ by Lemma~\ref{l4}.
So this
contradicts the minimality of $\lambda$.
\end{proof}

Let $\KP := \bigcup_{\alpha \in Q^+} \KP(\alpha)$.
For $\lambda = (\lambda_1,\dots,\lambda_l) \in \KP$,
we set
\begin{equation}\label{rlambda}
r_\lambda := r_{\lambda_1} \cdots r_{\lambda_l} / [\lambda]^!,
\qquad
r_\lambda^* := q^{s_\lambda} r_{\lambda_1}^* \cdots r_{\lambda_l}^*,
\end{equation}
where
\begin{equation*}
[\lambda]^! := \prod_{\beta\in R^+} [m_\beta(\lambda)]^!_\beta,\qquad
s_\lambda := 
\sum_{\beta \in R^+} \frac{d_\beta}{2} m_\beta(\lambda) (m_\beta(\lambda)-1).
\end{equation*}
The following key result is due to Lusztig;
it gives us the {\em PBW} and {\em dual PBW bases} for $\f$ arising
from the given convex ordering $\prec$.

\begin{Theorem}[Lusztig]\label{pbw}
The monomials 
$\left\{r_{\lambda}\:|\:\lambda \in \KP\right\}$
and 
$\left\{r^*_{\lambda}\:|\:\lambda \in \KP\right\}$
give a pair of dual bases for the free $\A$-modules 
$\f_\A$ and $\f_\A^*$, respectively.
\end{Theorem}

\begin{proof}
This follows from \cite[Corollary 41.1.4(b)]{Lubook},
\cite[Proposition 41.1.7]{Lubook},
\cite[Proposition 38.2.3]{Lubook}
and \cite[Lemma 1.2.8(b)]{Lubook}.
\end{proof}

\subsection{Proper standard modules}
The next results are taken from \cite[$\S$3]{Mac}, which generalizes \cite{KR2}.
Note that our conventions for the ordering $\preceq$ are consistent with
the notation in \cite{KR2}; the ordering in \cite{Mac} is the opposite
of the ordering here.
The modules $L(\alpha)$ in the following theorem are called {\em
cuspidal modules} in \cite{KR2, Mac}.

\begin{Theorem}\label{mac1}
For $\alpha \in R^+$ there is a unique (up to isomorphism) 
irreducible $H_\alpha$-module $L(\alpha)$ such that $[L(\alpha)] =
r_\alpha^*$.
Moreover,
for any $m \geq 1$, the module
$L(\alpha^m) := q_\alpha^{\frac{1}{2}m(m-1)}L(\alpha)^{\circ m}$
is irreducible.
\end{Theorem}

\begin{proof}
The existence of $L(\alpha)$ is the first part of \cite[Theorem
3.1]{Mac}.
The second part is \cite[Lemma 3.4]{Mac}.
\end{proof}

Suppose we are given $\alpha \in Q^+$ and $\lambda = (\lambda_1,\dots,\lambda_l)\in \KP(\alpha)$.
Define
the {\em proper standard module}
\begin{equation}
\bar\Delta(\lambda) :=
q^{s_\lambda}
L(\lambda_1)\circ\cdots\circ L(\lambda_l).
\end{equation}
It is immediate from Theorem~\ref{mac1} and the definition (\ref{rlambda})
that
$[\bar\Delta(\lambda)] = r_\lambda^*$,
i.e. proper standard modules categorify the dual PBW basis.
Let
\begin{equation*}
L(\lambda) := \bar\Delta(\lambda) / \rad \bar\Delta(\lambda).
\end{equation*}
The following theorem asserts in particular that this is a self-dual
irreducible module.

\begin{Theorem}\label{mac2}
For $\alpha \in Q^+$
the modules $\{L(\lambda)\:|\:\lambda \in \KP(\alpha)\}$
give a complete set of pairwise inequivalent $\circledast$-self-dual
irreducible $H_\alpha$-modules.
Moreover, for any $\lambda \in \KP(\alpha)$, all composition factors
of $\rad \bar\Delta(\lambda)$ 
are of the form $q^n L(\mu)$ for $\mu \prec \lambda$ and
$n \in \Z$.
\end{Theorem}

\begin{proof}
This is \cite[Theorem 3.1]{Mac}.
\end{proof}

We remark further that all irreducible modules of KLR algebras are absolutely
irreducible by \cite[Corollary 3.9]{KL}, i.e. their endomorphism algebras are
isomorphic to $\k$.
For $\lambda \in \KP$, we denote the projective cover of $L(\lambda)$ by
$P(\lambda)$.
Also introduce the {\em proper costandard module}
\begin{equation}
\bar\nabla(\lambda) := \bar\Delta(\lambda)^\circledast.
\end{equation}
It is immediate from Theorem~\ref{mac2} that $\bar\nabla(\lambda)$ has
socle $L(\lambda)^\circledast \cong L(\lambda)$.
Let us also record the key lemma (known by the first two authors as
``McNamara's Lemma'') 
at the heart of the proof of both of
the above theorems.

\begin{Lemma}\label{maclem}
Suppose we are given $\alpha \in R^+$
and $\beta,\gamma \in Q^+$ with $\beta+\gamma=\alpha$.
If $\Res^{\alpha}_{\beta,\gamma} L(\alpha) \neq 0$
then $\beta$ is a sum of positive roots $\preceq \alpha$ and $\gamma$ is a sum
of positive roots $\succeq \alpha$.
\end{Lemma}

\begin{proof}
This is \cite[Lemma 3.2]{Mac}.
\end{proof}

Here are some further consequences.

\begin{Lemma}\label{pl}
For $\alpha \in R^+$ and $m \geq 1$, 
we have that $$
\left[\Res^{m\alpha}_{\alpha,\dots,\alpha} L(\alpha^m)\right] = [m]^!_\alpha
\left[L(\alpha)^{\boxtimes m}\right].
$$
\end{Lemma}

\begin{proof}
It suffices to show that
$\big[\Res^{m\alpha}_{\alpha, (m-1)\alpha} L(\alpha^m)\big] = [m]_\alpha 
\big[L(\alpha)
\boxtimes L(\alpha^{m-1})\big]$
for $m \geq 2$.
For this we apply Theorem~\ref{mackey},
noting that $L(\alpha^m) = q_\alpha^{(m-1)}L(\alpha) \circ L(\alpha^{m-1})$. 
To understand the non-zero sections in
the Mackey filtration, we need to 
find all quadruples $(\beta_1,\beta_2,\gamma_1,\gamma_2)$
such that $\beta_1+\beta_2 = \beta_1+\gamma_1 = \alpha, 
\gamma_1+\gamma_2 = \beta_2+\gamma_2 = (m-1)\alpha$, 
$\Res^{\alpha}_{\beta_1,\beta_2} L(\alpha) \neq 0$
and $\Res^{(m-1) \alpha}_{\gamma_1,\gamma_2} L(\alpha^{m-1}) \neq 0$.
By Lemma~\ref{maclem} and Mackey, 
both $\beta_1$ and $\gamma_1$ are sums of
positive roots
$\preceq \alpha$.
Since $\beta_1+\gamma_1 = \alpha$, 
we deduce using 
Lemma~\ref{l1}
that either $\beta_1 = 0$ or $\gamma_1 = 0$.
This analysis shows that there are just two non-zero sections in the
Mackey filtration. The bottom non-zero section is obviously equal in the Grothendieck group
to $q_\alpha^{(m-1)} \left[L(\alpha) \circ
L(\alpha^{m-1})\right]$.
Also using some induction on $m$,
the top non-zero section 
contributes
$
q_\alpha^{-1} [m-1]_\alpha
\left[L(\alpha) \circ L(\alpha^{m-1})\right]$.
Finally observe that
$q_\alpha^{(m-1)} + q_\alpha^{-1} [m-1]_\alpha = [m]_\alpha.$
\end{proof}

\begin{Lemma}\label{sc}
Suppose we are given $\alpha \in Q^+$ and $\lambda =
(\lambda_1,\dots,\lambda_l)
\in \KP(\alpha)$.
Let $\Res^\alpha_\lambda$ denote the functor
$\Res^\alpha_{\lambda_1,\dots,\lambda_l}$.
Then
$$
\left[\Res^\alpha_{\lambda} \bar\Delta(\lambda)\right]
= [\lambda]^! \left[L(\lambda_1) \boxtimes\cdots\boxtimes L(\lambda_l)\right],
$$
Moreover for any $\mu \not\preceq \lambda$
we have that
$\Res^\alpha_\mu \bar\Delta(\lambda) = 0$.
\end{Lemma}

\begin{proof}
This follows from \cite[Lemma 3.3]{Mac}
and Lemma~\ref{pl}.
\end{proof}

\subsection{The contravariant form and Williamson's counterexample}\label{sscontra}
The results in this subsection are 
not needed in the remainder of
the article but are of independent interest. 
Throughout we fix $\alpha \in Q^+$ of height $n$.

\begin{Lemma}\label{sca}
For $\lambda \in \KP(\alpha)$
there is a unique (up to scalars) non-zero bilinear form $\langle\cdot,\cdot\rangle$ on
$\bar\Delta(\lambda)$
such that \begin{equation}\label{cprop}
\langle hv, v'\rangle = \langle v, T(h) v'\rangle
\end{equation} 
for all $v, v' \in
\bar\Delta(\lambda)$
and $h \in H_\alpha$.
The radical of this bilinear form is the unique maximal
submodule of $\bar\Delta(\lambda)$.
Moreover, for $\bi, \bi' \in \W_\alpha$ and $m, m' \in \Z$,
we have that $\langle 1_\bi
\bar\Delta(\lambda)_m,
1_{\bi'} \bar\Delta(\lambda)_{m'}\rangle = 0$ 
unless $\bi = \bi'$ and $m+m' = 0$.
\end{Lemma}

\begin{proof}
There is an isomorphism 
from $\HOM_{H_\alpha}(\bar\Delta(\lambda),
\bar\nabla(\lambda))$
to the space of bilinear forms on
$\bar\Delta(\lambda)$
with the
property (\ref{cprop}), mapping
$f:\bar\Delta(\lambda) \rightarrow \bar\nabla(\lambda)$
to the form $\langle v, v'\rangle := f(v)(v')$.
Moreover
$\HOM_{H_\alpha}(\bar\Delta(\lambda), \bar\nabla(\lambda))$
is one-dimensional, indeed, it is spanned by a homogeneous
homomorphism that sends the head of $\bar\Delta(\lambda)$ 
onto the socle of $\bar\nabla(\lambda)$.
The existence and uniqueness of the contravariant form follow at
once. The last two parts of the lemma are immediate consequences too.
\end{proof}

The next lemma is useful when trying to compute the contravariant 
form on $\bar\Delta(\lambda)$ in practice.
Recall for $\lambda = (\lambda_1,\dots,\lambda_l) \in \KP(\alpha)$
that $$
\bar\Delta(\lambda) = 
q^{s_\lambda}
H_\alpha 1_\lambda 
\otimes_{H_\lambda} 
(L(\lambda_1) \boxtimes\cdots\boxtimes L(\lambda_l)),$$ 
where
$H_\lambda := H_{\lambda_1} \otimes\cdots\otimes H_{\lambda_l}$ with
identity $1_\lambda$. Letting $S_\lambda$ denote the parabolic
subgroup
$S_{\height(\lambda_1)}\times\cdots\times S_{\height(\lambda_l)}$
of $S_n$ and 
$D_\lambda$ be the set of minimal length $S_n / S_\lambda$-coset representatives,
any element of 
$\bar\Delta(\lambda)$ is a sum of vectors of the form 
$\tau_w 1_\lambda \otimes (v_1 \otimes\cdots\otimes v_l)$
for $w \in D_\lambda$ and $v_i \in L(\lambda_i)$.

\begin{Lemma}\label{formua}
For $\lambda = (\lambda_1,\dots,\lambda_l) \in \KP(\alpha)$,
let $x=x^{-1}$ be the longest element of $D_\lambda$ such that
$\tau_x 1_\lambda = 1_\lambda \tau_x$, and let $y=y^{-1}$ be the longest
element of $S_l$ such that $\lambda_{y(i)} = \lambda_i$ for each $i=1,\dots,l$.
The contravariant form $\langle\cdot,\cdot\rangle$ on
$\bar\Delta(\lambda)$ satisfies
$$
\langle \tau_w 1_\lambda \otimes (v_1 \otimes \cdots \otimes v_l),
\tau_{w'}1_\lambda \otimes (v_1' \otimes \cdots \otimes v_l')\rangle
= 
\delta_{w^{-1} w',x} \langle v_{y(1)}, v_{1}'\rangle_{_{\!1}} \cdots \langle v_{y(l)}, v_{l}'\rangle_{_{\!l}}
$$
for all $w,w' \in D_\lambda$ and $v_i, v_i' \in L(\lambda_i)$,
where $\langle\cdot,\cdot\rangle_{_{\!i}}$ is some choice of 
(non-degenerate) contravariant
form on each
$L(\lambda_i)$.
\end{Lemma}

\begin{proof}
Let $L' := L(\lambda_1) \boxtimes\cdots \boxtimes L(\lambda_l)$ for
short, and denote the product of the forms $\langle\cdot,\cdot\rangle_{_{\!i}}$
for $i=1,\dots,l$ by $\langle\cdot,\cdot\rangle'$, which is a non-degenerate
form on $L'$.
Recall from the proof of Lemma~\ref{sca} that the contravariant
form on $\bar\Delta(\lambda)$ is defined
from $\langle v, v' \rangle := f(v)(v')$ where
$f:\bar\Delta(\lambda) \rightarrow \bar\nabla(\lambda)$ is a non-zero
homomorphism.
We can identify
$\bar\nabla(\lambda) = \HOM_{\k}(\bar\Delta(\lambda), \k)$ 
with the coinduced module
$
q^{-s_\lambda} \HOM_{H_\lambda}(1_\lambda H_\alpha, L')
$ so that
$\theta:1_\lambda H_\alpha\rightarrow L'$
is identified
with the functional
$\bar\Delta(\lambda)\mapsto \k, h 1_\lambda \otimes v \mapsto
\langle \theta(1_\lambda T(h)), v\rangle'$.
Then by adjointness of restriction and coinduction we get a canonical isomorphism
\begin{equation}\label{theiso}
\HOM_{H_\alpha}(\bar \Delta(\lambda), \bar\nabla(\lambda))
\cong
\HOM_{H_\lambda}(\Res^\alpha_\lambda \bar\Delta(\lambda),
q^{-s_\lambda}L').
\end{equation}
Now we observe as in the proof of Lemma~\ref{sc} 
that the top non-zero section in the Mackey filtration of
$\Res^\alpha_\lambda \bar \Delta(\lambda)$ is isomorphic
to $
q^{-s_\lambda} L'$.
In this way, we obtain an explicit homomorphism
$\bar f:\Res^\alpha_\lambda \bar\Delta(\lambda) \rightarrow
q^{-s_\lambda} L'$
such that $$
\bar f (1_\lambda \tau_w 1_\lambda \otimes (v_1 \otimes
\cdots \otimes v_l)) = \delta_{w,x} v_{y(1)} \otimes\cdots\otimes
v_{y(l)}
$$
for $w \in D_\lambda$ and $v_i \in L(\lambda_i)$.
Then choose
$f:\bar\Delta(\lambda) \rightarrow \bar\nabla(\lambda)$
so that it corresponds to $\bar f$ under (\ref{theiso}).
This means that
$$
\langle h 1_\lambda \otimes v, h' 1_\lambda \otimes v'
\rangle = 
f(h 1_\lambda \otimes v)(h'1_\lambda \otimes v') = 
\langle \bar f(1_\lambda T(h')h
1_\lambda \otimes v), v' \rangle',
$$
using our earlier identification of $\bar{\nabla}(\lambda)$ with
$q^{-s_\lambda} 
\HOM_{H_\lambda}(1_\lambda H_\alpha, L')$ for the second equality.
The lemma follows from the last two displayed formulae.
\end{proof}

Using this we can prove that the contravariant form is symmetric.

\begin{Lemma}
For each $\lambda \in \KP(\alpha)$ the form $\langle\cdot,\cdot\rangle$ on 
$\bar\Delta(\lambda)$ is symmetric.
\end{Lemma}

\begin{proof}
We proceed by induction on $\height(\alpha)$.
The base case 
$\alpha = \alpha_i$ 
is trivial as $L(\alpha_i)$ is
one-dimensional.
For the induction step, 
Lemma~\ref{formua} reduces us to the case
that 
$\lambda = (\alpha)$ for some non-simple $\alpha \in R^+$.
Pick $i \in I$ and 
$m \geq 1$ 
such that 
$\Res^{\alpha}_{\alpha - m \alpha_i, m \alpha_i} L(\alpha) \neq 0$,
and either $\alpha-(m+1)\alpha_i \notin Q^+$ or 
$\Res^{\alpha}_{\alpha -(m+1)\alpha_i, (m+1) \alpha_i} L(\alpha) = 0$.
By general theory,
$\Res^{\alpha}_{\alpha - m \alpha_i, m \alpha_i} L(\alpha) \cong
L(\mu) \boxtimes L(\alpha_i^m)$ for some $\mu \in
\KP(\alpha-m\alpha_i)$; see e.g. \cite[Lemma 3.7]{KL}.
By Lemma~\ref{sca}, 
the restriction of the form $\langle\cdot,\cdot\rangle$ to this
copy of
$L(\mu) \boxtimes L(\alpha_i^m)$ is a non-degenerate contravariant
form in the obvious sense, 
hence it is the product of contravariant forms
on $L(\mu)$ and $L(\alpha_i^m)$. Both of these are already known to be
symmetric by induction.
This shows the form is symmetric on restriction to some
non-zero word space of $L(\alpha)$. Since $L(\alpha)$ is
irreducible, this implies it is symmetric on the entire module.
\end{proof}

Now we recall a conjecture from \cite[Conjecture 7.3]{KR2} asserting
in finite type that
the 
formal character of an irreducible $H_\alpha$-module $L(\lambda)$ does not depend
on the 
characteristic $p$ of the ground field $\k$.
Using geometric techniques, 
Williamson \cite{W} has recently shown
that this is 
false, and the question of finding a satisfactory bound
on $p$ remains open.
The smallest counterexample found by Williamson (motivated
by a counterexample constructed by Kashiwara and Saito in \cite{KS}
to a much older but geometrically related conjecture
of Kazhdan and Lusztig) is as follows. 

\begin{Example}\label{willcex}\rm
Assume we are in type A$_5$.
Index the simple roots in the usual way
$1 \text{---} 2 \text{---} 3 \text{---} 4 \text{---} 5$
and choose the signs $\eps_{i,j}$ in the definition of the KLR algebra
so that $\eps_{1,2} = \eps_{2,3} =
\eps_{3,4} = \eps_{4,5} = +$.
The positive roots 
are $\alpha_{i,j} :=
\alpha_i+\alpha_{i+1}+\cdots+\alpha_j$
for $1 \leq i \leq j \leq 5$. 
Let $\prec$ be the convex ordering 
defined so that $\alpha_{i,j} \prec \alpha_{k,l}$ if either $i < k$, or
$i=k$ and $j < l$; cf. Example~\ref{eg1} below.
For this choice the module $L(\alpha_{i,j})$ is trivial to construct explicitly:
it is the one-dimensional module spanned
by a degree zero
vector belonging to the $i (i+1) \cdots j$-word space.
Let 
\begin{align*}
\lambda&:=(\al_{4,5},\al_{4,5},\al_3,\al_3,\al_{2,4},\al_{2,4},
\al_{1,2},\al_{1,2}),
\\
\bi &:=
4534234523123412.
\end{align*}
We claim that $1_\bi L(\lambda)_0$ has dimension $2$ if
$\operatorname{char} \k = 2$ and dimension 3 in all other characteristics.
To see this we compute the rank of the contravariant form
on $1_\bi \bar\Delta(\lambda)_0$.
Let $v$ span
$L(\al_{4,5})^{\boxtimes 2} \boxtimes L(\al_3)^{\boxtimes 2}
\boxtimes L(\al_{2,4})^{\boxtimes 2} \boxtimes
L(\al_{1,2})^{\boxtimes 2}$.
Adopting all the notation from Lemma~\ref{formua}, 
the vectors $\{\tau_w 1_\lambda \otimes v\:|\:w \in D_\lambda\}$ give
a basis for $\bar\Delta(\lambda)$ with $1_\lambda \otimes v$ of degree
$4$ and $\tau_x 1_\lambda \otimes v$ of degree $-4$.
We normalize the contravariant form on $\bar\Delta(\lambda)$ so that
$\langle 1_\lambda \otimes v, \tau_x 1_\lambda \otimes v \rangle = -1$.
Let 
\begin{align*}
a &:= \tau_3 \tau_7 \tau_6 \tau_5 \tau_4 \tau_9 \tau_8 \tau_7 \tau_6
\tau_{12} \tau_{11} \tau_{13} \tau_{12},\\
b &:= \tau_3 \tau_7 \tau_6 \tau_5 \tau_4 \tau_{12} \tau_{11} \tau_{10}
\tau_9
\tau_8 \tau_7 \tau_6 \tau_{13} \tau_{12},\\
c_1 &:= \tau_2 \tau_1 \tau_3 \tau_2,\\
c_2 &:= \tau_5,\\
c_3 &:= \tau_9 \tau_8 \tau_7 \tau_{10} \tau_9 \tau_8 \tau_{11}
\tau_{10} \tau_9,\\
c_4 &:= \tau_{14} \tau_{13} \tau_{15} \tau_{14}.
\end{align*}
We have that 
$c_1c_2c_3c_4 1_\lambda \otimes v = \tau_x 1_\lambda \otimes v$ (as
should be clear on drawing the appropriate diagrams), and
$1_\bi \bar\Delta(\lambda)_0$ is 5-dimensional with basis
$$
a c_1c_2c_3c_4 1_\lambda \otimes v, b c_2 c_3 c_4 1_\lambda \otimes v,
b c_1 c_3 c_4 1_\lambda \otimes v,
b c_1 c_2 c_4 1_\lambda \otimes v,
b c_1 c_2 c_3 1_\lambda \otimes v.
$$
Using Lemma~\ref{formua} and making some explicit but lengthy straightening calculations,
one can then check that the Gram matrix of the contravariant form on
$1_\bi \bar\Delta(\lambda)_0$ with respect to this basis is 
$$
\left(
\begin{array}{rrrrr}
0&1&1&1&1\\
1&0&0&0&1\\
1&0&0&0&1\\
1&0&0&0&1\\
1&1&1&1&0
\end{array}
\right).
$$
It remains to compute the rank of this matrix.
\end{Example}

\section{Standard modules}\label{secsm}

We continue to work with a fixed choice of convex ordering $\prec$ on $R^+$.
In this section we will give an elementary definition first of 
{\em root modules} $\Delta(\alpha)$ categorifying the root elements
$r_\alpha$, then of {\em standard modules}
$\Delta(\lambda)$ 
categorifying 
the PBW basis elements $r_\lambda$.
We show that standard modules satisfy homological
properties analogous to the standard modules of a quasi-hereditary
algebra, hence in simply-laced types they
are isomorphic to the modules $\widetilde{E}_b$ constructed using Saito
reflection functors in
\cite[$\S$4]{Kato}.

\subsection{Root modules}
We fix a positive root $\alpha$ throughout the subsection.
All the good homological properties of KLR algebras proved in this
article stem from the following key observation made in \cite[$\S$4]{Mac}.

\begin{Theorem}\label{source}
Recall that $L(\alpha)$ is the cuspidal module associated to the
positive root $\alpha$.
For $d \geq 2$ we have that
$$
\EXT^1_{H_\alpha}(L(\alpha), L(\alpha)) \cong
  q_\alpha^{-2} \k,\qquad
\EXT^d_{H_\alpha}(L(\alpha), L(\alpha)) = 0.
$$
\end{Theorem}

\begin{proof}
As explained in the second half of the proof of \cite[Proposition 4.5]{Mac},
this is a consequence of the finiteness of the
global dimension of $H_\alpha$. The latter property is established in
simply-laced types for $\k$ of characteristic zero in \cite[Corollary 2.9]{Kato}
and in non-simply-laced types for $\k$ of arbitrary characteristic in
\cite[Theorem 4.6]{Mac}. This leaves us with simply-laced types in positive
characteristic. 
These can be treated by the same argument used in the
non-simply-laced case in \cite{Mac}. However some
additional computations are needed which we postpone
to the appendix; see Corollary~\ref{fgd}.
\end{proof}

Our first application is to the construction of root modules.

\begin{Lemma}\label{fe}
For $n \geq 0$, 
there exist unique (up to isomorphism)
indecomposable $H_\alpha$-modules $\Delta_n(\alpha)$ 
with $\Delta_0(\alpha) = 0$ such that
there are short exact sequences
\begin{align}\label{ses1}
0 \longrightarrow
q_\alpha^{2 (n-1)} L(\alpha)
\stackrel{i_n}{\longrightarrow}
&\Delta_{n}(\alpha) \stackrel{p_{n}}{\longrightarrow} \Delta_{n-1}(\alpha)
\longrightarrow 0,\\
0 \longrightarrow
q_\alpha^{2} \Delta_{n-1}(\alpha)
\stackrel{j_n}{\longrightarrow}
&\Delta_{n}(\alpha) \stackrel{q_n}{\longrightarrow} L(\alpha)
\longrightarrow 0.\label{ses2}
\end{align}
Moreover the following hold for all $n \geq 1$:
\begin{itemize}
\item[(1)] $\displaystyle[\Delta_n(\alpha)] = \frac{1-q_\alpha^{2n}}{1-q_\alpha^2}
  [L(\alpha)]$;
\item[(2)] $\Delta_n(\alpha)$ is a cyclic module with irreducible head isomorphic to
  $L(\alpha)$ and socle isomorphic to $q_\alpha^{2(n-1)}
  L(\alpha)$;
\item[(3)]
the map $i_n^*: \EXT^1_{H_\alpha}(\Delta_n(\alpha), L(\alpha)) \rightarrow
\EXT^1_{H_\alpha}(q_\alpha^{2(n-1)} L(\alpha), L(\alpha))$ induced by
$i_{n}$ is an isomorphism,
hence 
$\EXT^1_{H_\alpha}(\Delta_n(\alpha), L(\alpha)) \cong
  q_\alpha^{-2n}\k$;
\item[(4)]
$\EXT^d_{H_\alpha}(\Delta_n(\alpha), L(\alpha)) = 0$ for all $d \geq 2$.
\end{itemize}
\end{Lemma}

\begin{proof}
Let $\Delta_0(\alpha) := 0$ and $\Delta_1(\alpha) := L(\alpha)$.
The properties (1)--(4) hold when $n=1$ by Theorem~\ref{source}.
Now suppose that $n \geq 2$ and that we have constructed
$\Delta_{n-1}(\alpha)$ satisfying the properties (1)--(4).
In particular by (3) we have that
$\EXT^1_{H_\alpha}(\Delta_{n-1}(\alpha), L(\alpha)) \cong
q_\alpha^{-2(n-1)}\k$,
so there exists a unique (up to isomorphism) module
$\Delta_{n}(\alpha)$ fitting into a non-split short exact sequence of
the form (\ref{ses1}).
We must prove that $\Delta_n(\alpha)$ also satisfies the properties
(1)--(4) (hence by (2) it is indecomposable) and that there exists a short
exact sequence of the form (\ref{ses2}). Property (1) is immediate
from (\ref{ses1}) and the induction hypothesis.

Applying $\HOM_{H_\alpha}(-, L(\alpha))$ to (\ref{ses1}) using
Theorem~\ref{source} and the induction hypothesis, we deduce that (4) holds. Moreover there is an exact sequence
\begin{multline*}
0 \longrightarrow \k \longrightarrow
\HOM_{H_\alpha}(\Delta_{n}(\alpha), L(\alpha))
\stackrel{f}{\longrightarrow}
q_\alpha^{-2(n-1)} \k 
\longrightarrow
q_\alpha^{-2(n-1)} \k\\
\longrightarrow
\EXT^1_{H_\alpha}(\Delta_{n}(\alpha), L(\alpha))
\stackrel{i_n^*}{\longrightarrow} \EXT^1_{H_\alpha}(q_\alpha^{2(n-1)} L(\alpha), L(\alpha))
\longrightarrow 0.
\end{multline*}
The map $f$ is zero, for otherwise there exists a non-zero
homogeneous homomorphism
$\Delta_n(\alpha) \rightarrow q_\alpha^{2(n-1)} L(\alpha)$, which is a contradiction as (\ref{ses1}) is non-split.
Hence $\HOM_{H_\alpha}(\Delta_n(\alpha), L(\alpha)) \cong \k$.
Since all composition factors of $\Delta_{n}(\alpha)$ are of the form
$L(\alpha)$ (up to shift) by (1), this shows $\Delta_{n}(\alpha)$
has
irreducible head $L(\alpha)$.
Then we see that $\iota_n^*$ is an isomorphism proving the first half of (3); the
second half of (3) is immediate from Theorem~\ref{source}.

To complete the proof of (2), it remains to 
compute the socle of
$\Delta_{n}(\alpha)$.
The image of the non-trivial extension represented by
(\ref{ses1})
under the map $i_{n-1}^*$ is the extension represented by
\begin{equation*}
0 \longrightarrow q_\alpha^{2(n-1)} L(\alpha)
\stackrel{i_n}{\longrightarrow} V \stackrel{p_n}{\longrightarrow} q_\alpha^{2(n-2)}
L(\alpha)\longrightarrow 0
\end{equation*}
where $V := p_n^{-1}(\im i_{n-1})$.
Since $i^*_{n-1}$ is an isomorphism by the induction hypothesis, this
extension is non-trivial, hence $q_\alpha^{2(n-2)} L(\alpha)$ does not
appear in the socle of $V$.
If the socle of $\Delta_{n}(\alpha)$ is not irreducible then, in view of
(\ref{ses1}) and the induction hypothesis, its socle must be isomorphic to
$q_\alpha^{2(n-1)} L(\alpha) \oplus
q_\alpha^{2(n-2)} L(\alpha)$,
contradicting the previous sentence.
Hence $\soc \Delta_{n}(\alpha) \cong q_\alpha^{2(n-1)} L(\alpha)$.

Finally we prove the existence of the short exact sequence (\ref{ses2}).
This is just the same as (\ref{ses1}) if $n=2$, so assume $n \geq 3$.
Applying $\hom_{H_\alpha}(q_\alpha^2 \Delta_{n-1}(\alpha), -)$ to
(\ref{ses1}) and using induction, we deduce that
$\hom_{H_\alpha}(q_\alpha^2 \Delta_{n-1}(\alpha), \Delta_{n}(\alpha))
\cong \k.$ 
Let $j_n: q_\alpha^2 \Delta_{n-1}(\alpha) \rightarrow
\Delta_{n}(\alpha)$ be any non-zero homogeneous homomorphism.
It must be injective since it is injective on $\soc q_\alpha^2
\Delta_n(\alpha)$ thanks to (1)--(2). 
Moreover $\operatorname{coker} j_n \cong L(\alpha)$ by (1).
\end{proof}

This shows that there is an inverse system
$\Delta_0(\alpha)\stackrel{p_1}{\twoheadleftarrow} \Delta_1(\alpha) \stackrel{p_2}{\twoheadleftarrow} \Delta_2(\alpha) \stackrel{p_3}{\twoheadleftarrow} \cdots$.
Define the {\em root module}
\begin{equation*}
\Delta(\alpha) := 
\varprojlim \Delta_n(\alpha).
\end{equation*}
Recalling (\ref{drv}), part (1) of the following theorem shows that
$\Delta(\alpha)$ categorifies the root vector $r_\alpha$.

\begin{Theorem}\label{rm}
There is
a short exact sequence
\begin{equation}\label{ses3}
0 \longrightarrow q_\alpha^2 \Delta(\alpha)
\stackrel{j}{\longrightarrow}
\Delta(\alpha) \longrightarrow L(\alpha) \longrightarrow 0.
\end{equation}
Moreover:
\begin{itemize}
\item[(1)]
$\Delta(\alpha)$ is a cyclic module with $[\Delta(\alpha)] = [L(\alpha)] /
(1-q_\alpha^2)$;
\item[(2)] $\Delta(\alpha)$ has irreducible head isomorphic to
  $L(\alpha)$;
\item[(3)] 
we have that 
$\EXT^d_{H_\alpha}(\Delta(\alpha), V) =0$ 
for $d \geq 1$ and
any finitely generated $H_\alpha$-module $V$ with
  all irreducible subquotients isomorphic to $L(\alpha)$ (up to
  degree shift);
\item[(4)]
$\END_{H_\alpha}(\Delta(\alpha)) \cong \k[x]$ for $x$ in
degree $2d_\alpha$.
\end{itemize}
\end{Theorem}

\begin{proof}
By definition of inverse limit, there are canonical maps $\pi_n:\Delta(\alpha)
\twoheadrightarrow \Delta_n(\alpha)$ such that $p_n \circ \pi_n =
\pi_{n-1}$ for each $n \geq 1$.
For $i \in \Z$ the dimension of the graded component $\Delta_n(\alpha)_i$ is
bounded independent of $n$, so the inverse system of vector spaces
$\Delta_0(\alpha)_i \twoheadleftarrow \Delta_1(\alpha)_i
\twoheadleftarrow\cdots$ stabilizes after finitely many terms. 
Since $\Delta(\alpha)_i = \varprojlim \Delta_n(\alpha)_i$,
we deduce that $\bigcap_{n \geq 0} \ker \pi_n = 0$.
Combined also with Lemma~\ref{fe}(1), it follows that
$[\Delta(\alpha)] = [L(\alpha)] / (1-q_\alpha^2)$.
Also $\Delta(\alpha)$ is cyclic because each $\Delta_n(\alpha)$ is
cyclic by Lemma~\ref{fe}(2).
This proves (1).

In view of (1), the head of $\Delta(\alpha)$ is 
isomorphic to a finite direct sum of copies of $q_\alpha^{2n} L(\alpha)$
for $n \geq 0$.
To deduce (2) we must show
$\hom_{H_\alpha}(\Delta(\alpha), L(\alpha)) \cong \k$
and $\hom_{H_\alpha}(\Delta(\alpha), q_\alpha^{2n} L(\alpha)) = 0$
for $n > 0$.
Suppose we are given a non-zero homogeneous
homomorphism 
$f:\Delta(\alpha) \rightarrow q_\alpha^{2n} L(\alpha)$
for some $n \geq 0$.
By (1), all irreducible subquotients of $\ker \pi_{n+1}$
are of the form $q_\alpha^{2m} L(\alpha)$ for $m > n$, 
hence $f$ factors through the quotient to induce
$\bar f:\Delta_{n+1}(\alpha) \rightarrow q_\alpha^{2n} L(\alpha)$.
It remains to apply Lemma~\ref{fe}(2).

Next we construct the short exact sequence (\ref{ses3}).
We claim that we can choose the injective homomorphisms $j_n$ in 
(\ref{ses2}) so that the following diagrams
commute for all $n \geq 1$:
$$
\begin{CD}
q_\alpha^2 \Delta_{n-1}(\alpha)&@<p_n<<&q_\alpha^2 \Delta_n(\alpha)\\
@V j_nVV&&@VV j_{n+1}V\\
\Delta_n(\alpha)&@<p_{n+1}<<&\Delta_{n+1}(\alpha).
\end{CD}
$$
This is automatic if $n=1$ (both compositions are zero), so assume $n
\geq 2$
and that we are given $j_n$.
The map $j_n \circ p_n$ is obviously non-zero, as is $p_{n+1} \circ
j_{n+1}$ because
$\soc \Delta_n(\alpha) \cong q_\alpha^{2(n-1)} L(\alpha)$.
Also by Lemma~\ref{fe}(1)--(2),
$\hom_{H_\alpha}(q_\alpha^2 \Delta_n(\alpha), \Delta_n(\alpha))$ is
one-dimensional.
Hence $j_{n+1}$ can be rescaled by a non-zero scalar if necessary to
ensure that the above
diagram commutes.
This proves the claim. Hence we get induced an injective homogeneous homomorphism
$j:q_\alpha^2 \Delta(\alpha) \rightarrow \Delta(\alpha)$ such that
$\pi_n \circ j = j_n \circ \pi_{n-1}$ for each $n \geq 1$.
By (1) we have that $\operatorname{coker} j \cong L(\alpha)$ and
(\ref{ses3}) follows.

To prove (3), let $V$ be as in its statement and take some fixed $d
\geq 1$.
By Theorem~\ref{source} and Lemma~\ref{mittagleffler}, 
we have that $\EXT^{d+1}_{H_\alpha}(L(\alpha), V) =0$.
Applying $\HOM_{H_\alpha}(-,V)$ to the short exact sequence
(\ref{ses3}), we
deduce the existence of a homogeneous surjection
$$
\EXT^d_{H_\alpha}(\Delta(\alpha), V) \twoheadrightarrow
q_\alpha^{-2} \EXT^d_{H_\alpha}(\Delta(\alpha), V).
$$
As $V$ and $\Delta(\alpha)$ are both finitely generated,
$\EXT^d_{H_\alpha}(\Delta(\alpha), V)$ is bounded below.
The only way to avoid a contradiction is if
$\EXT^d_{H_\alpha}(\Delta(\alpha), V)$ is zero.

Finally we prove (4). The map $j$ in (\ref{ses3}) can be viewed as an
injective endomorphism $x \in \END_{H_\alpha}(\Delta(\alpha))_{2d_\alpha}$.
It generates a free polynomial subalgebra $\k[x]$ of $\END_{H_\alpha}(\Delta(\alpha))$.
But we also know that $\DIM \END_{H_\alpha}(\Delta(\alpha))\leq 
\DIM \k[x]$ by (1)--(2). So we must have equality here.
\end{proof}

\begin{Corollary}\label{Equiv}
The functor 
$\HOM_{H_\alpha}(\Delta(\alpha), -)$ defines an equivalence
from
the category of finitely generated graded
$H_\alpha$-modules all of whose irreducible subquotients are isomorphic to
$L(\alpha)$ (up to a shift) to the category of finitely generated
graded $\k[x]$-modules (viewing $\k[x]$ as a graded algebra with $x$ in degree $2 d_\alpha$).
\end{Corollary}

\begin{proof}
The given category of $H_\alpha$-modules is an abelian category.
By Theorem~\ref{rm}(3), $\Delta(\alpha)$ is a projective
generator.
Hence 
the category is equivalent via the given functor to the category of
finitely generated graded modules over
$\END_{H_\alpha}(\Delta(\alpha))$, which is
$\k[x]$ by Theorem~\ref{rm}(4).
\end{proof}

\begin{Corollary}\label{class}
Any finitely generated (graded) $H_\alpha$-module with all irreducible
subquotients
isomorphic to $L(\alpha)$ (up to a shift)
is a finite direct sum of degree-shifted copies of the
indecomposable modules
$\Delta_n(\alpha)\:(n \geq 1)$ and $\Delta(\alpha)$.
\end{Corollary}

\begin{proof}
The algebra $\k[x]$ is a principal ideal domain so its finitely generated
indecomposable
graded modules are isomorphic up to degree-shift to $\k[x] / (x^n)$ or
$\k[x]$.
Under the
equivalence from Corollary~\ref{Equiv}, these correspond to
$\Delta_n(\alpha)$ and $\Delta(\alpha)$, 
respectively.
\end{proof}

\subsection{Divided powers}\label{ssdp}
Throughout the subsection we fix $\alpha \in R^+$ of height $n$.
We are going to compute the endomorphism algebra of
$\Delta(\alpha)^{\circ m}$.
Choose a non-zero homogeneous vector $v_\alpha$ of minimal degree in
$\Delta(\alpha)$, so that 
$\Delta(\alpha)$ is generated as an $H_\alpha$-module by $v_\alpha$.
Also pick a non-zero endomorphism $x
\in \END_{H_\alpha}(\Delta(\alpha))_{2d_\alpha}$.
By Theorem~\ref{rm}(4) we have that
$\END_{H_\alpha}(\Delta(\alpha)) = \k[x]$,
so that $x$ is unique up to a scalar.

From the endomorphism $x$,  we obtain commuting 
endomorphisms $\dx_1,\dots,\dx_m
\in \END_{H_{m\alpha}}(\Delta(\alpha)^{\circ m})_{2d_\alpha}$
with
$\dx_i := \id^{\circ(i-1)} \circ x \circ \id^{\circ (m-i)}$.
Similarly, the endomorphism $\tau$ from the following lemma 
yields $\dtau_1,\dots,\dtau_{m-1} \in \END_{H_{m
    \alpha}}(\Delta(\alpha)^{\circ m})_{-2d_\alpha}$
with $\dtau_i := \id^{\circ (i-1)} \circ\tau\circ \id^{\circ(m-i-1)}$.
(It is a bit confusing here that we are using the same notation $x_i$
and $\tau_i$ for these endomorphisms as we use for the elements of the
KLR algebra, but hopefully it is clear from the context which
we mean.)

\begin{Lemma}\label{dpl1}
Let $w \in S_{2n}$ be the permutation 
mapping $(1,\dots,n, n+1,\dots,2n)$ to $(n+1,\dots,2n,1,\dots,n)$.
There is a unique 
$H_{2\alpha}$-module homomorphism
$$
\tau:
\Delta(\alpha) \circ \Delta(\alpha) \rightarrow \Delta(\alpha) \circ
\Delta(\alpha)
$$
of degree $-2d_\alpha$
such that $\tau(1_{\alpha,\alpha} \otimes (v_\alpha \otimes v_\alpha))
= \tau_w 1_{\alpha,\alpha} \otimes (v_\alpha \otimes v_\alpha)$.
\end{Lemma}

\begin{proof}
We apply the Mackey theorem to 
$\Res^{2\alpha}_{\alpha,\alpha} \Delta(\alpha) \circ \Delta(\alpha)$.
By exactly the same argument as in the proof of Lemma~\ref{pl}, there
are just two non-zero sections in the Mackey filtration, corresponding
to the double coset representatives $1$ and $w$.
We deduce that there is a short exact sequence
$$
0 \longrightarrow \Delta(\alpha) \boxtimes \Delta(\alpha)
\stackrel{f}{\longrightarrow} \Res^{2\alpha}_{\alpha,\alpha} \Delta(\alpha) \circ
\Delta(\alpha)
\stackrel{g}{\longrightarrow} q_\alpha^{-2}\Delta(\alpha) \boxtimes \Delta(\alpha)
\longrightarrow 0
$$
such that $f(v_\alpha \otimes v_\alpha) = 1_{\alpha,\alpha} \otimes (v_\alpha \otimes
v_\alpha)$ and
$g(\tau_w 1_{\alpha,\alpha} \otimes (v_\alpha \otimes v_\alpha)) = v_\alpha \otimes v_\alpha$.
By Theorem~\ref{rm}(3) we have that
$$
\EXT^1_{H_{\alpha,\alpha}}(\Delta(\alpha)\boxtimes\Delta(\alpha),
\Delta(\alpha)\boxtimes\Delta(\alpha)) = 0.
$$
So 
the short exact sequence splits.
Let $\bar g:q_\alpha^{-2}\Delta(\alpha) \boxtimes \Delta(\alpha)
\rightarrow \Res^{2\alpha}_{\alpha,\alpha} \Delta(\alpha) \circ
\Delta(\alpha)$
be the unique splitting.
Since $\im f = 1_{\alpha,\alpha} \otimes (\Delta(\alpha) \boxtimes
\Delta(\alpha))$
contains no non-zero vectors of degree $2 \deg(v_\alpha) - 2 d_\alpha$,
we must have that
$\bar g(v_\alpha \otimes v_\alpha)  = \tau_w 1_{\alpha,\alpha} \otimes
(v_\alpha \otimes v_\alpha)$.
Applying Frobenius reciprocity, $\bar g$ induces a map $\tau$
as in the statement of the lemma. 
\end{proof}

\begin{Lemma}\label{dpl2}
The endomorphisms $\dtau_1,\dots,\dtau_{m-1}
\in \END_{H_{m\alpha}}(\Delta(\alpha)^{\circ m})$
square to zero and satisfy the usual type $\operatorname{A}$ braid relations.
\end{Lemma}

\begin{proof}
For the quadratic relation, 
we need to show in the setup of Lemma~\ref{dpl1} that
$\tau^2 = 0$.
As a vector space, the Mackey theorem analysis from the proof of that lemma
tells us that
$$
1_{\alpha,\alpha}(\Delta(\alpha) \circ \Delta(\alpha))= 
1_{\alpha,\alpha} \otimes (\Delta(\alpha) \boxtimes \Delta(\alpha))\oplus
\tau_w 1_{\alpha,\alpha} \otimes (\Delta(\alpha) \boxtimes \Delta(\alpha)).
$$
Thus the vector $\tau_w 1_{\alpha,\alpha} \otimes (v_\alpha \otimes v_\alpha)$ is of
minimal degree in $1_{\alpha,\alpha}(\Delta(\alpha) \circ
\Delta(\alpha))$, namely, $2\deg(v_\alpha)-2d_\alpha$.
The vector $\tau_w^2 1_{\alpha,\alpha}\otimes (v_\alpha \otimes v_\alpha)$ is of strictly
smaller degree $2\deg(v_\alpha)-4d_\alpha$, hence it must be zero.
This shows that $\tau^2$ sends a generator of $\Delta(\alpha)
\circ \Delta(\alpha)$ to zero, hence $\tau^2 = 0$.

For the braid relations, the commuting ones are trivial from the
definitions.
For the length three braid relation,
it suffices to show that
$\dtau_1 \circ \dtau_2 \circ\dtau_1 = \dtau_2 \circ \dtau_1 \circ \dtau_2$
working in
$\END_{H_{3\alpha}}(\Delta(\alpha)\circ\Delta(\alpha)\circ\Delta(\alpha))$.
Let $w_1, w_2 \in S_{3n}$ be the permutations
mapping $(1,\dots,n,n+1,\dots,2n,2n+1,\dots,3n)$ to
$(n+1,\dots,2n,1,\dots,n,2n+1,\dots,3n)$ and
$(1,\dots,n,2n+1,\dots,3n,n+1,\dots,2n)$, respectively,
and set $w_0 := w_1 w_2 w_1 = w_2 w_1 w_2$.
By the defining relations for $H_{3\alpha}$, it is clear that
$(\tau_{w_2} \tau_{w_1} \tau_{w_2} - \tau_{w_2} \tau_{w_1} \tau_{w_2})
1_{\alpha,\alpha,\alpha} \otimes (v_\alpha \otimes v_\alpha \otimes
v_\alpha)$
lies in 
$$
S := \sum_{w < w_0} \tau_w 1_{\alpha,\alpha,\alpha}
\otimes (\Delta(\alpha) \boxtimes \Delta(\alpha) \boxtimes
\Delta(\alpha)).
$$
By the Mackey theorem, 
we have that
$$
S = \bigoplus_{w \in \{1,w_1,w_2,w_1w_2,w_2w_1\}}
\tau_w 1_{\alpha,\alpha,\alpha} \otimes (\Delta(\alpha) \boxtimes
\Delta(\alpha)\boxtimes\Delta(\alpha)).
$$
But the
vector 
$(\tau_{w_2} \tau_{w_1} \tau_{w_2} - \tau_{w_2} \tau_{w_1} \tau_{w_2})
1_{\alpha,\alpha,\alpha} \otimes (v_\alpha \otimes v_\alpha \otimes
v_\alpha)$ is of degree $3\deg(v_\alpha) - 6d_\alpha$, while all the vectors in $S$ are of
degree $\geq 3\deg(v_\alpha)-4d_\alpha$.
Hence this vector is zero, and we have shown that 
the endomorphisms $\dtau_2\circ\dtau_1\circ\dtau_2$ and $\dtau_1\circ\dtau_2\circ\dtau_1$ agree on
the generator $1_{\alpha,\alpha,\alpha} \otimes (v_\alpha\otimes
v_\alpha \otimes v_\alpha)$. Hence they are equal.
\end{proof}

In view of Lemma~\ref{dpl2}, we get well-defined endomorphisms
$\dtau_w$ of $\Delta(\alpha)^{\circ m}$ for each $w \in S_m$,
defined as usual from any reduced expression for $w$.
(This creates further ambiguity with the elements of the KLR
algebra with the same name, but this is only temporary.)

\begin{Lemma}\label{dpl3}
The endomorphisms
$\{\dtau_w \circ\dx_m^{k_m}\circ \cdots \circ \dx_1^{k_1} \:|\:w \in S_m, k_1,\dots,k_m
\geq 0\}$
give a basis for $\END_{H_{m\alpha}}(\Delta(\alpha)^{\circ m})$.
\end{Lemma}

\begin{proof}
These endomorphisms are linearly independent because they produce
linearly independent vectors when applied to $1_{\alpha,\dots,\alpha} \otimes (v_\alpha \otimes
\cdots \otimes v_\alpha)$.
It remains to show that
$$
\DIM 
\END_{H_{m\alpha}}(\Delta(\alpha)^{\circ m})
\leq \sum_{w \in S_m} \frac{q_\alpha^{-2\ell(w)}}{(1-q_\alpha^2)^m}.
$$
As $\Delta(\alpha)^{\boxtimes m}$ has head isomorphic to $L(\alpha)^{\boxtimes m}$
and $[\Delta(\alpha)^{\circ m}] = [L(\alpha)^{\circ m}] / (1-q_\alpha^2)^m$, 
we have by Frobenius reciprocity and Lemma~\ref{pl} that
\begin{align*}
\DIM 
\END_{H_{m\alpha}}(\Delta(\alpha)^{\circ m})
&= 
\DIM 
\HOM_{H_{\alpha,\dots,\alpha}}(\Delta(\alpha)^{\boxtimes m},
\Res^{m\alpha}_{\alpha,\dots,\alpha} \Delta(\alpha)^{\circ m})\\
&\leq 
[\Res^{m\alpha}_{\alpha,\dots,\alpha} \Delta(\alpha)^{\circ
  m}:L(\alpha)^{\boxtimes m}]\\
&=
[\Res^{m\alpha}_{\alpha,\dots,\alpha} L(\alpha)^{\circ m}:L(\alpha)^{\boxtimes m}] / (1-q_\alpha^2)^m\\
&=
q_\alpha^{-\frac{1}{2} m(m-1)}[m]_\alpha^! / (1-q_\alpha^2)^m.
\end{align*}
By the formula for the Poincar\'e polynomial of $S_m$ this
is 
$\sum_{w \in S_m} \frac{q_\alpha^{-2\ell(w)}}{(1-q_\alpha^2)^m}$.
\end{proof}

\begin{Lemma}\label{dpl4}
There is a {\em unique} choice for $x
\in \END_{H_\alpha}(\Delta(\alpha))_{2d_\alpha}$ 
such that the following relations hold: 
$\dtau_i \circ\dx_j = \dx_j \circ \tau_i$ for $j \neq i, i+1$,
$\dtau_i \circ \dx_{i+1} = \dx_i \circ\tau_i+1$ and
$\dx_{i+1} \circ \dtau_i = \dtau_i \circ \dx_i+1$.
\end{Lemma}

\begin{proof}
The commuting relations are automatic.
For the remaining relations, it suffices to
show working in
$\END_{H_{2\alpha}}(\Delta(\alpha)\circ\Delta(\alpha))$
that the (unique up to scalars) endomorphism $x
\in \END_{H_\alpha}(\Delta(\alpha))_{2 d_\alpha}$ 
can be chosen so that
$\tau \circ \dx_{2} = \dx_1 \circ\tau+1$ and
$\dx_{2} \circ \tau = \tau \circ \dx_1+1$.
Consider the endomorphisms
$$
\theta_+ := \tau \circ \dx_2 - \dx_1 \circ \tau,
\qquad
\theta_- := \tau \circ \dx_1 - \dx_2 \circ \tau.
$$
They are of degree zero and map
$1_{\alpha,\alpha} \otimes (v_\alpha \otimes v_\alpha)$ into
$1_{\alpha,\alpha} \otimes (\Delta(\alpha)\boxtimes \Delta(\alpha))$,
so we deduce from Lemma~\ref{dpl3} that
$\theta_{\pm} = c_\pm$
for some scalars 
$c_{\pm} \in \k$.
We have that
\begin{align*}
\tau x_1 x_2 - x_1 x_2 \tau &= 
 (x_2 \tau + c_-) x_2 - x_2 (\tau x_2 - c_+)
= (c_-+c_+)x_2,\\
\tau x_1 x_2 - x_1 x_2 \tau &= 
(x_1 \tau + c_+)x_1 - x_1 (\tau x_1 - c_-)=(c_++c_-) x_1.
\end{align*}
Since $x_1$ and $x_2$ are linearly independent this implies that
$c_-=-c_+$,
and we have proved that $c_+ = -c_- =c$ for some $c \in \k$.
It remains to show that $c \neq 0$, for then we can replace $x$ by
$x / c$ and get that $\theta_+ = 1, \theta_- = -1$ as required.
Suppose for a contradiction that $c = 0$.
Then $\tau \circ \dx_1 = \dx_2 \circ \tau$ and
$\tau \circ \dx_2 = \dx_1 \circ \tau$.
This means that $\tau$ leaves invariant the submodule
$S := \im \dx_1 + \im \dx_2$
of $\Delta(\alpha) \circ \Delta(\alpha)$,
hence it induces a well-defined endomorphism $\bar \tau$ of the quotient
$\Delta(\alpha) \circ \Delta(\alpha) / S$ with $\bar \tau^2 = 0$.
But $\Delta(\alpha) \circ \Delta(\alpha) / S \cong L(\alpha) \circ
L(\alpha)$, and under this isomorphism $\bar \tau$ corresponds to an
endomorphism sending $1_{\alpha,\alpha} \otimes (\bar v_\alpha \otimes \bar
v_\alpha)$ to $\tau_w 1_{\alpha,\alpha} \otimes (\bar v_\alpha \otimes
\bar v_\alpha)$.
This shows that $\END_{H_{2\alpha}}(L(\alpha) \circ L(\alpha))$ is
more than one-dimensional, contradicting the irreducibility of this
module from Theorem~\ref{mac1}.
\end{proof}

Henceforth, we assume that the endomorphism $x$ 
has been normalized according to Lemma~\ref{dpl4}.
Recalling the definition of the nil Hecke algebra $\NH_m$ from $\S$\ref{newSeccat},
Lemmas~\ref{dpl2}--\ref{dpl4} show that there is a unique
algebra {\em isomorphism}
\begin{equation*}
\NH_m \stackrel{\sim}{\rightarrow}
\END_{H_{m\alpha}}(\Delta(\alpha)^{\circ m})^{\op},\qquad
x_i \mapsto \dx_i,
\tau_j \mapsto \dtau_j.
\end{equation*}
The $\op$ here means that we view $\Delta(\alpha)^{\circ
  m}$ as a {\em right} $\NH_m$-module, i.e. it is an $(H_{m\alpha},
\NH_m)$-bimodule (a convention which finally eliminates the confusion between
the elements
$x_i, \tau_j$ of $H_{m\alpha}$ and the
elements  of $\NH_m$ with the same name: they act on different sides).
Finally define
the {\em divided power module}
\begin{equation}\label{divided}
\Delta(\alpha^m) := q_\alpha^{\frac{1}{2}m(m-1)}
\Delta(\alpha)^{\circ m}e_m
\end{equation}
where $e_m \in \NH_m$ is the idempotent (\ref{idemp}).

\begin{Lemma}\label{dp}
We have that $\Delta(\alpha)^{\circ m} \cong [m]_\alpha^! \Delta(\alpha^m)$
as an $H_{m\alpha}$-module.
Moreover $\Delta(\alpha^m)$ has irreducible head
$L(\alpha^m)$, and in the Grothendieck group we have that
$[\Delta(\alpha^m)] = [L(\alpha^m)] / (1-q_\alpha^2)(1-q_\alpha^4)\cdots (1-q_\alpha^{2m})$.
\end{Lemma}

\begin{proof}
So far we have identified the endomorphism algebra
$\END_{H_{m\alpha}}(\Delta(\alpha)^{\circ m})^{\op}$ with $\NH_m$
(graded
so that $x_i$ is in degree $2d_\alpha$ and $\tau_i$ is in degree $-2d_\alpha$).
Since $\NH_m \cong [m]_\alpha^! P$ 
where $P:=q_\alpha^{\frac{1}{2}m(m-1)} \NH_m e_m$,
we deduce that
$$
\Delta(\alpha)^{\circ m}
=
\Delta(\alpha)^{\circ m} \otimes_{\NH_m} \NH_m
\cong [m]_\alpha^! \Delta(\alpha)^{\circ m} \otimes_{\NH_m} P
\cong [m]_\alpha^! \Delta(\alpha^m).
$$
The fact that
$[\Delta(\alpha^m)] = [L(\alpha^m)] / (1-q_\alpha^2)(1-q_\alpha^4)\cdots (1-q_\alpha^{2m})$
follows from this using
$[\Delta(\alpha)] = [L(\alpha)] / (1-q_\alpha^2)$ 
and 
$L(\alpha^m)
= q_\alpha^{\frac{1}{2} m(m-1)} 
L(\alpha)^{\circ m}$.
Finally to show that the head of $\Delta(\alpha^m)$ is $L(\alpha^m)$,
it suffices to show that
$$
\DIM \HOM_{H_{m\alpha}}(\Delta(\alpha)^{\circ m}, L(\alpha^m)) = [m]_\alpha^!,
$$
which follows from Theorem~\ref{rm}(2), 
Lemma~\ref{pl} and Frobenius reciprocity.
\end{proof}

Thus we have constructed a module $\Delta(\alpha^m)$ which is equal in
the Grothendieck group to the divided power $r_\alpha^{m} / [m]_\alpha^!$.
More generally, for a Kostant partition $\lambda \in \KP$, gather together its equal parts to write 
it as $(\gamma_1^{m_1},\dots,\gamma_s^{m_s})$
with $\gamma_1 \succ \cdots \succ \gamma_s$, then define the {\em
  standard module}
\begin{equation}\label{standard}
\Delta(\lambda) := \Delta(\gamma_1^{m_1}) \circ\cdots \circ
\Delta(\gamma_s^{m_s}).
\end{equation}
Recalling (\ref{rlambda}),
the following theorem implies in particular 
that
$[\Delta(\lambda)]= r_\lambda$.

\begin{Theorem}\label{shead}
For $\lambda = (\lambda_1,\dots,\lambda_l) \in \KP$ we have that
$$
\Delta(\lambda_1) \circ \cdots \circ \Delta(\lambda_l)
\cong
[\lambda]^!
\Delta(\lambda).
$$
Moreover $V_0 := \Delta(\lambda)$ has an exhaustive
filtration
$V_0 \supset V_1 \supset
V_2 \supset \cdots$
such that $V_0 / V_1 \cong \bar\Delta(\lambda)$ 
and all other sections 
of the form $q^{2m}\bar\Delta(\lambda)$ for $m > 0$.
Finally $\Delta(\lambda)$
has irreducible head isomorphic to $L(\lambda)$, and 
$$
[\Delta(\lambda)] = 
[\bar \Delta(\lambda)] \:\:\Big / \prod_{\substack{\beta \in R^+ \\ 1
    \leq r \leq m_\beta(\lambda)}} (1-q_\beta^{2r}).
$$
\end{Theorem}

\begin{proof}
The isomorphism $\Delta(\lambda_1) \circ \cdots \circ \Delta(\lambda_l)
\cong
[\lambda]^!
\Delta(\lambda)$
and the Grothendieck group identity both follow from Lemma~\ref{dp}.
The existence of the filtration follows from Lemma~\ref{dp} and
exactness of induction.
Finally, to show that $\Delta(\lambda)$ has irreducible head, 
the 
filtration together with Theorem~\ref{mac2} implies that 
the only module that
could possibly appear with non-zero multiplicity 
in the head of $\Delta(\lambda)$ is $L(\lambda)$.
Now calculate using Frobenius reciprocity, Lemma~\ref{sc} and Lemma~\ref{dp}:
\begin{multline*}
\DIM \HOM_{H_\alpha}(\Delta(\lambda), L(\lambda))
=
\DIM
\HOM_{H_\alpha}(\Delta(\lambda_1)\circ\cdots\circ\Delta(\lambda_l), 
L(\lambda))\big /[\lambda]^!\\
=
\DIM
\HOM_{H_{\lambda_1}\otimes\cdots\otimes H_{\lambda_l}}(\Delta(\lambda_1)\boxtimes\cdots\boxtimes
\Delta(\lambda_l),L(\lambda_1)
\boxtimes\cdots\boxtimes L(\lambda_l))
= 1.
\end{multline*}
\end{proof}

\subsection{Standard homological properties}\label{shom}
In this subsection $\alpha \in Q^+$ is arbitrary.
The following theorem 
extends
the homological properties
proved originally in \cite[Theorem 4.12]{Kato} to non-simply-laced
types and to fields of positive characteristic.
Recall $\bar\nabla(\lambda) = \bar\Delta(\lambda)^{\circledast}$.

\begin{Theorem}\label{shp}
Suppose that $\lambda =(\lambda_1,\dots,\lambda_l) \in \KP(\alpha)$.
\begin{itemize}
\item[(1)] 
We have that
$\EXT^d_{H_\alpha}(\Delta(\lambda),V) = 0$
for all $d \geq 1$ and any finitely generated $H_\alpha$-module $V$
all of whose irreducible subquotients are 
of the form $q^nL(\mu)$ for $n \in \Z$
and $\mu \in \KP(\alpha)$ with $\mu\not\succ \lambda$.
\item[(2)]
We have that
$$
\DIM \EXT^d_{H_\alpha}(\Delta(\lambda), \bar\nabla(\mu)) =
\left\{
\begin{array}{ll}
1&\text{if $d=0$ and $\lambda=\mu$,}\\
0&\text{otherwise,}
\end{array}\right.
$$
for all $d \geq 0$ and $\mu \in \KP(\alpha)$.
\end{itemize}
\end{Theorem}

\begin{proof}
(1) 
We first prove this in the special case that $V = \bar\Delta(\mu)$
for $\mu\not\succ\lambda$.
By Theorem~\ref{shead} we have that
$$
\DIM \EXT^d_{H_\alpha}(\Delta(\lambda), \bar\Delta(\mu))
= \DIM \EXT^d_{H_\alpha}(\Delta(\lambda_1)\circ\cdots\circ
\Delta(\lambda_l),
\bar\Delta(\mu)) / [\lambda]^!.
$$
By generalized Frobenius reciprocity and 
Lemma~\ref{sc}, this is zero unless $\lambda \preceq \mu$.
If $\lambda = \mu$ it equals
$$
\sum_{d_1+\cdots+d_l = d}
\left(\prod_{k=1}^l\DIM \EXT^{d_k}_{H_{\lambda_k}}
\left(\Delta(\lambda_k), L(\lambda_k)\right) \right),
$$
which is zero by Theorem~\ref{rm}(3).
Using this special case and arguing by induction
on the ordering $\preceq$, it is straightforward to deduce that
the result is also true if $V = L(\mu)$
for $\mu\not\succ\lambda$. Then the result for general $V$ follows by 
Lemma~\ref{mittagleffler}.

\noindent
(2) By dualizing Theorem~\ref{mac2},
$\bar\nabla(\mu)$ has irreducible socle isomorphic to $L(\mu)$ and all
its other composition factors are of the form $q^m L(\nu)$ for $\nu \prec
\mu$ and $m \in \Z$.
Now use Theorem~\ref{shead}
to deduce the result when $d=0$. 
If $d \geq 1$ and $\mu\not\succ\lambda$
then we have that
$\EXT^d_{H_\alpha}(\Delta(\lambda), \bar\nabla(\mu)) = 0$
by (1).
Finally if $d \geq 1$ and $\mu \succ \lambda$, it suffices
to show equivalently that $\EXT^d_{H_\alpha}(\bar\Delta(\mu),
\nabla(\lambda)) = 0$ where $\nabla(\lambda) := \Delta(\lambda)^\circledast$.
This follows from Lemma~\ref{sc} and generalized Frobenius reciprocity
once again.
\end{proof}

We say that an $H_\alpha$-module $V$ has a {\em $\Delta$-flag}
if there is a (finite!) filtration $V = V_0 \supset V_1 \supset \cdots \supset V_n = 0$
such that $V_{i} / V_{i-1} \cong q^{m_i} \Delta(\lambda_i)$ for each
$i=1,\dots,n$ and
some $m_i \in
\Z$, $\lambda_i \in \KP(\alpha)$.
Note then by Theorem~\ref{shp}(2) that
\begin{equation*}
[V:\Delta(\lambda)] := 
 \sum_{\substack{1 \leq i \leq n \\ \lambda_i = \lambda}} q^{m_i}
=
\overline{\DIM \HOM_{H_\alpha}(V, \bar\nabla(\lambda))}
\end{equation*}
for each $\lambda \in \KP(\alpha)$,
so that this multiplicity is well-defined 
independent of the particular choice of 
$\Delta$-flag.
The following theorem (and its proof) is analogous to a 
well-known result (and proof) in the context of quasi-hereditary
algebras;
see e.g. \cite[Proposition A2.2(iii)]{Donkin}.

\begin{Theorem}\label{wef}
Suppose for $\alpha \in Q^+$ that $V$ is a finitely generated $H_\alpha$-module
with $\EXT^1_{H_\alpha}(V, \bar\nabla(\mu)) = 0$ for all 
$\mu \in \KP(\alpha)$.
Then $V$ has a $\Delta$-flag.
\end{Theorem}

\begin{proof}
Let $\ell(V) \in \N$ denote the sum of the dimensions of
$\HOM_{H_\alpha}(V, \bar\nabla(\lambda))$ for all $\lambda \in
\KP(\alpha)$; this makes sense because $V$ is finitely generated.
We proceed by induction on $\ell(V)$, the result being
trivial if $\ell(V) = 0$.
Suppose that $\ell(V) > 0$.
Let $\lambda$ be minimal such that $\HOM_{H_\alpha}(V, L(\lambda)) \neq 0$.
Then let $m \in \Z$ be minimal such that $\hom_{H_\alpha}(q^m V, L(\lambda)) \neq
0$.

We show in this paragraph that
$\EXT^1_{H_\alpha}(V, L(\mu)) = 0$ for all $\mu \preceq \lambda$.
There is a short exact sequence
$0 \rightarrow L(\mu) \rightarrow \bar\nabla(\mu) \rightarrow Q
\rightarrow 0$ where all composition factors of $Q$ are 
of the form $q^n L(\nu)$ for $n \in \Z$ and $\nu \prec \mu \preceq \lambda$.
By the minimality of $\lambda$,
$\HOM_{H_\alpha}(V, Q) = 0$.
Hence applying $\HOM_{H_\alpha}(V, -)$ to this short exact sequence, 
we obtain an exact sequence
$0 \rightarrow \EXT^1_{H_\alpha}(V, L(\mu)) \rightarrow \EXT^1_{H_\alpha}(V,
\bar\nabla(\mu)) = 0.$
We are done.

Now recall by Theorem~\ref{shead} that $\Delta(\lambda)$ has
irreducible head isomorphic to $L(\lambda)$.
In this paragraph, we show that there is a homogeneous surjection 
$q^m V \twoheadrightarrow \Delta(\lambda)$
by showing that
the natural map $\hom_{H_\alpha}(q^m V, \Delta(\lambda)) \rightarrow
\hom_{H_\alpha}(q^m V, L(\lambda))$ is surjective.
The long exact sequence
obtained by applying $\hom_{H_\alpha}(q^m V, -)$
to the short exact sequence
$0 \rightarrow \rad \Delta(\lambda) \rightarrow \Delta(\lambda) \rightarrow
L(\lambda) \rightarrow 0$ gives us an exact sequence
$$
\hom_{H_\alpha}(q^m V, \Delta(\lambda)) \rightarrow
\hom_{H_\alpha}(q^m V, L(\lambda)) \rightarrow
\ext^1_{H_\alpha}(q^m V, \rad \Delta(\lambda)).
$$
 Thus we are reduced to showing that $\ext^1_{H_\alpha}(q^m V, \rad
\Delta(\lambda)) = 0$.
This follows using the previous paragraph and
Lemma~\ref{mittagleffler},
noting 
by Theorems~\ref{mac2} and \ref{shead} that
all irreducible subquotients of $\rad \Delta(\lambda)$
are of the form $q^n L(\mu)$ for $n \in \Z$ and $\mu \preceq
\lambda$.

We have now proved that there is a
short exact sequence
$$
0 \rightarrow U \rightarrow V \rightarrow q^{-m} \Delta(\lambda)
\rightarrow 0
$$
for some submodule $U$ of $V$.
Applying $\HOM_{H_\alpha}(-, \bar\nabla(\mu))$
we get from the long exact sequence and Theorem~\ref{shp}(2)
that $\ell(U) < \ell(V)$ and
$\EXT^1_{H_\alpha}(U, \bar\nabla(\mu)) = 0$
for all $\mu \in \KP(\alpha)$.
Thus by induction $U$ has a $\Delta$-flag, hence so does $V$.
\end{proof}

As a corollary we obtain ``BGG reciprocity.''
For simply-laced types in characteristic zero, 
this was noted already in \cite[Remark 4.17]{Kato} (for convex orderings 
that are adapted to the orientation of the quiver).

\begin{Corollary}\label{bgg}
For any $\alpha \in Q^+$ and $\lambda \in \KP(\alpha)$,
the projective module $P(\lambda)$ has a $\Delta$-flag
with $[P(\lambda):\Delta(\mu)] = [\bar\Delta(\mu):L(\lambda)]$ 
(the latter notation denotes graded Jordan-H\"older multiplicity).
\end{Corollary}

\begin{proof}
Theorem~\ref{wef} immediately implies that $P(\lambda)$ has a $\Delta$-flag.
Moreover using also Theorem~\ref{shp}(2) we have that
$$
[P(\lambda):\Delta(\mu)] = 
\overline{\DIM\HOM_{H_\alpha}[P(\lambda), \bar\nabla(\mu)]}
= \overline{[\bar\nabla(\mu):L(\lambda)]}
= [\bar\Delta(\mu):L(\lambda)],
$$
as $L(\lambda)^\circledast \cong L(\lambda)$.
\end{proof}

\begin{Corollary}
For any $\alpha \in Q^+$ we have that
$$
\DIM H_\alpha = \!\!\sum_{\lambda \in \KP(\alpha)}
(\DIM \Delta(\lambda)) (\DIM \bar\Delta(\lambda))
= \!\!\sum_{\lambda \in \KP(\alpha)}
(\DIM \bar\Delta(\lambda))^2
\:\: \Big / 
\!\!\!\!\!\displaystyle\!\!
\prod_{\substack{\beta \in R^+ \\
      1 \leq r \leq m_\beta(\lambda)}}
\!\!\!(1-q_\beta^{2r}).
$$
\end{Corollary}

\begin{proof}
Again $H_\alpha$ has a $\Delta$-flag by Theorem ~\ref{wef}, so its dimension
is given by
\begin{align*}
\DIM H_\alpha &= 
\sum_{\lambda \in \KP(\alpha)}
(\DIM \Delta(\lambda))
(\overline{\DIM \HOM_{H_\alpha}(H_\alpha, \bar\nabla(\lambda))})\\
&=
\sum_{\lambda \in \KP(\alpha)}
(\DIM \Delta(\lambda))
(\overline{\DIM \bar\nabla(\lambda))})
=
\sum_{\lambda \in \KP(\alpha)}
(\DIM \Delta(\lambda))
(\DIM \bar\Delta(\lambda)).
\end{align*}
To deduce the second equality use the last part of Theorem~\ref{shead}.
\end{proof}

\begin{Corollary}\label{c}
For any $\lambda \in \KP$ we have that
\begin{align*}
\Delta(\lambda) &\cong
P(\lambda) \:\bigg/ 
\sum_{\phantom{q^n}\mu\not\preceq\lambda\phantom{q^n}}
\sum_{f:P(\mu)
  \rightarrow P(\lambda)} \im f,\\
\bar\Delta(\lambda) &
\cong P(\lambda) \:\bigg/ 
\sum_{\phantom{q^n}\mu\not\prec\lambda\phantom{q^n}}
\sum_{f:P(\mu)
  \rightarrow \rad P(\lambda)} \im f,
%\cong \Delta(\lambda) \bigg/ 
%\sum_{\phantom{q^n}n > 0\phantom{q^n}} \sum_{f:q^n P(\lambda) \rightarrow \Delta(\lambda)} \im f,
\end{align*}
summing over all (not necessarily homogeneous) homomorphisms $f$.
\end{Corollary}

\begin{proof}
For the first isomorphism, we fix a $\Delta$-flag $P(\lambda) = V_0 \supset
\cdots\supset V_n =0$.
As $P(\lambda)$ has irreducible head $L(\lambda)$,
the top section $V_0 / V_1$ must be isomorphic to $\Delta(\lambda)$,
while by Corollary~\ref{bgg} and Theorem~\ref{mac2} the other sections
are of the form $q^m \Delta(\mu)$ for $m\in \Z$ and $\mu \succ \lambda$.
Hence 
$\sum_{\mu\not\preceq\lambda}
\sum_{f:P(\mu)
  \rightarrow P(\lambda)} \im f$ is equal to $V_1$ and we are done.
The second isomorphism is proved in a similar way, using also 
Theorem~\ref{shead}.
\end{proof}

\begin{Remark}\rm
In simply-laced types with $\operatorname{char}\,\k = 0$,
Corollary~\ref{c} 
implies that our modules $\Delta(\lambda)$ and
$\bar\Delta(\lambda)$
coincide with the modules $\widetilde{E}_b$ and $E_b$
from \cite[Corollary 4.18]{Kato}
(for $b \in B(\infty)$ chosen so that $L(\lambda) \cong L_b$).
\end{Remark}

\section{Minimal pairs}\label{secmp}

In this section we show that the root modules $\Delta(\alpha)$
fit into some short exact sequences, giving an alternative
inductive way to deduce their properties. We apply this
to bound the projective dimension of standard modules, then construct
some projective resolutions of root modules. 
The key to the proofs is a useful recursive
formula for the root vectors $r_\alpha$.
This involves certain scale factors which were rather
mysterious before; cf. \cite{Lec}. 
As usual we work with a fixed convex ordering $\prec$ on $R^+$.

\subsection{Scale factors}\label{minp}
As in \cite{Mac}, we refer to the pairs $\lambda = (\beta,\gamma)$ from
the statement of Lemma~\ref{l3} 
as the {\em minimal pairs} for $\alpha \in R^+$.
Equivalently, a minimal pair for $\alpha$ is a pair $(\beta,\gamma)$
of positive roots with $\beta+\gamma=\alpha$ and
$\beta \succ \gamma$
such that there exists no other pair $(\beta',\gamma')$ of positive
roots with $\beta'+\gamma'=\alpha$ and $\beta \succ \beta' \succ
\alpha \succ \gamma' \succ \gamma$.
Let $\MP(\alpha)$ denote the set of all minimal pairs for $\alpha$.

For $\lambda = (\beta,\gamma) \in \MP(\alpha)$,
it is immediate from Theorem~\ref{mac2} and the minimality of
$\lambda$ that all composition 
factors of $\rad \bar\Delta(\lambda)$ are isomorphic 
to $L(\alpha)$ (up to degree shift). 
Since $\bar\Delta(\lambda) = L(\beta)\circ L(\gamma)$
and $(L(\beta)\circ L(\gamma))^\circledast \cong q^{\beta\cdot\gamma}L(\gamma) \circ
L(\beta)$
by
Lemma~\ref{lvl}, 
we deduce that there are short exact sequences
\begin{align}\label{sesonea}
0 \longrightarrow q^{-\beta\cdot\gamma}\X^\circledast &\longrightarrow L(\beta)\circ L(\gamma)
\longrightarrow L(\lambda) \longrightarrow 0,\\
0 \longrightarrow q^{-\beta\cdot\gamma} L(\lambda)
&\longrightarrow L(\gamma)\circ L(\beta)
\longrightarrow \X
\longrightarrow 0,\label{sestwoa}
\end{align}
where $\X := q^{-\beta\cdot\gamma}(\rad
\bar\Delta(\lambda))^\circledast$
is a finite dimensional module
with all 
composition factors isomorphic to $L(\alpha)$ (up to degree shift).
For $\beta,\gamma \in R$, let
$$
p_{\beta,\gamma} := \max\left(p \in \Z\:|\:\beta - p \gamma \in
  R\right).
$$

\begin{Lemma}\label{Cases}
For any $\alpha,\beta,\gamma \in R^+$ with $\beta+\gamma=\alpha$,
we have that
$$
d_\alpha(p_{\beta,\gamma} - \beta\cdot\gamma) = d_\beta d_\gamma
(p_{\beta,\gamma}+1),\qquad
[d_\alpha][p_{\beta,\gamma} - \beta\cdot\gamma] = [d_\beta][d_\gamma]
[p_{\beta,\gamma}+1].
$$
\end{Lemma}

\begin{proof}
By inspection of the rank two root systems, 
we have 
that
\begin{equation}\label{castab}
p_{\beta,\gamma} = \left\{
\begin{array}{ll}
2&\text{if $d_\alpha =3$ and $d_\beta = d_\gamma=1$,}\\
1&\text{if $d_\alpha =2$ and $d_\beta = d_\gamma=1$,}\\
1&\text{if
 $d_\alpha=d_\beta=d_\gamma=1$ in a subsystem of type
$\mathrm{G}_2$},\\
0&\text{otherwise.}
\end{array}\right.
\end{equation}
Moreover in the last case we have that $d_\alpha =
\min(d_\beta,d_\gamma)$.
Now consider the four cases in turn,
noting also that
$\beta\cdot\gamma = d_\alpha-d_\beta-d_\gamma$.
\end{proof}

\begin{Theorem}\label{sf}
Let $(\beta,\gamma)$ be a minimal pair for $\alpha \in R^+$.
Then
\begin{align}\label{ball1}
r_\gamma r_\beta - q^{-\beta\cdot\gamma} r_\beta
  r_\gamma &= [p_{\beta,\gamma}+1]
r_\alpha,\\
r_\gamma^* r_\beta^* - q^{-\beta\cdot\gamma} r_\beta^* r_\gamma^*
&= q^{-p_{\beta,\gamma}}\big(1-q^{2(p_{\beta,\gamma}-\beta\cdot\gamma)}
\big) r_\alpha^*.\label{ball2}
\end{align}
\end{Theorem}

\begin{proof}
In this paragraph we consider the special case that $\beta$ is simple.
Under this assumption \cite[Lemma 3.9]{KL} shows that
$L(\gamma) \circ L(\beta)$ has irreducible head, and moreover no
composition factors of $\rad (L(\gamma)\circ L(\beta))$ are isomorphic
to its head (up to degree shift).
We deduce that the module $\X$ in (\ref{sestwoa}) 
is isomorphic to $q^{-p} L(\alpha)$ for
some $p \in \Z$. Using this 
and considering the Grothendieck group
identities obtained from
(\ref{sesonea})--(\ref{sestwoa}), 
we deduce that
\begin{equation}\label{id1}
r^*_\gamma r^*_\beta - q^{-\beta\cdot\gamma} r_\beta^* r_\gamma^* =
q^{-p}(1 -q^{2(p-\beta\cdot\gamma)})r_\alpha^*.
\end{equation}
Rearranging this using (\ref{drv})
and $\beta\cdot\gamma = d_\alpha-d_\beta-d_\gamma$, 
we get that
\begin{equation}\label{id2}
r_\gamma r_\beta - q^{-\beta\cdot\gamma} r_\beta r_\gamma
=
\frac{[d_\alpha][p-\beta\cdot\gamma]}{[d_\beta][d_\gamma]} r_\alpha.
\end{equation}
Let us show further that $p \geq \beta\cdot\gamma$.
Applying $(\cdot,r_\alpha^*)$ to (\ref{id2}), we deduce that
$$
[d_\alpha][p-\beta\cdot\gamma] = [d_\beta][d_\gamma]
(r_\gamma r_\beta - q^{-\beta\cdot\gamma} r_\beta r_\gamma,
r_\alpha^*).
$$
Since $(r_\beta r_\gamma, r_\alpha^*) = (r_\beta \otimes r_\gamma,
r(r_\alpha^*))$,
where $r$ is the twisted coproduct from (\ref{twistedco})
which corresponds to restriction under the categorification theorem,
we see that
$(r_\beta r_\gamma, r_\alpha^*)$ is the graded composition
multiplicity
$[\res^\alpha_{\beta,\gamma} L(\alpha), L(\beta)\boxtimes L(\gamma)]$,
which is zero by Lemma~\ref{sc}.
Similarly $(r_\gamma r_\beta, r_\alpha^*) =
[\res^\alpha_{\gamma,\beta} L(\alpha):L(\gamma)\boxtimes L(\beta)]$,
which lies in $\N[q,q^{-1}]$.
Altogether this shows that $[d_\alpha][p - \beta\cdot \gamma]
\in \N[q,q^{-1}]$, hence indeed we must have that 
$p \geq \beta\cdot \gamma$.
Finally, when we specialize at $q=1$, the PBW basis elements $\{r_\alpha\}$
become Chevalley basis elements $\{e_\alpha\}$, and the identity
(\ref{id2})
becomes
$$
[e_\gamma, e_\beta] = \frac{d_\alpha (p-\beta\cdot\gamma)}{d_\beta
  d_\gamma}
e_\alpha.
$$
But for a Chevalley basis, $[e_\gamma, e_\beta] = \pm (p_{\beta,\gamma}+1)
e_\alpha$.
Since $p \geq \beta\cdot\gamma$,
we deduce from Lemma~\ref{Cases} that $p = p_{\beta,\gamma}$.
Hence (\ref{id1}) proves (\ref{ball2}).
Also using Lemma~\ref{Cases} once more
(\ref{id2}) proves (\ref{ball1}).

Now we deduce the general case.
Note it is sufficient just to prove (\ref{ball1}), for then (\ref{ball2})
follows by rearranging using (\ref{drv}) and Lemma~\ref{Cases}.
Let $w_0 = s_{i_1} \cdots s_{i_N}$ be the reduced expression of $w_0$
corresponding to the given convex ordering $\prec$.
Since $\gamma \prec \alpha \prec \beta$, 
there exist $c < a < b$ 
such that $\alpha = s_{i_1} \cdots s_{i_{a-1}}(\alpha_{i_a}),
\beta = s_{i_1} \cdots s_{i_{b-1}}(\alpha_{i_b})$ and
$\gamma = s_{i_1} \cdots s_{i_{c-1}}(\alpha_{i_c})$.
Working now in $U_q(\mathfrak{g})$, we need to prove that
\begin{multline*}
[p_{\beta,\gamma}+1]T_{i_1} \cdots T_{i_{a-1}}(E_{i_a}) =
T_{i_1}\cdots T_{i_{c-1}}(E_{i_c}) 
T_{i_1}\cdots T_{i_{b-1}}(E_{i_b}) \\
- q^{-\beta\cdot\gamma}
T_{i_1}\cdots T_{i_{b-1}}(E_{i_b}) 
T_{i_1}\cdots T_{i_{c-1}}(E_{i_c}).
\end{multline*}
Let $s_{j_1} \cdots s_{j_{N-b}}$ be a reduced expression for $w_0 s_{i_b} \cdots s_{i_1}$
and $\prec'$ be the convex ordering corresponding to the decomposition
$w_0 = s_{j_1} \cdots s_{j_{N-b}} s_{i_1} \cdots s_{i_b}$.
Let $\alpha' := s_{j_1} \cdots s_{j_{N-b}}(\alpha)$,
$\beta' := s_{j_1} \cdots s_{j_{N-b}}(\beta)$ and
$\gamma' := s_{j_1} \cdots s_{j_{N-b}}(\gamma)$.
Note that $\beta'$ is simple and $(\beta',\gamma')$ is a minimal pair
for $\alpha'$ with respect to the convex ordering $\prec'$.
Acting with $T_{j_1} \cdots T_{j_{N-b}}$, the identity we are trying
to prove is equivalent to the identity
\begin{multline*}
[p_{\beta,\gamma}+1]T_{j_1} \cdots T_{j_{N-b}} T_{i_1} \cdots T_{i_{a-1}}(E_{i_a}) =\\
T_{j_1} \cdots T_{j_{N-b}} T_{i_1}\cdots T_{i_{c-1}}(E_{i_c}) 
T_{j_1} \cdots T_{j_{N-b}} T_{i_1}\cdots T_{i_{b-1}}(E_{i_b}) \qquad\\
- q^{-\beta\cdot\gamma}
T_{j_1} \cdots T_{j_{N-b}} T_{i_1}\cdots T_{i_{b-1}}(E_{i_b}) 
T_{j_1} \cdots T_{j_{N-b}} T_{i_1}\cdots T_{i_{c-1}}(E_{i_c}).
\end{multline*}
But in $\f$
this is saying simply that
$[p_{\beta',\gamma'}+1]r_{\alpha'} = r_{\gamma'} r_{\beta'} - q^{-\beta'\cdot\gamma'}
r_{\beta'} r_{\gamma'}$,
where the root elements here are defined with respect to the new
convex ordering $\prec'$.
This
follows from the special case treated in  the previous paragraph.
\end{proof}

\begin{Corollary}\label{chprop}
Let $(\beta,\gamma)$ be a minimal pair for $\alpha
\in R^+$. In the Grothendieck group we have that
$\left[\Res^{\alpha}_{\gamma,\beta} L(\alpha)\right] = 
[p_{\beta,\gamma}+1]\big[L(\gamma) \boxtimes L(\beta)\big]$.
\end{Corollary}

\begin{proof}
By \cite[Lemma 4.1]{Mac}, $\left[\Res^\alpha_{\gamma,\beta} L(\alpha) \right]$ is a
scalar multiple of $[L(\gamma)\boxtimes L(\beta)]$.
To compute the scalar we make a computation with Lusztig's form
like we did in the proof of Theorem~\ref{sf}:
\begin{align*}
(r_\gamma \otimes r_\beta, r (r_\alpha^*))
&= (r_\gamma r_\beta, r_\alpha^*)
= (r_\gamma r_\beta-q^{-\beta\cdot\gamma} r_\beta r_\gamma, r_\alpha^*)
+ q^{-\beta\cdot\gamma} (r_\beta r_\gamma, r_\alpha^*)\\
&= [p_{\beta,\gamma}+1](r_\alpha, r_\alpha^*) + q (r_\beta \otimes r_\gamma, r(r_\alpha^*))
= [p_{\beta,\gamma}+1].
\end{align*}
\end{proof}

\begin{Remark}\rm
We will show in Theorem~\ref{as} below that the module $\X$ in (\ref{sestwoa})
is isomorphic to $q^{-p_{\beta,\gamma}} L(\alpha)$. Hence
for $(\beta,\gamma) \in \MP(\alpha)$ the module
$L(\gamma)\circ L(\beta)$ has irreducible head $q^{-p_{\beta,\gamma}}
L(\alpha)$.
Applying Frobenius reciprocity, it follows that 
the self-dual
module $\Res^{\alpha}_{\gamma,\beta} L(\alpha)$
has irreducible socle $q^{p_{\beta,\gamma}} L(\gamma)\boxtimes L(\beta)$.
Hence it 
is uniserial with composition factors as described by Corollary~\ref{chprop}.
\end{Remark}

\subsection{Leclerc's algorithm}
Theorems~\ref{ct} and \ref{mac2} imply that
\begin{equation}\label{tr1}
\b^*(r_\lambda^*) = r_\lambda^* + (\text{a $\Z[q,q^{-1}]$-linear
  combination of $r_\mu^*$ for $\mu \prec \lambda$}).
\end{equation}
Hence by duality we also have that 
\begin{equation}\label{tr2}
\b(r_\lambda) = r_\lambda + (\text{a $\Z[q,q^{-1}]$-linear
  combination of $r_\mu$ for $\mu \succ \lambda$}).
\end{equation}
There are several other ways to prove the unitriangularity of the
transition matrices here: the identity
(\ref{tr1}) follows from the
$\b^*$-invariance of the dual root vectors noted earlier together with
the Levendorskii-Soibelman formula \cite[Proposition 5.5.2]{LS}; the
identity (\ref{tr2}) can be deduced directly starting
from \cite[Proposition 1.9]{Lubraid}.
Combining (\ref{tr1})--(\ref{tr2}) 
with Lusztig's lemma, it follows that there exist unique bases
$\{b_\lambda\:|\:\lambda \in \KP\}$ and $\{b_\lambda^*\:|\:\lambda \in
\KP\}$
for $\f_\A$ and $\f_\A^*$, respectively, such that
\begin{align}
\b(b_\lambda) &= b_\lambda,
 &b_\lambda &= r_\lambda + \text{(a $q \Z[q]$-linear combination of
  $r_\mu$ for $\mu \succ \lambda$)},\\
\b^*(b^*_\lambda) &= b^*_\lambda,
&b_\lambda^*  &= r^*_\lambda
 + \text{(a $q \Z[q]$-linear combination of
  $r^*_\mu$ for $\mu \prec \lambda$).}\label{tri}
\end{align}
Of course these are the canonical and dual canonical bases for
$\f$, respectively, as follows by the definition in
\cite{Lu}
in simply-laced types
or \cite{Saito} in non-simply-laced types.

For simply-laced types over fields of characteristic zero, the results
of
Rouquier \cite[$\S$5]{R2} and
Varagnolo-Vasserot \cite{VV}
show 
under the identification
from Theorem~\ref{ct} 
that $b_\lambda = [P(\lambda)]$, hence also $b_\lambda^* =
[L(\lambda)]$, for each $\lambda \in \KP$.
In all types, there is a recursive algorithm essentially due to 
Leclerc \cite[$\S$5.5]{Lec} to compute 
$b_\lambda^*$, or rather its image $\CH b_\lambda^*$ in the quantum
shuffle algebra; we will explain this in more detail shortly.
Putting these two statements together, we obtain an effective
algorithm to compute the characters of the irreducible
$H_\alpha$-modules 
in all simply-laced types over fields of
characteristic zero. At present there is no reasonable way to 
compute the irreducible
characters in non-simply-laced types 
or in positive characteristic; the approach via the contravariant
form illustrated by
Example~\ref{willcex} is seldom feasible in practice.

In the remainder of the subsection, we recall Leclerc's
algorithm for computing dual canonical bases in more detail. 
Actually we describe a slightly
modified version of the algorithm which works for an arbitrary convex ordering
(rather than just for the Lyndon orderings considered in Leclerc's work).
Assume that we have chosen a minimal pair
$\mp(\alpha) \in \MP(\alpha)$ for each $\alpha \in R^+$ of height at
least two.
Dependent on these choices, we
recursively define a word $\bi_\alpha \in \W_\alpha$
and a bar-invariant Laurent polynomial $\kappa_\alpha \in \A$: for $i
\in I$ set $\bi_{\alpha_i} := i$ and $\kappa_{\alpha_i} := 1$; then for $\alpha \in R^+$ of height $\geq 2$
suppose that $(\beta,\gamma) = \mp(\alpha)$ and set
\begin{equation}\label{issue}
\bi_\alpha := \bi_\gamma \bi_\beta,\qquad
\kappa_\alpha := [p_{\beta,\gamma}+1]\kappa_\beta \kappa_\gamma.
\end{equation}
For example, in simply-laced types
we have that $\kappa_\alpha = 1$ for all $\alpha \in R^+$; 
this is also the case in non-simply-laced types for 
{\em multiplicity-free} positive roots, i.e. roots 
$\alpha = \sum_{i \in I} c_i \alpha_i$ with $c_i \in \{0,1\}$ for all
$i$.
Then for $\lambda = (\lambda_1,\dots,\lambda_l) \in \KP$ let
\begin{equation}
\bi_\lambda := \bi_{\lambda_1} \cdots \bi_{\lambda_l},
\qquad
\kappa_\lambda := [\lambda]^! \kappa_{\lambda_1} \cdots
\kappa_{\lambda_l}.
\end{equation}
Part (1) of the following lemma shows that the words $\bi_\lambda$
distinguish irreducible modules,
generalizing \cite[Theorem 7.2(ii)]{KR2}.

\begin{Lemma}\label{rtheory}
The following hold for any $\alpha \in Q^+$ and 
$\lambda, \mu \in \KP(\alpha)$.
\begin{itemize}
\item[(1)] We have that $\DIM 1_{\bi_\mu} L(\lambda) = 0$ if $\mu \not\preceq
\lambda$, and $\DIM 1_{\bi_\lambda} L(\lambda) = \kappa_\lambda$.
\item[(2)]
The $\bi_\mu$-coefficient of $\CH b_\lambda^*$ is zero 
if $\mu \not\preceq
\lambda$, and the $\bi_\lambda$-coefficient of $\CH b_\lambda^*$ is equal
to $\kappa_\lambda$.
\end{itemize}
\end{Lemma}

\begin{proof}
(1) Since $L(\lambda)$ is a quotient of $\bar\Delta(\lambda)$, 
the first statement is immediate from Lemma~\ref{sc}.
Using also the triangularity property from 
Theorem~\ref{mac2}, the same lemma 
reduces the proof of the second statement to showing that
$\DIM 1_{\bi_\alpha} L(\alpha) = \kappa_\alpha$ for $\alpha \in R^+$.
To see this proceed by induction on height. For the induction step,
suppose that
 $(\beta,\gamma)=\mp(\alpha)$ and apply Corollary~\ref{chprop}:
\begin{align*}
\DIM 1_{\bi_\alpha} L(\alpha)
&= [p_{\beta,\gamma}+1]
(\DIM 1_{\bi_\beta} L(\beta))(\DIM 1_{\bi_\gamma} L(\gamma))= [p_{\beta,\gamma}+1] \kappa_\beta \kappa_\gamma = \kappa_\alpha.
\end{align*}

\noindent
(2)
Repeat the proof of (1) working in $\f_\A^*$ rather
than in $[\Rep{H}]$,
using $[\bar\Delta(\lambda)] = r_\lambda^*$,
$[L(\alpha)] = r_\alpha^*$ and the triangularity from (\ref{tri}).
\end{proof}

Now we can explain the algorithm to compute $\CH b_\lambda^*$ for $\lambda
\in \KP$.
This goes by induction on the partial order $\preceq$, so we assume 
$\CH b_\mu^*$ is known for all $\mu \prec \lambda$.
We can compute 
$\CH r_\alpha^* \in \A\W$
for any $\alpha \in R^+$ by using the recursive formula 
obtained by applying $\CH$ to the identity
(\ref{ball2}).
Hence we can compute $\chi := \CH r_\lambda^* \in \A\W$.
Now inspect the $\bi_\mu$-coefficients of $\chi$
for $\mu \prec \lambda$. 
If they are all bar-invariant 
then $\chi = \CH b_\lambda^*$ and we are done. Otherwise, let
$\mu \prec \lambda$ be maximal such that the coefficient $a(q)$ 
of $\bi_\mu$
in $\chi$ is not bar-invariant; Lemma~\ref{rtheory}(2) and (\ref{tri}) imply 
that there is a unique $c(q)
\in q \Z[q]$ such that $a(q) - c(q) \kappa_\mu$ is bar-invariant;
then subtract $c(q) \CH b_\mu^*$ from $\chi$ and repeat.

\iffalse
\begin{Example}\rm
Suppose we are in type G$_2$ with simple roots $\alpha_1$ and $\alpha_2$ 
chosen so that they are long and short, respectively.
There are just two convex orderings. We consider the one in which
$\alpha_1 \prec \alpha_2$. Then the above algorithm gives:
\begin{align*}
b^*_{(\alpha_1)} &= 1;\\
b^*_{(\alpha_1+\alpha_2)} &= 12,\\
b^*_{(\alpha_2,\alpha_1)} &= 21;\\
b^*_{2 \alpha_1+3\alpha_2} &= [2]_1 [2]_2 [3]_2 11222+[2]_2 [3]_2
12122,\\
b^*_{(\alpha_1+2\alpha_2,\alpha_1+\alpha_2)} &=[2]_2 12212,\\
b^*_{(\alpha_1+3\alpha_2,\alpha_1)} &= [2]_2 [3]_2 12221,\\
b^*_{(\alpha_2,\alpha_1+\alpha_2,\alpha_1+\alpha_2)} &=[2]_2
12122+[2]_1 [2]_2 21122+[2]_2 21212,\\
b^*_{(\alpha_2,\alpha_1+2\alpha_2,\alpha_1)} &=[2]_2 21221,\\
b^*_{(\alpha_2,\alpha_2,\alpha_1+\alpha_2,\alpha_1)} &=[2]_2
21212+[2]_1 [2]_2 22112+[2]_2 22121,\\
b^*_{(\alpha_2,\alpha_2,\alpha_2,\alpha_1,\alpha_1)}&= [2]_2 [3]_2
22121+[2]_1 [2]_2 [3]_2 22211;\\
b^*_{(\alpha_1+2\alpha_2)} &=[2]_2 122,\\
b^*_{(\alpha_2,\alpha_1+\alpha_2)}&=212,\\
b^*_{(\alpha_2,\alpha_2,\alpha_1)}&= [2]_2 221;\\
b^*_{(\alpha_1+3\alpha_2)} &=[2]_2 [3]_2 1222,\\
b^*_{(\alpha_2,\alpha_1+2\alpha_2)} &=[2]_2 2122,\\
b^*_{(\alpha_2,\alpha_2,\alpha_1+\alpha_2)} &=[2]_2 2212,\\
b^*_{(\alpha_2,\alpha_2,\alpha_2,\alpha_1)} &=[2]_2 [3]_2 2221;\\
b^*_{(\alpha_2)} &= 2.
\end{align*}
Here $[2]_1 = q^3+q^{-3}, [2]_2 = q+q^{-1}$ and $[3]_2 = q^2+1+q^{-2}$.
\end{Example}
\fi

\begin{Example}\rm\label{G2ex}
The following example is mostly taken from \cite[$\S$5.5.4]{Lec}.
Suppose we are in type G$_2$ with simple roots $\alpha_1$ and $\alpha_2$ 
chosen so that they are short and long, respectively.
There are just two convex orderings. We consider the one in which
$\alpha_1 \prec \alpha_2$. Then the above algorithm gives:

$\CH b^*_{(\alpha_1)} = 1;$

$\CH b^*_{(3\alpha_1+\alpha_2)}=[2]_1[3]_1 1112,$

$\CH b^*_{(2\alpha_1+\alpha_2,\alpha_1)} =[2]_1 1121,$

$\CH b^*_{(\alpha_1+\alpha_2,\alpha_1,\alpha_1)} =[2]_1 1211,$

$\CH b^*_{(\alpha_2,\alpha_1,\alpha_1,\alpha_1)} =[2]_1[3]_1 2111;$

$\CH b^*_{(2\alpha_1+\alpha_2)}=[2]_1 112,$

$\CH b^*_{(\alpha_1+\alpha_2,\alpha_1)} =121,$

$\CH b^*_{(\alpha_2,\alpha_1,\alpha_1)} =[2]_1 211;$

$\CH b^*_{(3\alpha_1+2\alpha_2)} =[2]_2 [2]_1[3]_1 11122+[2]_1[3]_1
11212,$

$\CH b^*_{(\alpha_1+\alpha_2,2\alpha_1+\alpha_2)} =[2]_1 12112,$

$\CH b^*_{(\alpha_1+\alpha_2,\alpha_1+\alpha_2,\alpha_1)} =[2]_1
11212+[2]_2 [2]_1 11221+[2]_1 12121,$

$\CH b^*_{(\alpha_2,3\alpha_1+\alpha_2)} =[2]_1[3]_1 21112,$

$\CH b^*_{(\alpha_2,2\alpha_1+\alpha_2,\alpha_1)}=[2]_1 21121,$

$\CH b^*_{(\alpha_2,\alpha_1+\alpha_2,\alpha_1,\alpha_1)}=[2]_1
12121+[2]_2 [2]_1 12211+[2]_1 21211,$

$\CH b^*_{(\alpha_2,\alpha_2,\alpha_1,\alpha_1,\alpha_1)}=[2]_1[3]_1
21211+[2]_2 [2]_1[3]_1 22111;$

$\CH b^*_{(\alpha_1+\alpha_2)}=12,$

$\CH b^*_{(\alpha_2,\alpha_1)} =21;$

$\CH b^*_{(\alpha_2)} =2.$
\iffalse
\begin{align*}
\CH b^*_{(\alpha_1)} &= 1;\\
\CH b^*_{(3\alpha_1+\alpha_2)}&=[2]_1[3]_1 1112,\\
\CH b^*_{(2\alpha_1+\alpha_2,\alpha_1)} &=[2]_1 1121,\\
\CH b^*_{(\alpha_1+\alpha_2,\alpha_1,\alpha_1)} &=[2]_1 1211,\\
\CH b^*_{(\alpha_2,\alpha_1,\alpha_1,\alpha_1)} &=[2]_1[3]_1 2111;\\
\CH b^*_{(2\alpha_1+\alpha_2)}&=[2]_1 112,\\
\CH b^*_{(\alpha_1+\alpha_2,\alpha_1)} &=121,\\
\CH b^*_{(\alpha_2,\alpha_1,\alpha_1)} &=[2]_1 211;\\
\CH b^*_{(3\alpha_1+2\alpha_2)} &=[2]_2 [2]_1[3]_1 11122+[2]_1[3]_1 11212,\\
\CH b^*_{(\alpha_1+\alpha_2,2\alpha_1+\alpha_2)} &=[2]_1 12112,\\
\CH b^*_{(\alpha_1+\alpha_2,\alpha_1+\alpha_2,\alpha_1)} &=[2]_1
11212+[2]_2 [2]_1 11221+[2]_1 12121,\\
\CH b^*_{(\alpha_2,3\alpha_1+\alpha_2)} &=[2]_1[3]_1 21112,\\
\CH b^*_{(\alpha_2,2\alpha_1+\alpha_2,\alpha_1)}&=[2]_1 21121,\\
\CH b^*_{(\alpha_2,\alpha_1+\alpha_2,\alpha_1,\alpha_1)}&=[2]_1
12121+[2]_2 [2]_1 12211+[2]_1 21211,\\
\CH b^*_{(\alpha_2,\alpha_2,\alpha_1,\alpha_1,\alpha_1)}&=[2]_1[3]_1 21211+[2]_2 [2]_1[3]_1 22111;\\
\CH b^*_{(\alpha_1+\alpha_2)}&=12,\\
\CH b^*_{(\alpha_2,\alpha_1)} &=21;\\
\CH b^*_{(\alpha_2)} &=2.
\end{align*}
\fi

\noindent
Here $[2]_1 = q+q^{-1}, [2]_2 = q^3+q^{-3}$ and $[3]_1 =
q^2+1+q^{-2}$.
\end{Example}

\subsection{The length two property}\label{ltp}
Suppose that $\alpha \in R^+$.
The following theorem 
was stated as a conjecture in \cite[Conjecture 2.16]{BrK}.
It shows that the proper standard module $\bar\Delta(\lambda)$
has length two for all minimal pairs $\lambda$ for $\alpha$.

\begin{Theorem}\label{as}
For $\lambda = (\beta,\gamma) \in \MP(\alpha)$
there are short exact sequences
\begin{align}\label{sesoneas}
0 \longrightarrow q^{p_{\beta,\gamma}-\beta\cdot\gamma}L(\alpha) &\longrightarrow L(\beta)\circ L(\gamma)
\longrightarrow L(\lambda) \longrightarrow 0,\\
0 \longrightarrow q^{-\beta\cdot\gamma} L(\lambda)
&\longrightarrow L(\gamma)\circ L(\beta)
\longrightarrow q^{-p_{\beta,\gamma}} L(\alpha)
\longrightarrow 0.\label{sestwoas}
\end{align}
\end{Theorem}

\begin{proof}
Recall the short exact sequence (\ref{sestwoa}).
The module 
$\X$ has all composition factors isomorphic to $L(\alpha)$ (up to
shift).
To prove the theorem we need to show that $\X \cong
q^{-p_{\beta,\gamma}} L(\alpha)$.

Suppose first that $\alpha,\beta,\gamma$ lie in a subsystem of type
G$_2$.
There are just two convex orderings. For one of these, the character
$b_\alpha^*$ of
$L(\alpha)$ is listed in Example~\ref{G2ex}, and using this
it is
easy to check that 
$\CH L(\gamma) \circ \CH L(\beta)
= \kappa_\lambda \bi_\lambda + q^{-p_{\beta,\gamma}} \CH L(\alpha)$. 
Comparing with (\ref{sestwoa}) we therefore must have that
$\DIM 1_{\bi_\alpha} \X + q^{-\beta\cdot\gamma} 
\DIM 1_{\bi_\alpha} L(\lambda) = q^{-p_{\beta,\gamma}} \DIM
1_{\bi_\alpha} L(\alpha)$. 
The only way this could happen is if
$1_{\bi_\alpha} L(\lambda) =
0$ and $\X \cong q^{-p_{\beta,\gamma}}L(\alpha)$.
The argument for the other convex ordering is similar.

From now on we assume we are not in G$_2$.
Next we treat the case that $d_\alpha = d_\beta = d_\gamma$, when
$p_{\beta,\gamma} = 0$ by (\ref{castab}).
Applying the functor $\HOM_{H_\alpha}(-, L(\alpha))$ to
the short exact sequence (\ref{sestwoa}), using Frobenius
reciprocity
and Corollary~\ref{chprop}, we
deduce that
$$
\HOM_{H_\alpha}(\X, L(\alpha)) \cong 
\HOM_{H_{\gamma}\otimes H_{\beta}}(L(\gamma)
\boxtimes L(\beta), \Res^\alpha_{\gamma,\beta} L(\alpha))
\cong \k.
$$
Hence $\X$ has irreducible head $L(\alpha)$. Therefore by
Corollary~\ref{class}
we have that $\X \cong \Delta_n(\alpha)$ for some $n \geq 1$, i.e.
$[\X] = \frac{1-q_\alpha^{2n}}{1-q_\alpha^2} [L(\alpha)]$.
To prove the theorem we must show that $n=1$.
For this we compute $r_\gamma^* r_\beta^* - q^{-\beta\cdot\gamma}
r_\beta^* r_\gamma^*$ first from (\ref{sesonea})--(\ref{sestwoa})
then from (\ref{ball2}) to deduce that
$$
\frac{1-q_\alpha^{2n}}{1-q_\alpha^2} - q^{-2 \beta\cdot\gamma}
\frac{1-q_\alpha^{-2n}}{1-q_\alpha^{-2}} = 1 - q^{-2 \beta\cdot\gamma}.
$$
This easily implies that $n=1$.

Next we treat the case that $d_\alpha > d_\beta = d_\gamma$,
when $p_{\beta,\gamma}=1$ and $\beta\cdot\gamma = 0$.
Then $\left[\Res^\alpha_{\gamma,\beta} L(\alpha)\right] = (q+q^{-1})
[L(\gamma)\boxtimes L(\beta)]$, and the above calculation and 
Corollary~\ref{class}
show that
$\X \cong q^{-1}\Delta_n(\alpha) \oplus q \Delta_m(\alpha)$
for $n \geq 1$ and $m \geq 0$.
As above, we obtain the identity
$$
q^{-1} \frac{1-q^{4n}}{1-q^4} - q
\frac{1-q^{-4n}}{1-q^{-4}}
+q\frac{1-q^{4m}}{1-q^4} - q^{-1}
\frac{1-q^{-4m}}{1-q^{-4}}
= q^{-1}-q,
$$
which implies $n=1$ and $m=0$ as required.

We are left with the case that $\beta$ and $\gamma$ are of different
lengths in a subsystem of type B$_r$, C$_r$ or F$_4$,
when $p_{\beta,\gamma} = 0$ and $\beta\cdot\gamma = -2$. Unfortunately
here the above method breaks down: it shows only that
$\X \cong L(\alpha)$ (as required) or that $\X \cong \Delta_2(\alpha)$.
To rule out the latter possibility in types B$_r$ or C$_r$
it is enough using (\ref{sestwoa}) and Corollary~\ref{chprop} to show
that $\res^\alpha_{\gamma,\beta} L(\gamma)\circ L(\beta) \cong
L(\gamma)\boxtimes L(\beta)$.
Realizing $R^+$ as $\{\epsilon_i \pm \epsilon_j, d \epsilon_k\:|\:1 \leq i < j \leq r,
1 \leq k \leq r\}$ in the standard way, where $d=1$ for B$_r$ or $2$ for C$_r$,
the assumptions $\beta+\gamma \in R^+$ and $d_\beta \neq d_\gamma$
imply that
$\{\beta,\gamma\} = \{\epsilon_i - \epsilon_j, d \epsilon_j\}$
for some $1 \leq i < j \leq r$. Now apply Theorem~\ref{mackey}, noting
that
there is only one non-zero section in the resulting filtration.

Finally we treat F$_4$ for $\beta$ and $\gamma$ of different
lengths.
Here we must resort to some explicit
calculation. However there are now 
2,144,892 
different convex
orderings!
The following argument avoids the need to make a separate computation
for each one in turn.
Suppose for a contradiction that $\X \cong \Delta_2(\alpha)$.
Then (\ref{sesonea}) implies that
$[L(\lambda)] = r_\beta^* r_\gamma^*- (1+q^{2}) r_\alpha^*$.
This is $\b^*$-invariant, as is $r_\alpha^* = b_{(\alpha)}^*$, hence
$r_\beta^* r_\gamma^*  - q^2 r_\alpha^*$ is $\b^*$-invariant too.
We deduce from (\ref{tri}) that
$b_\lambda^* = r_\beta^* r_\gamma^* - q^2 r_\alpha^*$.
This shows that 
$\CH L(\lambda) =
\CH b_\lambda^*  - \CH b_{(\alpha)}^*$, 
hence $\CH (b_\lambda^* - b_{(\alpha)}^*)$
is an $\N[q,q^{-1}]$-linear combination of words in
$\A\W$.
Now we made an explicit computer calculation of the dual canonical
basis of $\f^*_\alpha$, using the algorithm in the previous subsection
with respect to a particular choice of convex ordering, 
showing at the end 
that no two dual canonical basis elements of $\f_\alpha$
have positive difference in this sense; see \cite{Brun} 
for details. This produces the
desired contradiction.
\end{proof}

\begin{Corollary}
For $\lambda \in \MP(\alpha)$ we have that $[L(\lambda)] =
b_\lambda^*$.
\end{Corollary}

\begin{proof}
By (\ref{sesoneas}) we have that $[L(\lambda)] = 
r_\beta^* r_\gamma^* -
q^{p_{\beta,\gamma}-\beta\cdot\gamma}r_\alpha^*$,
and 
$p_{\beta,\gamma}-\beta\cdot\gamma > 0$ by Lemma~\ref{Cases}.
Hence 
using the characterization (\ref{tri}) this is also $b_\lambda^*$.
\end{proof}

\subsection{A short exact sequence}\label{sses}
In this subsection we fix $\alpha \in R^+$
of height $n \geq 2$.
Let $(\beta,\gamma)$ be a minimal pair for $\alpha$
and set $m := \hgt(\gamma)$.

\begin{Lemma}\label{inj1}
Let $w \in S_{n}$ be
$(1,\dots,n) \mapsto (n-m+1,\dots,n,1,\dots,n-m)$, so that
$\tau_w 1_{\gamma,\beta} = 1_{\beta,\gamma} \tau_w$.
There is a unique homogeneous homomorphism 
$$
\phi:q^{-\beta\cdot\gamma}\Delta(\beta)\circ\Delta(\gamma)
\rightarrow \Delta(\gamma) \circ \Delta(\beta)
$$
such that 
$\phi(1_{\beta,\gamma} \otimes (v_1 \otimes v_2))
= \tau_w 1_{\gamma,\beta} \otimes (v_2 \otimes v_1)$
for all $v_1 \in \Delta(\beta), v_2 \in \Delta(\gamma)$.
\end{Lemma}

\begin{proof}
It suffices by Frobenius reciprocity to show that there is an isomorphism
$$
q^{-\beta\cdot\gamma}\Delta(\beta)\boxtimes \Delta(\gamma)
\stackrel{\sim}{\rightarrow}
\Res^{\alpha}_{\beta,\gamma} \Delta(\gamma)\circ\Delta(\beta),\quad
v_1 \otimes v_2 \mapsto \tau_w 1_{\gamma,\beta} \otimes (v_2 \otimes
v_1).
$$
To see this we apply Theorem~\ref{mackey}.
Suppose we are given
$\beta_1,\beta_2,\gamma_1,\gamma_2 \in Q^+$
such that
$\gamma = \gamma_1+\gamma_2 = \gamma_2+\beta_2$,
$\beta = \beta_1+\beta_2 = \gamma_1+\beta_1$, 
and both of the restrictions 
$\Res^\gamma_{\gamma_1,\gamma_2} \Delta(\gamma)$
and
$\Res^\beta_{\beta_1,\beta_2} \Delta(\beta)$
are non-zero.
By Lemma~\ref{maclem},
$\gamma_1$ is
a sum of positive roots $\preceq \gamma \prec \beta$ and 
$\beta_1$ is a sum of positive roots $\preceq \beta$.
Since $\gamma_1+\beta_1 = \beta$ 
we deduce from Lemma~\ref{l1} that $\beta_1 =
\beta,
\beta_2 = 0, \gamma_1 = 0$ and $\gamma_2 = \gamma$.
Thus the only non-zero section in the Mackey filtration is the top section, which
is isomorphic to
$q^{-\beta\cdot\gamma} \Delta(\beta) \boxtimes \Delta(\gamma)$.
\end{proof}

\begin{Theorem}\label{inj3}
For $(\beta,\gamma) \in \MP(\alpha)$
there is a short exact sequence
$$
0 \longrightarrow q^{-\beta\cdot\gamma} \Delta(\beta)\circ\Delta(\gamma)
\stackrel{\phi}{\longrightarrow} \Delta(\gamma)\circ\Delta(\beta)
\longrightarrow [p_{\beta,\gamma}+1]\Delta(\alpha) \longrightarrow 0.
$$
\iffalse
Moreover 
there are unique scalars $k_i\in \k$ for $i\in\{1,2\}$, at least one of which is non-zero,
such that
$\psi \circ \dx_i = k_i x \circ \psi$,
where $x \in \END_{H_\alpha}(\Delta(\alpha))_{2d_\alpha}$ is the endomorphism 
from Lemma~\ref{dpl4}.
\fi
\end{Theorem}

\begin{proof}
In Lemma~\ref{inj1} we have already constructed the
map $\phi$. We need to show that it
is injective and compute its cokernel.
Let $V := q^{-\beta\cdot\gamma} \Delta(\beta)\cdot \Delta(\gamma)$
and $W := \Delta(\gamma)\circ \Delta(\beta)$.
The endomorphism $x$ of $\Delta(\beta)$ from Lemma~\ref{dpl4}
induces injective endomorphisms
\begin{align*}
y&:= x \circ 1 \in 
\END_{H_\alpha}(V)_{2 d_\beta},&
y &:= 1 \circ x\in \END_{H_\alpha}(W)_{2d_\beta}.
\end{align*}
Similarly the 
endomorphism $x$ of $\Delta(\gamma)$
gives us
\begin{align*}
z &:= 1 \circ x
\in \END_{H_\alpha}(V)_{2 d_\gamma},&
z &:= x \circ 1
\in \END_{H_\alpha}(W)_{2 d_\gamma}.
\end{align*}
We then have that $y \circ z = z \circ y$,
$\phi \circ y = y\circ \phi$ and $\phi \circ z = z\circ\phi$.
Thus we have defined algebra embeddings $\k[y,z]
\hookrightarrow \END_{H_\alpha}(V)$
and $\k[y,z]
\hookrightarrow \END_{H_\alpha}(W)$.

Let $\k[y,z] = I_0 \supset I_1 \supset \cdots$
be a chain of ideals of $\k[y, z]$ 
such that for all $m \geq 0$ 
there 
exist $b, c \geq 0$ with $I_m = \k y^b z^c \oplus
I_{m+1}$,
and set 
$d_m := b d_\beta + c d_\gamma$ for $b,c$ associated to $m$ in this way.
This chain of ideals induces filtrations $V = V_0 \supset V_1
\supset\cdots$ and $W = W_0 \supset W_1 \supset \cdots$
with 
$V_m := I_m (V)$
and $W_m := I_m(W)$.
The sections of these 
filtrations are $V_m / V_{m+1} \cong q^{2d_m-\beta\cdot\gamma}
L(\beta)\circ L(\gamma)$
and $W_m / W_{m+1}\cong q^{2d_m}L(\gamma) \circ L(\beta)$. 
Recalling (\ref{sesonea})--(\ref{sestwoa}), the map $\phi$ induces a 
map $V_m / V_{m+1} \rightarrow W_m / W_{m+1}$
which corresponds to the unique (up to a scalar) 
map $q^{2d_m-\beta\cdot\gamma}L(\beta)\circ
L(\gamma)
\rightarrow q^{2d_m}L(\gamma) \circ L(\beta)$ sending the head
of the first onto the socle of the second.
Hence $\phi$ induces a bijection between occurrences of $L(\lambda)$ 
as sections of a filtration of $V$ and of $W$.
This shows that the cokernel of $\phi$ can only 
have irreducible subquotients 
of the form $L(\alpha)$ (up to degree shift).

Now observe using Corollary~\ref{chprop} and Theorem~\ref{rm}(2) that 
\begin{align*}
\DIM \HOM_{H_\alpha}(\operatorname{coker}\phi, L(\alpha)) &\leq 
\DIM \HOM_{H_\alpha}(\Delta(\gamma)\circ \Delta(\beta), L(\alpha))\\
&=
\DIM \HOM_{H_\gamma \otimes H_\beta}(\Delta(\gamma) \boxtimes
\Delta(\beta),
\Res^\alpha_{\gamma,\beta} L(\alpha))\\
&\leq [p_{\beta,\gamma}+1].
\end{align*}
Applying Corollary~\ref{class}, we deduce that
$\DIM \operatorname{coker} \phi \leq [p_{\beta,\gamma}+1]
\DIM \Delta(\alpha)$.
Then by (\ref{ball1}) we have that
\begin{multline*}
[p_{\beta,\gamma}+1]\DIM \Delta(\alpha)
= 
\DIM
\Delta(\gamma)\circ\Delta(\beta)
-
q^{-\beta\cdot\gamma}
\DIM \Delta(\beta)\circ\Delta(\gamma)\\
= \DIM \operatorname{coker}\phi
-\DIM \operatorname{ker} \phi \leq \DIM \operatorname{coker}\phi
\leq [p_{\beta,\gamma}+1]\DIM \Delta(\alpha).
\end{multline*}
Equality must hold in both places, showing that
$\operatorname{ker}\phi = 0$ hence $\phi$ is injective, and
$\DIM \operatorname{coker} \phi = [p_{\beta,\gamma}+1] \DIM
\Delta(\alpha)$.
Finally, now that we have worked out both the head of $\operatorname{coker} \phi$
and its graded dimension, another application of
Corollary~\ref{class} 
shows indeed that $\operatorname{coker}\phi \cong [p_{\beta,\gamma}+1] \Delta(\alpha)$.
\end{proof}

\begin{Corollary}\label{pd}
For $\alpha \in Q^+$ of height $n$ 
and $\lambda = (\lambda_1,\dots,\lambda_l) \in \KP(\alpha)$, the projective dimension of $\Delta(\lambda)$
satisfies $\operatorname{pd} \Delta(\lambda) \leq n-l$.
\end{Corollary}

\begin{proof}
We need to show that $\ext^d_{H_\alpha}(\Delta(\lambda), V) = 0$ for
any $H_\alpha$-module $V$ and $d > n-l$.
Using Theorem~\ref{shead} and generalized Frobenis reciprocity, this
reduces 
to checking in the case that $\alpha$ is a positive root that
$\ext^d_{H_\alpha}(\Delta(\alpha), V) = 0$
for all $d > n-1$.
To see this, apply $\hom_{H_\alpha}(-, V)$ to the short exact
sequence from Theorem~\ref{inj3}
and use generalized Frobenius reciprocity and induction.
\end{proof}

\subsection{Projective resolutions}\label{ssres}
Implicit in the proof of Corollary~\ref{pd} are some
interesting projective resolutions.
To explain this, we again fix a choice of minimal pairs $\mp(\alpha)
\in \MP(\alpha)$ for each $\alpha \in R^+$ of height at least two, and
define $\kappa_\alpha$ and $\kappa_\lambda$ as in (\ref{issue}).
Let
\begin{equation}
\tilde{\Delta}(\alpha) := \kappa_\alpha \Delta(\alpha),
\qquad
\tilde{\Delta}(\lambda) := \kappa_\lambda \Delta(\lambda).
\end{equation}
In the next paragraph, we construct a 
projective resolution $P_*(\alpha)$ of $\tilde{\Delta}(\alpha)$
for each $\alpha \in R^+$, i.e.
a complex
$$\cdots \rightarrow P_2(\alpha) \stackrel{\partial_2}{\rightarrow}
P_1(\alpha)
\stackrel{\partial_1}{\rightarrow} 
P_0(\alpha)\stackrel{\partial_0}{\rightarrow} 0
$$ 
of projective modules with
$\H_0(P_*(\alpha)) \cong \tilde{\Delta}(\alpha)$
and $\H_d(P_*(\alpha)) = 0$
for $d \neq 0$.
More generally, given $\alpha \in Q^+$ of height $n$ and
$\lambda = (\lambda_1,\dots,\lambda_l) \in \KP(\alpha)$,
the total complex of the tensor product of the complexes
$P_*(\lambda_1),\dots,P_*(\lambda_l)$ gives a projective resolution
$P_*(\lambda)$
of $\tilde\Delta(\lambda)$.
It will be clear from its definition that $P_d(\lambda) = 0$ for $d >
n-l$, consistent with Corollary~\ref{pd}.

The construction of $P_*(\alpha)$ is recursive. To start with
for $i \in I$ we have that $\tilde{\Delta}(\alpha_i) = H_{\alpha_i}$,
which is projective already. 
So we just have to set $P_0(\alpha_i) := H_{\alpha_i}$
and $P_d(\alpha_i) := 0$ for $d \neq 0$ to obtain the required
resolution.
Now suppose that $\alpha \in R^+$ is of height at least two
and let $(\beta,\gamma) := \mp(\alpha)$.
We may assume by induction that the projective resolutions $P_*(\beta)$
and $P_*(\gamma)$ are already defined.
Taking the total complex of their tensor product using \cite[Acyclic
Assembly Lemma 2.7.3]{Wei}, 
we obtain a projective resolution
$P_*(\beta,\gamma)$ of $\tilde{\Delta}(\beta) \circ
\tilde{\Delta}(\gamma)$
with
\begin{align*}
P_d(\beta,\gamma) &:= \bigoplus_{d_1+d_2= d} P_{d_1}(\beta) \circ
P_{d_2}(\gamma),\\
\qquad \partial_d &:= \left(\id \circ \partial_{d_2}
-(-1)^{d_2}\partial_{d_1} \circ \id
\right)_{d_1+d_2=d}:P_d(\beta,\gamma) \rightarrow P_{d-1}(\beta,\gamma).
\end{align*}
Similarly 
we 
obtain a projective resolution $P_*(\gamma,\beta)$ of
$\tilde{\Delta}(\gamma)\circ\tilde{\Delta}(\beta)$ with
\begin{align*}
P_d(\gamma,\beta) &:= \bigoplus_{d_1+d_2= d} P_{d_1}(\gamma) \circ
P_{d_2}(\beta),\\
\qquad \partial_d &:= \left(\partial_{d_1} \circ \id + (-1)^{d_1}
\id \circ \partial_{d_2}\right)_{d_1+d_2=d}:P_d(\gamma,\beta) \rightarrow P_{d-1}(\gamma,\beta).
\end{align*}
(We have chosen signs carefully here so that 
Theorem~\ref{expres} works out nicely.)
There is an injective homomorphism
$$
\tilde{\phi}:
q^{-\beta\cdot\gamma}
\tilde{\Delta}(\beta)\circ\tilde{\Delta}(\gamma)
\hookrightarrow \tilde{\Delta}(\gamma) \circ
\tilde{\Delta}(\beta)
$$
defined in exactly the same way as the map $\phi$ in Lemma~\ref{inj1},
indeed, it
is just a direct sum of copies of the map $\phi$ from there.
Applying \cite[Comparision Theorem 2.2.6]{Wei}, 
$\tilde{\phi}$ lifts to a 
chain map $\tilde{\phi}_*:q^{-\beta\cdot\gamma} P_*(\beta,\gamma)
\rightarrow P_*(\gamma,\beta)$.
Then we take the mapping cone of $\tilde{\phi}_*$ to obtain a complex
$P_*(\alpha)$ with
\begin{align*}
\qquad P_d(\alpha) &:= P_d(\gamma,\beta) \oplus q^{-\beta\cdot\gamma} P_{d-1}(\beta,\gamma),\\
\partial_d &:= (\partial_d, \partial_{d-1}+(-1)^{d-1} \tilde{\phi}_{d-1}):
P_d(\alpha) \rightarrow P_{d-1}(\alpha).
\end{align*}
In view of Theorem~\ref{inj3} and \cite[Acyclic Assembly Lemma 2.7.3]{Wei}
once again, 
$P_*(\alpha)$ is a projective resolution of $\tilde{\Delta}(\alpha)$.

Let us describe
$P_*(\alpha)$ more explicitly.
First for $i \in I$ and the empty tuple $\sigma$,
set $\bi_{\alpha_i,\sigma} := i$.
Now suppose that $\alpha$ is of height $n \geq 2$ and
that $(\beta,\gamma) = \mp(\alpha)$ with $\gamma$ of height $m$.
For $\sigma  = (\sigma_1,\dots,\sigma_{n-1})
\in \{0,1\}^{n-1}$,
let $|\sigma| := \sigma_1+\cdots+\sigma_{n-1}$,
$\sigma_{< m} := (\sigma_1,\dots,\sigma_{m-1})$
and $\sigma_{> m} := (\sigma_{m+1},\dots,\sigma_{n-1})$.
Define $\bed_{\alpha,\sigma} \in \W_\alpha$ 
and $d_{\alpha,\sigma} \in \N$
recursively from
\begin{align*}
\bed_{\alpha,\sigma} &:= \left\{
\begin{array}{ll}
\bed_{\gamma,\sigma_{< m}}
\bed_{\beta,\sigma_{> m}}\hspace{16.5mm}&\text{if $\sigma_m = 0$,}\\
\bed_{\beta,\sigma_{> m}}
\bed_{\gamma,\sigma_{< m}}
&\text{if $\sigma_m = 1$;}
\end{array}\right.\\
d_{\alpha,\sigma} &:= \left\{
\begin{array}{ll}
d_{\beta,\sigma_{> m}}
+d_{\gamma,\sigma_{< m}}
&\text{if $\sigma_m = 0$,}\\
d_{\beta,\sigma_{> m}}+
d_{\gamma,\sigma_{< m}}-\beta\cdot\gamma
&\text{if $\sigma_m = 1$.}
\end{array}\right.
\end{align*}
Note in particular that $d_{\alpha,\sigma} = |\sigma|$ in simply-laced types.
Also if $\sigma = (0,\dots,0)$ then $\bed_{\alpha,\sigma}$ is the tuple $\bed_\alpha$ from
(\ref{issue}) and $d_{\alpha,\sigma}=0$.
Then 
we have that
\begin{equation}\label{exp}
P_d(\alpha) =\bigoplus_{\substack{\sigma \in \{0,1\}^{n-1}\\ 
|\sigma| = d}}
q^{d_{\alpha,\sigma}} H_\alpha 1_{\bed_{\alpha,\sigma}}.
\end{equation}
For the differentials 
$\partial_d:P_d(\alpha)\rightarrow P_{d-1}(\alpha)$,
 there are elements $\tau_{\sigma,\rho} \in
 1_{\bed_{\alpha,\sigma}} 
H_\alpha 1_{\bed_{\alpha,\rho}}$
for each $\sigma,\rho \in \{0,1\}^{n-1}$
with $|\sigma|=d, |\rho| = d-1$
such that, on
viewing elements of (\ref{exp}) as row vectors,
the differential $\partial_d$ is 
defined by right multiplication by the matrix
$\left(\tau_{\sigma,\rho}
\right)_{|\sigma|=d,|\rho|=d-1}$.
Moreover 
$\tau_{\sigma,\rho} = 0$ unless the tuples $\sigma$ and $\rho$ 
differ in just one entry.
Unfortunately we have not been able to find a satisfactory description
of such elements $\tau_{\sigma,\rho}$, except in the following special case.

\begin{Theorem}\label{expres}
Suppose that $\alpha \in R^+$ is {multiplicity-free}, so that
$\kappa_\alpha=1$
and $P_*(\alpha)$
is a projective resolution of the root module $\Delta(\alpha)$ itself.
Then the elements $\tau_{\sigma,\rho}$
inducing the differential
$\partial_d:P_d(\alpha) \rightarrow P_{d-1}(\alpha)$
as above may be chosen
so that
$$
\tau_{\sigma,\rho} := (-1)^{\sigma_1+\cdots+\sigma_{r-1}} \tau_w
$$
if $\sigma$ and $\rho$ differ just in the $r$th entry,
where $w\in S_n$ is the {\em unique} permutation with
$1_{\bed_{\alpha,\sigma}} \tau_w = \tau_w 1_{\bed_{\alpha,\rho}}$.
\end{Theorem}

\begin{proof}
This goes by induction on height.
The key point for the induction step is that the chain map
$\tilde{\phi}_*:q^{-\beta\cdot\gamma} P_*(\beta,\gamma) \rightarrow P_*(\gamma,\beta)$ in the
above construction
can be chosen so that $\tilde{\phi}_d:q^{-\beta\cdot\gamma} P_{d_1}(\beta)\circ
P_{d_2}(\gamma)
\rightarrow P_{d_2}(\gamma) \circ P_{d_1}(\beta)$ is defined by right multiplication by
$(-1)^{d_1}\tau_w$, where $w$ is the permutation from Lemma~\ref{inj1}.
The proof that this is indeed a chain map relies on the fact that the
braid relations hold exactly in $H_\alpha$ under the assumption that
$\alpha$ is multiplicity-free.
\end{proof}

\appendix
\section*{Appendix: Simply-laced types}\label{slt}
\setcounter{section}{1}
\setcounter{Proposition}{0}
In this appendix we perform 
the calculations needed to
fill in the gap in the proof of
Theorem~\ref{source} in the simply-laced types.
We assume that $\mathfrak{g}$ is of
type A$_r$, D$_r$ or E$_r$
and index the simple roots by $I = \{1,\dots,r\}$ as follows:
% This is the regular fonts version
$$
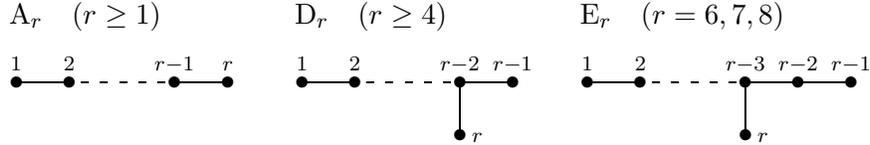
\begin{picture}(108,60)
\put(0,15){$\stackrel{1}{\bullet}$}
\put(2,17.5){\line(1,0){20}}
\put(20,15){$\stackrel{2}{\bullet}$}
\dashline{3}(20,17.5)(59,17.5)
\put(55,15){$\stackrel{r-1}{\bullet}$}
\put(62,17.5){\line(1,0){20}}
\put(80,15){$\stackrel{r}{\bullet}$}
\put(0,40){$\operatorname{A}_r\quad(r\geq 1)$}
\end{picture}
\begin{picture}(108,60)
\put(0,15){$\stackrel{1}{\bullet}$}
\put(2,17.5){\line(1,0){20}}
\put(20,15){$\stackrel{2}{\bullet}$}
\dashline{3}(20,17.5)(59,17.5)
\put(55,15){$\stackrel{r-2}{\bullet}$}
\put(62,17.5){\line(1,0){20}}
\put(62.5,17.5){\line(0,-1){20}}
\put(75,15){$\stackrel{r-1}{\bullet}$}
\put(59.8,-5){$\bullet\,{\scriptstyle r}$}
\put(0,40){$\operatorname{D}_r\quad(r\geq 4)$}
\end{picture}
\begin{picture}(128,60)
\put(0,15){$\stackrel{1}{\bullet}$}
\put(2,17.5){\line(1,0){20}}
\put(20,15){$\stackrel{2}{\bullet}$}
\dashline{3}(20,17.5)(59,17.5)
\put(55,15){$\stackrel{r-3}{\bullet}$}
\put(62,17.5){\line(1,0){20}}
\put(62.5,17.5){\line(0,-1){20}}
\put(75,15){$\stackrel{r-2}{\bullet}$}
\put(82,17.5){\line(1,0){20}}
\put(95,15){$\stackrel{r-1}{\bullet}$}
\put(59.8,-5){$\bullet\,{\scriptstyle r}$}
\put(0,40){$\operatorname{E}_r\quad(r=6,7,8)$}
\end{picture}
$$
The natural ordering on $I$ induces a lexicographic
total order $<$ on the set $\W$ of words.
In \cite[$\S$4.3]{Lec}, Leclerc shows that there is
a well-defined injective map 
$l:R^+ \hookrightarrow \W$ defined recursively
by setting $l(\alpha_i) := i$ for each $i \in I$, then
\begin{equation*}
l(\alpha) := \max\left(l(\gamma) l(\beta)\:|\:\beta,\gamma \in R^+,
\beta+\gamma=\alpha,
l(\beta) > l(\gamma)\right)
\end{equation*}
for non-simple $\alpha \in R^+$.
The words $\{l(\alpha)\:|\:\alpha \in R^+\}$ are 
the {\em good Lyndon words} associated to the order $<$.
We define the corresponding {\em Lyndon ordering} $\prec$ on $R^+$ 
by declaring that $\alpha \prec \beta$
if and only if $l(\alpha) < l(\beta)$.
This is known to be a convex
ordering thanks to a general result from \cite{Rosso}.
In the following examples we make it explicit by
listing the words $\{l(\alpha)\:|\:\alpha \in R^+\}$ in each type. 

\begin{Example}\rm\label{eg1}
In type A$_r$ for the above numbering of the simple roots,
the positive roots are $\alpha_{i,j} := \alpha_i+\cdots+\alpha_j$
for $1 \leq i \leq j \leq r$.
The good Lyndon word $l(\alpha_{i,j})$ is the
{\em increasing segment} $i \cdots j$.
Thus the convex ordering $\prec$ satisfies $\alpha_{i,j} \prec
\alpha_{k,l}$
if and only if $i < k$, or $i=k$ and $j < l$. 
The corresponding reduced expression for $w_0$
is $(s_1 \cdots s_r) (s_1 \cdots s_{r-1}) \cdots (s_1 s_2) s_1$.
\end{Example}

\begin{Example}\rm
In type D$_r$
the good Lyndon words corresponding to the positive roots are
the words
$i \cdots j$ for $1 \leq i \leq j \leq r-1$
and the words 
$i \cdots (r-2) \underline{r \cdots j}$ 
for $1 \leq i < j \leq r$ (where the
underline denotes a {\em decreasing segment}).
\end{Example}

\begin{Example}\rm\label{eg3}
In E$_6$
the good Lyndon words arranged in lexicographic order are:

\small
\noindent\:
1, 12, 123, 
1234, 
12345, 
1236, 
12364, 
123643, 
1236432, 
123645, 
1236453,
12364532,

\noindent\:
12364534, 
123645342, 
1236453423, 
12364534236, 
2, 
23, 
234, 
2345, 
236, 
2364, 
23643,

\noindent\:
23645, 
236453, 
2364534, 
3, 
34, 
345, 
36, 
364, 
3645, 
4, 
45, 
5, 
6.
\normalsize
\end{Example}

\begin{Example}\rm\label{eg4}
In E$_7$
the good Lyndon words are:

\small
\noindent\:
1,
12,
123,
1234,
12345,
123456,
12347,
123475,
1234754,
12347543,
123475432,
1234756,

\noindent\:
12347564,
123475643,
1234756432,
123475645,
1234756453,
12347564532,
12347564534,

\noindent\:
123475645342,
1234756453423,
123475645347,
1234756453472,
12347564534723,

\noindent\:
123475645347234,
1234756453472345,
12347564534723456,
2,
23,
234,
2345,
23456,

\noindent\:
2347,
23475,
234754,
2347543,
234756,
2347564,
23475643,
23475645,
234756453,

\noindent\:
2347564534,
23475645347,
3,
34,
345,
3456,
347,
3475,
34754,
34756,
347564,
3475645,

\noindent\:
4,
45,
456,
47,
475,
4756,
5,
56,
6,
7.
\normalsize
\end{Example}

\begin{Example}\rm\label{eg5}
In E$_8$
the good Lyndon words are:

\small
\noindent\:
1,
12,
123,
1234,
12345,
123456,
1234567,
123458,
1234586,
12345865,
123458654,

\noindent\:
1234586543,
12345865432,
12345867,
123458675,
1234586754,
12345867543,

\noindent\:
123458675432,
1234586756,
12345867564,
123458675643,
1234586756432,

\noindent\:
123458675645,
1234586756453,
12345867564532,
12345867564534,
123458675645342,

\noindent\:
1234586756453423,
12345867564534231234586756458,
1234586756458,

\noindent\:
12345867564583,
123458675645832,
123458675645834,
1234586756458342,

\noindent\:
12345867564583423,
1234586756458345,
12345867564583452,
123458675645834523,

\noindent\:
1234586756458345234,
12345867564583456,
123458675645834562,

\noindent\:
1234586756458345623,
12345867564583456234,
123458675645834562345,

\noindent\:
1234586756458345623458,
123458675645834567,
1234586756458345672,

\noindent\:
12345867564583456723,
123458675645834567234,
1234586756458345672345,

\noindent\:
12345867564583456723456,
12345867564583456723458,
123458675645834567234586,

\noindent\:
1234586756458345672345865,
12345867564583456723458654,

\noindent\:
123458675645834567234586543,
1234586756458345672345865432,
2,
23,
234,
2345,

\noindent\:
23456,
234567,
23458,
234586,
2345865,
23458654,
234586543,
2345867,
23458675,

\noindent\:
234586754,
2345867543,
234586756,
2345867564,
23458675643,
23458675645,

\noindent\:
234586756453,
2345867564534,
234586756458,
2345867564583,
23458675645834,

\noindent\:
234586756458345,
2345867564583456,
23458675645834567,
3,
34,
345,
3456,
34567,

\noindent\:
3458,
34586,
345865,
3458654,
345867,
3458675,
34586754,
34586756,
345867564,

\noindent\:
3458675645,
34586756458,
4,
45,
456,
4567,
458,
4586,
45865,
45867,
458675,
4586756,

\noindent\:
5,
56,
567,
58,
586,
5867,
6,
67,
7,
8.
\normalsize
\end{Example}

Working always now with the Lyndon ordering $\prec$ just defined,
we also fix the following
choice of a minimal pair $\mp(\alpha) \in
\MP(\alpha)$ for each $\alpha \in R^+$ of height at least two:
let it be the two-part Kostant partition $(\beta,\gamma)$ of $\alpha$ 
for which $\gamma$ is maximal in the ordering $\prec$.

\begin{Lemma}\label{a1}
Suppose that $\alpha \in R^+$ is of height at least two
and let $(\beta,\gamma) := \mp(\alpha)$.
Apart from the highest root in $\operatorname{E}_8$,
the word $l(\gamma)$ is obtained from $l(\alpha)$ by removing its last
letter, 
and $l(\beta)$ is the singleton consisting just of this last letter.
For the highest root in $\operatorname{E}_8$,
we have that
$l(\gamma) = 1234586756453423$
and $l(\beta) = 1234586756458$.
In all cases, $l(\alpha)= l(\gamma)l(\beta)$ is the costandard
factorization of $l(\alpha)$ from \cite[$\S$3.2]{Lec}, and 
the word $l(\alpha)$ coincides with the word $\bi_\alpha$ from (\ref{issue}).
\end{Lemma}

\begin{proof}
This follows by inspection of the data in the examples.
\end{proof}

Finally we recall the construction of {\em homogeneous representations}
from \cite{KR1}.
Let $\sim$ be the equivalence relation on $\W$
generated by
interchanging an adjacent pair of letters $i$ and $j$ which are not
connected by an edge in the Dynkin diagram.
A word $\bi \in \W$ of length $n$
is said to be {\em homogeneous} if 
it is impossible to find $\bj \sim \bi$
such that either $j_r = j_{r+1}$ for some $1 \leq r \leq n-1$
or $j_s = j_{s+2}$ for some $1 \leq s \leq n-2$.
If $\alpha \in R^+$ is such that $l(\alpha)$ is homogeneous,
then the module $L(\alpha)$ can be constructed explicitly
as the graded 
vector space with basis $\{v_\bi\:|\:\bi \sim l(\alpha)\}$ concentrated in degree zero, such that
each $v_\bi$ is in the $\bi$-word space,
all $x_j$ act as zero, 
and 
$\tau_k v_\bi := v_{(k\:k\!+\!1)(\bi)}$
if $i_k$ and $i_{k+1}$ are not connected by an edge in the Dynkin
diagram,
$\tau_k v_\bi := 0$ otherwise.

\begin{Lemma}\label{a2}
For all $\alpha \in R^+$ except for the highest root in type
$\operatorname{E}_8$, the word $l(\alpha)$ listed in
Examples~\ref{eg1}--\ref{eg5} is homogeneous, hence the module
$L(\alpha)$ is a homogeneous representation.
\end{Lemma}

\begin{proof}
Again this follows by checking each case in turn.
\end{proof}

The following lemma is a special case of Theorem~\ref{as}, but of course we
cannot use that here as the proof of Theorem~\ref{as} depends on Theorem~\ref{source}.

\begin{Lemma}\label{a3}
For $\alpha \in R^+$ of height at least two and $\lambda = 
(\beta,\gamma) := \mp(\alpha)$,
there are non-split short exact sequences
\begin{align*}
0 \longrightarrow q L(\alpha) &\longrightarrow L(\beta)\circ L(\gamma)
\longrightarrow L(\lambda) \longrightarrow 0,\\
0 \longrightarrow q L(\lambda)
&\longrightarrow L(\gamma)\circ L(\beta)
\longrightarrow L(\alpha)
\longrightarrow 0.
\end{align*}
\end{Lemma}

\begin{proof}
We must show that the module $\X$ in (\ref{sestwoa})
is isomorphic to $L(\alpha)$. This follows if we can show that
$\DIM 1_{l(\alpha)} (L(\gamma) \circ L(\beta)) = 1$.
We know by Lemmas~\ref{a1}--\ref{a2} that $l(\alpha) = l(\gamma) l(\beta)$ and
$$
\CH L(\gamma) = \sum_{\bk \sim l(\gamma)} \bk,
\qquad
\CH L(\beta) = \sum_{\bj \sim l(\beta)} \bj.
$$
Hence we are reduced to showing that the only pair $(\bk,\bj)$
with
$\bk \sim l(\gamma)$ and $\bj \sim l(\beta)$ 
such that $l(\alpha)$ has non-zero coefficient in the shuffle product
$\bk\circ \bj$ is $(\bk,\bj) = (l(\gamma), l(\beta))$,
and moreover for this pair the only shuffle of $\bk$ and $\bj$ that
produces $l(\alpha)$ is the identity.
Apart from the highest root of E$_8$, this follows 
because in all cases no $\bk \sim 
l(\gamma)$ ends in the letter $l(\beta)$.
The highest root of E$_8$ takes only a little more combinatorial
analysis
using the explicit descriptions of $l(\gamma)$ and $l(\beta)$ from Lemma~\ref{a1}.
\end{proof}

\begin{Theorem}\label{con}
Assume that $\mathfrak{g}$ is simply-laced and that
$\prec$ is the Lyndon ordering fixed above.
Let $\alpha \in R^+$ have height at least two and set $(\beta,\gamma)
:= \mp(\alpha)$.
Then there exists an $H_\alpha$-module $X$ such that
$X / \soc X \cong L(\gamma)\circ L(\beta)$ and
$\soc X \cong q^2 L(\alpha)$.
\end{Theorem}

\begin{proof}
By Lemma~\ref{a2},
the modules $L(\beta)$ and $L(\gamma)$ can be constructed
explicitly 
as above as they are homogeneous representations.
In a similar way we construct
a module $\Delta_2(\beta)$ with
$\Delta_2(\beta) / \soc \Delta_2(\beta) \cong L(\beta)$ and
$\soc \Delta_2(\beta) \cong q^2 L(\beta)$, by declaring
that it has homogeneous basis $\{v^{\pm}_\bi\:|\:\bi \sim l(\beta)\}$
with $v_\bi^{\pm} \in 1_\bi \Delta_2(\beta)_{1 \pm 1}$, such that
$x_j v^-_\bi := v_{\bi}^+$, $x_j v_{\bi}^+ := 0$, and
$\tau_k 
v^{\pm}_\bi := v^{\pm}_{(k\:k\!+\!1)(\bi)}$
if $i_k$ and $i_{k+1}$ are not connected by an edge in the Dynkin diagram,
$\tau_k v^{\pm}_\bi := 0$ otherwise.

Now consider the second short exact sequence from Lemma~\ref{a3}.
Combined with exactness of induction, it follows that $L(\gamma)
\circ \Delta_2(\beta)$
has a unique submodule $S \cong q^3 L(\lambda)$.
Set $X := L(\gamma) \circ \Delta_2(\beta) / S$
and $v^\pm := 1_{\gamma,\beta} \otimes (v_{l(\gamma)} \otimes
v_{l(\beta)}^\pm) + S$.
Then $X$ has a unique submodule $Y \cong q^2 L(\alpha)$ generated by
$v^+$, and
$X / Y \cong L(\gamma) \circ L(\beta)$.
It remains to show that $Y = \soc X$.
Suppose for a contradiction that the socle is larger. Then $X$ must
also have a 
submodule $Z \cong q L(\lambda)$.
In the next two paragraphs, we prove that there exists a word $\bi \in \W_\alpha$
and  elements $a, b \in
H_\alpha$ such that $1_\bi L(\alpha) = 0$,
$b v^- \in 1_\bi X$,
and $ab v^- = v^+$.
This is enough to complete the proof, 
for then we must have that $b v^- \in Z$, hence $v^+ = ab v^- \in
Z$ too, contradicting $Y \cap Z = 0$.

To construct $a, b$ and $\bi$, we first assume that $\alpha$ is {\em
  not}
the highest root in type E$_8$.
Suppose that $\alpha$ is of height $n$.
Let $1 \leq p < n$ be maximal such that the $p$th letter $i$
of $l(\alpha)$ is connected to its $n$th letter $j$ 
in the Dynkin diagram.
By inspection of the information in Examples~\ref{eg1}--\ref{eg5},
this is always possible and moreover
none of the letters in between $i$ and $j$ are equal to $j$.
Let $w$ be the cycle
$(p\:p+1\:\cdots\:n)$
and set $a := \tau_{w^{-1}}, b := \tau_{w}$.
Finally let $\bi$ be the word obtained from $l(\alpha)$ by deleting
the $n$th letter $j$ then reinserting it just before the $p$th
letter $i$; then we have that $b v^- \in 1_\bi X$.
An easy application of the relations shows that 
$ab v^- = v^+$ (up to a sign).
We are left with showing that $1_\bi L(\alpha) = 0$.
But in all these cases $l(\alpha)$ is also homogeneous so this follows
as $\bi \not\sim l(\alpha)$ by construction.

It remains to treat the highest root for E$_8$.
Here Lemma~\ref{a2} tells us that $l(\beta) = 1234586756458$ and
$l(\gamma) = 
1234586756453423$, but the word $l(\alpha) = l(\gamma)l(\beta)$ 
is no longer homogeneous.
We set
$\bi := 
123458675645342 12345867564 358$,
$a := \tau_{w}$ and $b := \tau_{w^{-1}}$
where $w$ is the cycle $(16\:\:\: 17\:\:\: \cdots\:\:\: 27)$; again we have that $b v^- \in 1_\bi V$.
Another explicit relation check (best made now by drawing a picture)
shows
that 
$ab v^- = v^+$ (up to a sign).
It remains to show that $1_\bi L(\alpha) = 0$.
From Lemma~\ref{a3}, we deduce that
$(1-q^2)[L(\alpha)] = 
[L(\gamma)\circ L(\beta)] - q [L(\beta)\circ
L(\gamma)]$, hence
\begin{equation*}
\CH L(\alpha) = \sum_{\bk \sim l(\gamma), \bj \sim l(\beta)} (\bk \circ \bj -
q \bj \circ \bk) / (1-q^2).
\end{equation*}
Now one more
calculation shows that the $\bi$-coefficient on the right hand side is
indeed zero.
\end{proof}

\begin{Corollary}\label{speccase}
Theorem~\ref{source} holds 
for simply-laced $\mathfrak{g}$
and the Lyndon ordering $\prec$ 
defined above.
\end{Corollary}

\begin{proof}
This follows 
by mimicking the argument explained
in the first half of the proof of \cite[Proposition 4.5]{Mac}.
The essential ingredients needed for this are provided by
Lemma~\ref{a3} and Theorem~\ref{con}.
\end{proof}

\begin{Corollary}\label{fgd}
For simply-laced $\mathfrak g$ and
$\alpha \in Q^+$ of height $n$,
the KLR algebra $H_\alpha$ has finite global dimension $n$,
i.e. $\sup \operatorname{pd} V = n$ where the supremum is taken over all
$H_\alpha$-modules $V$ and $\operatorname{pd}$ denotes projective
dimension in the category of graded modules.
\end{Corollary}

\begin{proof}
We choose the convex ordering 
to be the Lyndon ordering as in Corollary~\ref{speccase},
so that Theorem~\ref{source} is proved. This is all that is needed
for all the subsequent results
from sections
\ref{secsm} and \ref{secmp} to be proved for this ordering.
Then we argue as in the proof of \cite[Theorem 4.6]{Mac} to reduce to showing
that 
$\ext^d_{H_\alpha}(L(\lambda), V) = 0$ for any
$H_\alpha$-module $V$
and $d > n$. Since $L(\lambda)$ is the socle of $\bar\nabla(\lambda)$
and all its other composition factors are of the form $L(\mu)$ (up to
degree shift) for
$\mu \prec \lambda$, this follows by induction on the ordering if we can show
that $\ext^d_{H_\alpha}(\bar\nabla(\lambda), V) = 0$ and
$d > n$.
To prove this, note 
by Lemma~\ref{lvl} and up to a degree shift that 
$\bar\nabla(\lambda)$ is induced from 
$L(\lambda_l) \boxtimes\cdots\boxtimes L(\lambda_1)$, 
so we can use generalized Frobenius reciprocity to reduce  further to 
showing for a positive
root $\alpha$ of height $n$ that
$\ext^d_{H_\alpha}(L(\alpha), V) = 0$
for any $V$ and $d > n$. Finally this follows by applying 
$\hom_{H_\alpha}(-, V)$ to (\ref{ses3}) and using Corollary~\ref{pd}.
\end{proof}

\begin{Remark}\rm
As noted already in the introduction, 
the global dimension of $H_\alpha$ is equal to $\height(\alpha)$
in non-simply-laced types too;
see \cite[Theorem 4.6]{Mac}.
\end{Remark}


\begin{thebibliography}{CPS3}

\bibitem[BeK]{BK}
A. Berenstein and D. Kazhdan,
Geometric and unipotent crystals II: from unipotent bicrystals to
crystal bases, 
{\em Contemp. Math.} {\bf 433} (2007), 13--88.

\bibitem[Bo]{Bou}
N. Bourbaki, 
``Lie groups and Lie algebras,'' 
Chapters 7--9 of {\em Elements of
  Mathematics},
Springer, 2005.

\bibitem[B]{Brun}
J. Brundan,
Some dual canonical basis calculations in F$_4$, output from a
computer calculation available at
{\tt http://uoregon.edu/$\sim$brundan/papers/f4.txt}.

\bibitem[BK]{BrK}
J. Brundan and A. Kleshchev, Homological properties of ADE
Khovanov-Lauda-Rouquier algebras; {\tt  arXiv:1210.6900v1}.

\bibitem[D]{Dlab}
V. Dlab, 
Properly stratified algebras,
{\em C. R. Acad. Sci. Paris} {\bf 331} (2000), 191--196. 

\bibitem[Do]{Donkin}
S. Donkin.
{\em The $q$-Schur Algebra},
CUP, 1998.

\bibitem[G]{Green}
J. A. Green, Quantum groups, Hall algebras and quantum shuffles, {\em Progress
  Math.} {\bf 141} (1998), 273--290.

\bibitem[KOP]{KOH}
S.-J. Kang, S.-J. Oh and E. Park, 
Categorification of quantum generalized Kac-Moody algebras and crystal
bases;
{\tt arxiv:1102.5165}.

\bibitem[K]{Ka}
M. Kashiwara,
On crystal bases of the $q$-analogue of universal enveloping algebras,
{\em Duke Math. J.} {\bf 63} (1991), 465--516.

\bibitem[KS]{KS}
M. Kashiwara and Y. Saito,
Geometric construction of crystal bases, {\em Duke Math. J.} {\bf 89} (1997),
9--36.

\bibitem[Ka]{Kato}
S. Kato,
PBW bases and KLR algebras;
{\tt arXiv:1203.5254}.

\bibitem[KL1]{KL}
M. Khovanov and A. Lauda,
A diagrammatic approach to categorification of quantum
groups I, {\em Represent. Theory} {\bf 13} (2009), 309--347. 

\bibitem[KL2]{KL2}
M. Khovanov and A. Lauda,
A diagrammatic approach to categorification of quantum
groups II, {\em Trans. Amer. Math. Soc.} {\bf 363} (2011), 2685--2700. 

\bibitem[KR1]{KR1}
A. Kleshchev and A. Ram, 
Homogeneous representations of Khovanov-Lauda algebras, {\em
  J. Eur. Math. Soc.}
{\bf 12} (2010), 1293--1306.

\bibitem[KR2]{KR2}
A. Kleshchev and A. Ram, 
Representations of Khovanov-Lauda-Rouquier algebras and combinatorics
of Lyndon words, {\em Math. Ann.} {\bf 349} (2011), 943--975. 

\iffalse
\bibitem[LR]{LR}
P. Lalonde and A. Ram, Standard Lyndon bases of Lie algebras and enveloping 
algebras, {\em Trans. Amer. Math. Soc.} {\bf 347} (1995), 1821--1830. 
\fi

\bibitem[LV]{LV}
A. Lauda and M. Vazirani,
Crystals from categorified quantum groups, {\em Advances Math.} {\bf 228}
(2011), 
803--861.

\bibitem[Le]{Lec}
B. Leclerc, Dual canonical bases, quantum shuffles and $q$-characters,
{\em Math. Z.} {\bf 246} (2004), 691--732. 

\bibitem[LS]{LS}
S.  Levendorskii and Y. Soibelman, 
Algebras of functions on compact quantum groups, Schubert cells and quantum 
tori, {\em Comm. Math. Phys.} {\bf 139} (1991), 141--170.

\bibitem[L1]{Lu}
G. Lusztig,
Canonical bases arising from quantized enveloping algebras,
{\em J. Amer. Math. Soc.} {\bf 3} (1990), 447--498.

\bibitem[L2]{Lubook}
G. Lusztig, {\em Introduction to Quantum Groups}, Birkh\"auser, 1993.

\bibitem[L3]{Lubraid}
G. Lusztig, 
Braid group action and canonical bases,
{\em Advances Math.} {\bf 122} (1996), 237--261.

\bibitem[Ma]{Mak}
R. Maksimau,
Canonical basis, KLR-algebras and parity sheaves;
{\tt arxiv:1301.6261}.

\bibitem[M]{Mac}
P. McNamara,
Finite dimensional representations of Khovanov-Lauda-Rouquier algebras
I: finite type, to appear in {\em J. Reine Angew. Math.};
{\tt arxiv:1207.5860}.

\iffalse
\bibitem[OS]{OS}
E. Opdam and M. Solleveld,
Homological algebra for affine Hecke algebras,
{\em Advances Math.} {\bf 220} (2009), 1549--1601.
\fi

\bibitem[P]{Papi}
P. Papi,
A characterization of a special ordering in a root system,
{\em Proc. Amer. Math. Soc.} 
{\bf 120} (1994), 661--665. 

\iffalse
\bibitem[Re]{Re}
C. Reutenauer, {\em Free Lie Algebras}, Oxford University Press, 1993.
\fi

\bibitem[Ro1]{Rosso0}
M. Rosso, Quantum groups and quantum shuffles, {\em Invent. Math.} {\bf 133} (1998), 399--416.

\bibitem[Ro2]{Rosso}
M. Rosso,
Lyndon bases and the multiplicative formula for $R$-matrices, preprint, 2002. 

%\bibitem[Rot]{Rot}
%J. J. Rotman,
%{\em An Introduction to Homological Algebra},
%Springer, 2009.

\bibitem[R1]{R}
R. Rouquier, 
2-Kac-Moody algebras;
{\tt arXiv:0812.5023}.

\bibitem[R2]{R2}
R. Rouquier,
Quiver Hecke algebras and $2$-Lie algebras,
{\em Algebra Colloq.} {\bf 19} (2012), 359--410.

\bibitem[S]{Saito}
Y. Saito,
PBW basis of quantized universal enveloping algebras,
{\em Publ. RIMS. Kyoto Univ.}
{\bf 30} (1994), 209--232.

\bibitem[T]{T}
S. Tsuchioka,
Answer to question ``Where does the canonical basis differ from the KLR basis?''
{\tt http://mathoverflow.net/questions/91962/}.

\bibitem[VV]{VV}
M. Varagnolo and E. Vasserot,  
Canonical bases and KLR-algebras, {\em J. Reine
Angew. Math.} {\bf 659} (2011), 67--100.

\bibitem[W]{Wei}
C. Weibel,
{\em An Introduction to Homological Algebra},
CUP, 1994.

\bibitem[Wi]{W}
G. Williamson,
On an analogue of the James conjecture;
{\tt arXiv:1212.0794}.
\end{thebibliography}
\end{document}